\documentclass[11pt,reqno]{amsart}
\usepackage{amsmath, amsfonts, amsthm, amssymb,mathrsfs,  graphicx, cite, color}
 \usepackage{latexsym,hyperref}

\usepackage{enumerate}
 \usepackage{color}
\allowdisplaybreaks
 \date{\today}

 \numberwithin{equation}{section}

\newcommand{\dv}{\mathrm{div}\,}

\newtheorem{Theorem}{Theorem}[section]
\newtheorem{Lemma}{Lemma}[section]
\newtheorem{Proposition}{Proposition}[section]

\theoremstyle{definition}

\newtheorem{Remark}{Remark}[section]

\begin{document}

\title[strong solutions to the fluid-particle  flows] 
 {Global well-posedness and optimal large-time behavior of strong solutions to the non-isentropic particle-fluid  flows}

  \author[Y.-M. Mu ]{Yanmin Mu}
\address{School of Applied Mathematics, Nanjing University of Finance \& Economics, Nanjing
 210046,  China; and
 School of Mathematical Sciences and Mathematical Institute, Nanjing Normal University, Nanjing 210023,
   China}
 \email{yminmu@126.com.}

 \author[D. Wang]{Dehua Wang}
\address{Department of Mathematics,   University of Pittsburgh,  Pittsburgh, PA 15260, USA.}
 \email{dwang@math.pitt.edu.}

\begin{abstract}
In this paper, we study the three-dimensional non-isentropic compressible fluid-particle  flows.
The system    involves coupling between the Vlasov-Fokker-Planck equation and
  the non-isentropic compressible Navier-Stokes equations  through momentum and energy exchanges.
 For  the initial data near the given equilibrium
 we  prove the global well-posedness of strong
 solutions and   obtain the optimal  algebraic rate of convergence
  in the three-dimensional whole space.
For the  periodic domain the same global well-posedness result still holds
  while the convergence rate is  exponential.
New ideas and techniques are developed  to establish the well-posedness  and large-time behavior.
For the global well-posedness our methods  are based on the new macro-micro decomposition and fine energy estimates,
while the proofs  of the optimal large-time behavior rely on   the Fourier analysis of the linearized Cauchy problem
 and the energy-spectrum method.
 \end{abstract}


\date{\today}

\maketitle


\section{Introduction}


%

In this paper we study the global well-posedness and large time behavior of strong solutions for the three-dimensional fluid-particle flows, governed by the following  Navier-Stokes equations of compressible non-isentropic fluids coupled with the Vlasov-Fokker-Planck equation of particles \cite{BBFG,CGL,M}:
\begin{align}
&\partial_t n +\nabla\cdot(n u)=0,\label{v1.2}\\
&\partial_t(n u)+\nabla\cdot(n u\otimes u)-\mu\Delta u+\nabla p=\mathcal{M} ,   \label{v1.3}\\
&\partial_t(n E)+\nabla\cdot\big((nE+p)u\big)-\kappa\Delta \tilde{\theta} = \mathcal{F}, \label{vx1.3}\\
&\partial_t F+v\cdot\nabla_{x}F=L_{u,\tilde{\theta}}F,\label{v1.1}
\end{align}
where, $n=n(t,x)\geq 0, u=u(t,x)\in \mathbb{R}^{3}, p=p(t,x)\geq 0, E=E(t,x)\geq 0, \tilde{\theta}=\tilde{\theta}(t,x)\geq 0$
for $(t,x)\in \mathbb{R}^+\times \Omega$ denote  the  density,
  velocity,   pressure,  total energy, and   temperature of the fluids,  respectively;
 $F=F(t,x,v)\geq 0$ for $(t,x,v)\in \mathbb{R}^+\times \Omega\times \mathbb{R}^{3}$
denotes the density distribution function of particles in the phase space;    the spatial domain is
 $\Omega=\mathbb{R}^{3}$ or $\mathbb{T}^{3}$ (a periodic domain in $\mathbb{R}^3$);
  and $\mu, \kappa$ are the viscosity and heat conductivity constants.
 The total energy $E$,   internal energy $e$,   pressure $p$, and   temperature $\tilde{\theta}$
satisfy the following relations:
$ 
E=e+\frac{1}{2}|u|^{2},\,\,\, p=Rn\tilde{\theta},\,\,\,e=\frac{p}{(\gamma-1)n}, 
$ 
 where $\gamma> 1$ is the adiabatic constant and $R>0$ is constant. 
 The Fokker-Planck operator is defined by
 $$L_{u,\tilde{\theta}}F=\dv_{v}\Big((v-u)F+\tilde{\theta}\nabla_{v}F\Big),$$
  which accounts for the friction force  exerting on the particles by the surrounding fluids and the
 Brownian motion of the particles, that is,
the friction force is assumed proportional to the relative velocity $v-u$,
and the Brownian motion depending on the temperature of the fluid induces diffusion with respect to the velocity variable.
The two-phase  flows have a disperse phase from a statistical viewpoint for    particles,
and a dense phase   from continuum mechanics  for fluids.
The coupling terms $\mathcal{M}, \mathcal{F}$ depict the interaction between the disperse phase and the dense phase, and  read
\begin{align}
&\mathcal{M}=-\int_{\mathbb{R}^{3}}v L_{u,\tilde{\theta}}F\,  {\rm d}v=\int_{\mathbb{R}^{3}}(v-u)F\,  {\rm d}v,\nonumber\\
&\mathcal{F}=-\int_{\mathbb{R}^{3}}\frac{|v|^{2}}{2} L_{u,\tilde{\theta}}F\,  {\rm d}v=\int_{\mathbb{R}^{3}}[v\cdot(v-u)-3\tilde{\theta}]F\,  {\rm d}v, \nonumber
\end{align}
indicating the momentum exchanges   and the energy exchanges, respectively.
%
%
%
%
 For smooth solutions of the compressible non-isentropic Navier-Stokes-Vlasov-Fokker-Planck system \eqref{v1.2}-\eqref{v1.1}, the temperature
 $\tilde{\theta}$ satisfies the following equation
 \begin{align}
 \partial_{t}\tilde{\theta}+n\cdot\nabla \tilde{\theta}+(\gamma-1)\tilde{\theta}\dv u &-\frac{\kappa(\gamma-1)}{n}\Delta \tilde{\theta}
 +\frac{\mu}{n}\Delta u\cdot u \nonumber \\
 &=\frac{\gamma-1}{n}\int_{\mathbb{R}^3}[|v-u|^{2}-3\tilde{\theta}]F\,{\rm d}v. \label{vx1.33}
\end{align}
The fluid-particle flows have a wide range of applications from   dynamics of sprays, combustion, pollution processes, waste water treatment, to biomedical flows; see \cite{CBG,BBFG,BBJM,BBT,BWC,CGL,FFS,M,RM1,RM2,FAW,FAWa} and the references therein for the discussions of applications and modeling issues.
We remark that in the momentum equation \eqref{v1.3}, we only keep the shear viscosity and  skip the bulk viscosity term for the sake of simplicity of presentations,  since the bulk viscosity will not add significant difficulty and we shall focus on the complexity caused by the heat conductivity and the Fokker-Planck operator.

 The purpose of  this paper is to establish the well-posedness of the system  \eqref{v1.2}-\eqref{v1.1}  near  a global Maxwellian.
 Without loss of generality, we  normalize the  global Maxwellian as
$$M=M(v)=\frac{1}{(2\pi)^{3/2}}\exp\Big\{-\frac{|v|^{2}}{2}\Big\}.$$
For the Cauchy problem with the initial data
\begin{equation}
(F,n,u,\tilde{\theta})|_{t=0}=(F_0(x),n_0(x),u_0(x),\tilde{\theta}_0(x)),
\end{equation}
we shall prove the global existence and uniqueness as well as the large-time behavior of the strong solution for the unknowns $(F,n,u,\tilde{\theta})$ near the global equilibrium
  state 
$(F,n,u,\tilde{\theta})\equiv( M,1,0,1) $.



 There exist  many  different systems describing the kinetic-fluid for the   physical regimes under consideration,
such as the compressible or incompressible fluids, viscous or inviscid fluids, with or without thermal diffusion acting on the
particles  and so on. The mathematical analysis of such mathematical models is very difficult due to the nonlinear coupling of partial differential equations of different types.

We now give a brief review of  works in literature on some kinetic-fluid models related to our system \eqref{v1.2}-\eqref{v1.1}.
%
%
The global existence of classical and weak solutions for the incompressible fluid-particle flows have been studied in many papers, see  \cite{KH,BDGM,Yu,GHMZ,CDM,CKL2,BDM,GJV1,GJV2,LLZ,LM,MA,WY} and their references.
For the compressible fluid, when the drag force exerted by the surrounding fluid is proportional to the relative velocity $v-u$,
the global   weak solution  and the asymptotic analysis were obtained in Mellet and Vasseur \cite{MV1,MV2},
the global classical solutions near an equilibrium  and   exponential decay were obtained in Chae, Kang, and Lee \cite{CKL},
 and the dissipative quantities, equilibria and their stability were studied in  Carrillo and Goudon \cite{CG}.
When the drag force  depends on both the relative velocity $v-u$ and the density  of the fluid,
the local existence of  classical solutions to the Euler-Vlasov system was obtained in Baranger and  Desvillettes \cite{BD},
 and global  strong solution near  an equilibrium and large-time  behavior  to the  isentropic compressible Navier-Stokes-Vlasov-Fokker-Planck system
 were established in Li, Mu and Wang \cite{LMW}.

  To the best of our knowledge, there is few rigorous mathematical results concerning the case of non-isentropic kinetic-fluid equations.
  In\cite{BBF,TSB},    some numerical analysis on the kinetic-fluid models  with energy exchange involved was presented.
 In this paper, we shall address two problems for the non-isentropic system  \eqref{v1.2}-\eqref{v1.1}: (1) the global existence of   strong solution in the framework of small
perturbation of an equilibrium,  (2)  the asymptotic behavior to the given Maxwellian equilibrium.
As far as we know, this  is   the first rigorous mathematical  work which deals with the energy exchange between the disperse phase and the dense phase.

The perturbation of solutions to the Navier-Stokes-Vlasov-Fokker-Planck system \eqref{v1.2}-\eqref{vx1.33} near the global equilibrium
  state $(F,n,u,\tilde{\theta})\equiv( M,1,0,1)$ satisfies the system \eqref{v1.5}-\eqref{v1.8}.
To prove  the global existence of strong solution to the problem \eqref{v1.5}-\eqref{v1.8}, we mainly use the fine energy estimates
 together with the local existence of strong solutions and  continuum argument.
 It is well known that under the  condition of a smallness condition on the perturbation,
 using fine energy estimates will lead to the global existence of strong solutions.  This approach is in the spirit of the  papers \cite{YG2,YG3,MN,MN1} for the Boltzmann, Landau and Navier-Stokes equations.

The fact that the unknowns of the Navier-Stokes-Vlasov-Fokker-Planck system do not depend on the same set of variables yields many technical difficulties, and the proof requires sharp estimates to the kinetic equation and the fluid equations.
We shall adopt the techniques in the works of   Guo \cite{YG,YG1,YG2,YG3} where the full coercivity of the linearized collision operator of the Boltzmann equation  are crucial to obtain the global classical solution of the nonlinear kinetic equations near an equilibrium.
In brief,  its solution to the Boltzmann equation is decomposed  into macro and micro components,
 the dissipative effect through the microscopic H-theorem can be obtained for the microscopic component, which is  important in order to use the energy method.  The macro-micro decomposition is gained with the help of spectral structure  of the  linearized collision operator of the Boltzmann equation.
Here, the  linear Fokker-Planck operator $\mathcal{L}$ and the  collision operator of the Boltzmann equation  have   certain features in common, thus
we can use   the idea of  the macro-micro decomposition of $f$ depending on the spectral structure of the operator.
However,  our Navier-Stokes-Vlasov-Fokker-Planck system  is quite different from the pure Boltzmann equation or  the coupled Boltzmann
equations, therefore several new difficulties arise as described below.

For the Boltzmann equation, as  mentioned in \cite{LLZ}, the collision particles
have the same mass and momentum as well as kinetic energy.  As we mentioned earlier, there exists momentum and energy exchange between particles and the surrounding fluids. As a result, some linear terms appear in the kinetic equation  and  some coupling terms   appear in the
fluid equations, which  leads to some new  difficulties.
Because we want to establish the global well-posesdness in the case of small  perturbation near  an equilibrium, naturally these terms may be very small,
but throughout our analysis, the main difficulties arise from these linear terms which are worse than the nonlinear terms.
In order to handle these linear terms, we need  certain dissipation effect  and expect it   from the linear Fokker-Planck operator.
 Unfortunately,  as in  Section \ref{mm}  from the  macro-micro decomposition only by  the null space of $\mathcal{L}$, we know that the dissipation induced by  the linear Focker-Planck operator  is only partial, which is not enough to control new linear terms, we eventually find a new  macro-micro decomposition, and achieve   the desired dissipative effect  under the new decomposition.
Achieving such control is one of the main new contributions of the present paper.

 Due to the interaction between the particles and the fluids,  no existing results  on the Vlasov-Fokker-Planck system yields the regularity of $f$. Therefore,  the new mixed estimates involving  the derivatives of the particle velocity $v$  variable are necessary.

 To obtain the uniform estimates, the  macroscopic equations of the particles,  i.e., \eqref{v2.13}-\eqref{vx2.15},  are important  to achieve  that the macroscopic part is bounded by its microscopic part due to \eqref{v2.1}.
The macroscopic equations behave like   elliptic so that it is straightforward to estimate $L^{2}$ norms of their derivatives for $a,b,\omega$.
In periodic domain, $L^{2}$ norms of  $a,b,\omega$ can be estimated by the Poincar\'e inequality.
This leads to  different decay rates in the whole space and the periodic domain as time tends to infinity.

 In order to achieve  the optimal decay rate in the whole space, the main ideas are based on  the Fourier analysis to the linearized Cauchy problem of \eqref{v1.5}-\eqref{vx1.7}
 and the energy-spectrum method as in  \cite{CDM,CKL,DUYZ}.
Our main difficulties arise  from the strong coupling  terms  in the system \eqref{v1.5}-\eqref{vx1.7}, namely the nonlinear terms,
that is because  the Fourier transform of the product of functions is a convolution,  which is difficult in our global time-decay analysis.
Selecting subtly the functions $G,\varphi$  in \eqref{vx3.1} as the nonhomogeneous source  plays an important role to overcome this difficulty.

With the above new ideas and techniques, we shall be able to establish the global existence of strong solutions and optimal decay rates for the compressible non-isentropic Navier-Stokes-Vlasov-Fokker-Planck both in the whole space and in periodic domains.
We remark that the methods introduced in the present work may be applied  for other related non-sentropic kinetic-fluid models such as the drag force exerted by the fluid depending on the density of the fluid.

We organize the rest of the paper as follows.
In Section 2, we reformulate the system \eqref{v1.2}-\eqref{vx1.33} near the global equilibrium
  state, present the coercivity estimate of the linear part and the macro-micro decomposition, and state our main results.
 In Section 3, we first derive the  uniform-in-time a priori estimates
and  then establish the existence  of  global strong solution.
In Section 4, we prove the rate of convergence of solutions.
In Section 5, we adapt our proof to the periodic domain  case.

\bigskip

\section{Preliminaries and main results}

In this section, we reformulate the system \eqref{v1.2}-\eqref{vx1.33} near the global equilibrium
  state, present the coercivity estimate of the linear part and the macro-micro decomposition, and state our main results.

\subsection{Reformulation}
We consider the solution $(F,n,u,\tilde{\theta})$ of the system \eqref{v1.2}-\eqref{vx1.33} near the global equilibrium $(M,1,0,1)$, i.e.,
 $$F=M+\sqrt{M}f,\quad n=1+\rho,\quad u= u,\quad \tilde{\theta}=1+\theta.$$
We shall take the constants $\mu, \kappa, R$ to be one in this paper since their values do not play a role in the analysis.
From the sysem \eqref{v1.2}-\eqref{vx1.33}, the perturbations $(f,\rho,u,\theta)$ satisfy the following equations:
\begin{align}
&\partial_{t}f+v\cdot \nabla_{x}f +u\cdot \nabla_{v}f-\frac{1}{2}u\cdot vf-u\cdot v\sqrt{M}-(|v|^{2}-3)\sqrt{M}\theta\nonumber\\
 &\qquad \qquad=\mathcal{L}f +\theta M^{-\frac{1}{2}}\Delta_{v}(\sqrt{M}f), \label{v1.5}\\
 &\partial_{t}\rho+u\cdot \nabla \rho+(1+\rho)\dv u=0,\label{v1.6}\\
 &\partial_{t}u+u\cdot \nabla u+\frac{1+\theta}{1+\rho}\nabla\rho+\nabla_{x}\theta=\frac{1}{1+\rho}\big(\Delta u-u(1+a)+b\big),\label{v1.7}\\
 &\partial_{t}\theta+u\cdot\nabla \theta+\theta\dv u+\dv u-\sqrt{6}\omega+3\theta\nonumber\\
 &\qquad\qquad=\frac{1}{1+\rho}\big(\Delta\theta+|u|^{2}-2u\cdot b+a|u|^{2}-3a\theta\big)
- \frac{\rho}{1+\rho}\big(\sqrt{6}\omega-3\theta\big).\label{vx1.7}
\end{align}
Correspondingly, the  initial data  becomes
\begin{align}
(f,\rho,u,\theta)|_{t=0}&=(f_0(x,v), \rho_0(x),u_{0}(x),\theta_{0}(x))\nonumber\\
&= \Big(\frac{F_{0}-M}{\sqrt{M}}, n_0(x)-1,u_{0}(x),\tilde{\theta}_{0}(x)-1\Big),\label{v1.8}
\end{align}
where
$F_{0}=F(0,t,x),n_{0}=n(0,t,x),u_{0}=u(0,t,x),\tilde{\theta}_{0}=\tilde{\theta}(0,t,x)$ is a small perturbation near the above equilibrium.

In \eqref{v1.5}-\eqref{vx1.7}, we denote  the linearized Fokker-Planck operator  $\mathcal{L}$ by
$$\mathcal{L}f =\frac{1}{\sqrt{M}}\nabla_{v}\cdot\Big[M\nabla_{v}\big(\frac{f}{\sqrt{M}}\big)\Big],$$
and $a=a^{f},b=b^{f},\omega=\omega^{f}$ are defined by
\begin{align}
&a^{f}(t,x)=\int_{\mathbb{R}^{3}}\sqrt{M}f(t,x,v) \,{\rm d}v,\nonumber\\
& b^{f}(t,x)=\int_{\mathbb{R}^{3}}v\sqrt{M}f(t,x,v) \,{\rm d}v,\nonumber\\
&\omega^{f}(t,x)=\int_{\mathbb{R}^{3}}\frac{|v|^{2}-3}{\sqrt{6}}\sqrt{M}f(t,x,v) \,{\rm d}v. \nonumber
\end{align}

\smallskip
\subsection{Notations}

For $\nu(v)=1+|v|^{2}$,    $|\cdot|_{\nu}$ is the norm defined by
$$|g|_{\nu}^{2}:=\int_{\mathbb{R}^{3}}\Big\{|\nabla_{v}g(v)|^{2}+\nu(v)|g(v)|^{2}\Big\} \,{\rm d}v,\quad g=g(v).$$
 $\langle \cdot,\cdot\rangle$ is  the inner product of the space $L^{2}_{v}$, namely,
$$\langle g,h\rangle   :=\int_{\mathbb{R}^{3}}g(v)h(v) \,{\rm d}v,\qquad g,h\in L^{2}_{v}.$$
In case of no confusion,  we denote by $\|\cdot\|$   the norm of  $L^{2}_{x}$ or $L^{2}_{x,v}$ for simplicity.
Define
$$\|g\|_{\nu}^{2}:=\iint_{\Omega\times\mathbb{R}^{3}}\Big\{|\nabla_{v}g(x,v)|^{2}+\nu(v)|g(x,v)|^{2}\Big\}\, {\rm d}x  {\rm d}v, \quad g=g(x,v).$$

For $q\geq 1$, we also denote
$$Z_{q}=L^{2}_{v}(L^{q}_{x})=L^{2}(\mathbb{R}^{3}_{v};L^{q}(\mathbb{R}^{3}_{x})),\quad
\|g\|^{2}_{Z_{q}}=\int_{\mathbb{R}^{3}}\Big(\int_{\mathbb{R}^{3}}|g(x,v)|^{q}\, {\rm d}x\Big)^{\frac{2}{q}}\,{\rm d}v.$$
The norm  $\|(\cdot,\cdot,\cdot,\cdot)\|_{\mathcal{Z}_{q}}$ is defined by
 $$\|(f,\rho,u,\theta)\|_{\mathcal{Z}_{q}}=\|f\|_{Z_{q}}+\|(\rho,u,\theta)\|_{L^{1}},$$
 for $f=f(x,v),\big(\rho,u,\theta\big)=\big(\rho(x),u(x),\theta(x)\big)$ and $q\geq 1$.

For an integrable function $g:\,\mathbb{R}^{3}\longrightarrow \mathbb{R}$, its Fourier transform $\hat{g}=\mathcal{F}g$ is defined by
$$\hat{g}(\xi)=\mathcal{F}g(\xi)=\int_{\mathbb{R}^{3}}e^{-ix\cdot\xi}g(x)\, {\rm d}x ,\quad x\cdot\xi=\sum_{j=1}^{3}x_{j}\xi_{j},$$
for $\xi\in \mathbb{R}^{3}$

 For multi-indices $\alpha=(\alpha_{1},\alpha_{2},\alpha_{3})$ and $\beta=(\beta_{1},\beta_{2},\beta_{3})$, we denote by
$$\partial^{\alpha}_{\beta}\equiv \partial^{\alpha_{1}}_{x_{1}} \partial^{\alpha_{2}}_{x_{2}} \partial^{\alpha_{3}}_{x_{3}}
 \partial^{\beta_{1}}_{v_{1}}\partial^{\beta_{2}}_{v_{2}}\partial^{\beta_{3}}_{v_{3}}$$
 the partial derivatives with respect to $x=(x_1,x_2,x_3)$ and $v=(v_1,v_2,v_3)$.
 The length of $\alpha$ and $\beta$ are defined as $|\alpha|=\alpha_{1}+\alpha_{2}+\alpha_{3}$ and $ |\beta|=\beta_{1}+\beta_{2}+\beta_{3}$.
 We shall use the following norms:
 \begin{align}
    \|g\|_{H^s}:=\sum_{|\alpha|\leq s}\|\partial^\alpha g\|, \quad
    \|g\|_{H^s_{x,v}}:=\sum_{|\alpha|+|\beta|\leq s}\|\partial^\alpha_\beta g\|. \nonumber
 \end{align}

We shall use the letter $C$ to denote  a generic positive (generally large) constant,
 $\lambda$ a generic positive (generally small) constant;
 and use the  symbol $A\sim B$ to denote the relation
 $\frac{1}{C} A\leq B\leq C A$ for some constant $C>0$.

\subsection{Coercivity estimate of the linear part in \eqref{v1.5}}
In this subsection, we first apply the similar coercivity estimate of the linear collision operator of Boltzmann equation to the linearized Fokker-Planck operator $\mathcal{L}$
by spectrum analysis, that is, \eqref{Co} holds. Then, we find that the dissipative effect of $\mathcal{L}$ is partial and new difficulties have arisen.

In \eqref{v1.5}, we denote  the important linear part by
$$Lf=\mathcal{L}f+A=\mathcal{L}f+u\cdot v\sqrt{M}+(|v|^{2}-3)\sqrt{M}\theta.$$

It is well-known from \cite{A} that the classical linearized Fokker-Planck operator $\mathcal{L}$ enjoys the following dissipative properties:
\begin{itemize}
\item[(1)]The null space of $\mathcal{L}$ is the one dimensional space
                   $$\mathcal{N}_{0}= \text{Span}\Big\{\sqrt{M}\Big\}.$$
\item[(2)]
Define the projection in $L^{2}_{x,v}$ to the null space $\mathcal{N}_{0}$ by
$$\mathbf{P}_{0}f:=a^{f}\sqrt{M},\quad a^{f}(t,x)=\int_{\mathbb{R}^{3}}\sqrt{M}f(t,x,v) \,{\rm d}v.$$
Using the integration by parts,  we have
$$-\int_{\mathbb{R}^{3}}f\mathcal{L}f\, {\rm d}v=\int_{\mathbb{R}^{3}} \Big|\nabla_{v}(\mathbf{I}-\mathbf{P}_{0})f+\frac{v}{2}(\mathbf{I}-\mathbf{P}_{0})f\Big|^{2}\,{\rm d}v.$$
\item[(3)] There exists a constant  $\lambda_{0}>0$, such that the following coercivity estimate holds:
\begin{align}
-\int_{\mathbb{R}^{3}}f\mathcal{L}f\, {\rm d}v\geq \lambda_{0}|\{\mathbf{I}-\mathbf{P}_{0}\}f|^{2}_{\nu} , \quad \forall\, f=f(v).\label{Co}
\end{align}
 \end{itemize}

We hope  that the linear part $L$ also has  a similar coercivity estimate. However, we notice that it is straightforward to make estimates on $L-\mathcal{L}$ as
$$\int_{\mathbb{R}^{3}}\int_{\mathbb{R}^{3}}\big(u\cdot v\sqrt{M}+(|v|^{2}-3)\sqrt{M}\theta\big)f \,{\rm d}x{\rm d}v\leq C \big(\|u\|+\|\theta\|\big)\|\{\mathbf{I}-\mathbf{P}_{0}\}f\|,$$
and from \eqref{Co}, one gets
$$-\int_{\mathbb{R}^{3}}\int_{\mathbb{R}^{3}}f\mathcal{L}f\, {\rm d}v{\rm d}x\geq \lambda_{0}\|\{\mathbf{I}-\mathbf{P}_{0}\}f\|^{2}_{\nu}.$$
 It is difficult to decide which of the two terms on the right hand side of the above inequalities is bigger. To this end, we decompose
$f$ into the macroscopic component $\mathbf{P}_{0}f$ and the microscopic component$\{\mathbf{I}-\mathbf{P}_{0}\}f$, however, the dissipative effect for $\{\mathbf{I}-\mathbf{P}_{0}\}f$, i.e.,  \eqref{Co}, is not enough to deal with the linear part $L$.
  It is nontrivial to get
a coercivity estimate on the linear part $L$. In order to control $L$, we must extract part of
dissipation of $\mathcal{L}$ corresponding to the momentum component and the energy component respectively.

\subsection{New Macro-micro decomposition}  \label{mm}
In this subsection, we are concerned with the macro-micro decomposition of the solution into its macroscopic (fluid dynamic) and microscopic (kinetic) components.
In view of the difficulties mentioned above, we want to extract a better dissipative effect, which is hard and achieved  based on the inspiration from the idea of spectrum analysis of the linear operator.
In short, since the coupling terms in our system involve the momentum exchange   and the energy exchange,
we expand the original null space $\mathcal{N}_{0}$ to the new space  $\mathcal{N}$  below. Then,  by the macroscopic and  microscopic projection on the new space, we can get the new macro-micro decomposition of the solution,  as described  below.

Denote the linear space $\mathcal{N}$ by
$$\mathcal{N}=\text{Span} \Big\{\sqrt{M},v_{1}\sqrt{M},v_{2}\sqrt{M},v_{3}\sqrt{M},|v|^{2}\sqrt{M}\Big\},$$
and it has  the  following set of  orthogonal basis
\begin{equation}
\left\{\begin{aligned}
& \chi_{0}=\sqrt{M},\qquad \chi_{i}=v_{i}\sqrt{M},\quad i=1,2,3,\nonumber\\
& \chi_{4}=\frac{|v|^{2}-3}{\sqrt{6}} \sqrt{M}.\nonumber
\end{aligned}  \right.
\end{equation}
Define the projector operator $\mathbf{P}$ by
\begin{align}
 \mathbf{P}: \  L^{2}\rightarrow \mathcal{N},\quad  f\longmapsto  \mathbf{P}f=\Big\{a^{f}+b^{f}\cdot v+\omega^{f}\frac{|v|^{2}-3}{\sqrt{6}}\Big\}\sqrt{M}.\nonumber
\end{align}
We also introduce the projector $\mathbf{P}_{1},\mathbf{P}_{2}$ respectively by
\begin{align}
&\mathbf{P}_{1}f:=b^{f}\cdot v \sqrt{M},\nonumber\\
&\mathbf{P}_{2}f:=\omega^{f}\frac{|v|^{2}-3}{\sqrt{6}} \sqrt{M},\nonumber
\end{align}
so the projector $\mathbf{P}$ can be also written as
$$\mathbf{P}:=\mathbf{P}_{0}\oplus\mathbf{P}_{1}\oplus\mathbf{P}_{2}.$$

As usual, for fixed $(t, x), f(t, x, v)$ can be uniquely
decomposed as
\begin{equation}\label{decomp}
\left\{\begin{aligned}
 &f=\mathbf{P}f+\{\mathbf{I}-\mathbf{P}\}f,\\
 &\mathbf{P}f=\Big\{a^{f}+b^{f}\cdot v+\omega^{f}\frac{|v|^{2}-3}{\sqrt{6}}\Big\}\sqrt{M},
\end{aligned}  \right.
\end{equation}
where $\mathbf{P}f$ is called the macroscopic component of $f$, while $\{\mathbf{I}-\mathbf{P}\}f $ is called the corresponding
microscopic component. Interestingly,  our new  decomposition is formally the same as  that of the  linearized collision operator of the  Boltzmann equation \cite{UY}, but our decomposition here comes not only from the spectral analysis of the  linearized Fokker-Planck operator $\mathcal{L}$, but also from the coupling term involving  momentum   and  energy exchanges. Therefore, we further reveal the internal relations and differences between the particle-fluid system and the Boltzmann equation or coupled Boltzmann equation.

According to  new decomposition \eqref{decomp}, the  linearized Fokker-Planck operator $\mathcal{L}$ satisfies the following additional properties besides the above properties (1)-(3):
\begin{itemize}
\item[(4)]$\mathcal{L}f $ can be written as
$$\mathcal{L}f =\mathcal{L}\{\mathbf{I}-\mathbf{P}\}f+\mathcal{L}\mathbf{P}f=\mathcal{L}\{\mathbf{I}-\mathbf{P}\}f-\mathbf{P}_{1}f-2\mathbf{P}_{2}f.$$

\item[(5)]
 There exists a constant  $\lambda>0$, such that
\begin{align}
&\langle -\mathcal{L}\{\mathbf{I}-\mathbf{P}\}f, g\rangle  \geq \lambda|\{\mathbf{I}-\mathbf{P}\}f|^{2}_{\nu},\nonumber\\
&\langle -\mathcal{L}f, f \rangle    \geq \lambda|\{\mathbf{I}-\mathbf{P}\}f|^{2}_{\nu}+|b^{f}|^{2}+2|\omega^{f}|^{2}. \label{v2.1}
\end{align}
 \end{itemize}



\subsection{Main results}
We   now   state our main results. The first result is the global existence of
classical solutions with small initial data and  optimal  algebraic rate of decay in the whole space.
 \begin{Theorem}\label{vt1.1}
 Let $\Omega=\mathbb{R}^{3}$ and $(f_0,\rho_0,u_0,\theta_{0})$ be the initial data such that
$F_{0}=M+\sqrt{M}f_{0}\geq 0,$ and there exists $\varepsilon_{0}>0$,
 $\|f_{0}\|_{H^{4}_{x,v}}+\|(\rho_{0},u_{0},\theta_{0})\|_{H^{4}}<\varepsilon_{0}$.
 Then the Cauchy problem \eqref{v1.5}-\eqref{v1.8} admits a unique global solution
 $(f,\rho,u,\theta)$ satisfying  $F=M+\sqrt{M}f\geq 0$ and
 \begin{gather*}
f\in C\big([0,\infty);H^{4}(\mathbb{R}^{3}\times\mathbb{R}^{3})\big)^{3};
\quad (\rho, \, u,\,\theta)\in C\big([0,\infty);H^{4}(\mathbb{R}^{3})\big); \label{v1.9} \\
 \sup_{t\geq 0}\big(\|f(t)\|_{H^{4}_{x,v}}+\|(\rho,u,\theta)(t)\|_{H^{4}}\big)\leq C\big(\|f_{0}\|_{H^{4}_{x,v}}+\|\rho_{0},u_{0},\theta_{0}\|_{H^{4}}\big),\label{v1.10}
\end{gather*}
for some constant $C>0$.
Moreover, if we further assume that
\begin{align}\label{intt}
  \|(f_{0},\rho_{0},u_{0},\theta_{0})\|_{\mathcal{Z}_{1}\cap H^{4}}\leq\varepsilon_{0}
\end{align}then
\begin{align}
\|f(t)\|_{H^{4}_{x,v}}+\|(\rho,u,\theta)(t)\|_{H^{4}}\leq C (1+t)^{-\frac{3}{4}}\|(f_{0},\rho_{0},u_{0},\theta_{0})\|_{\mathcal{Z}_{1}\cap H^{4}},\label{v1.11}
\end{align}
for some constant $C>0$ and all $t\geq 0.$
 \end{Theorem}
\begin{Remark}
Here the algebraic rate of decay in \eqref{v1.11} is optimal in sense that this rate coincides with that of
the corresponding  linear system.
\end{Remark}
The second result is concerned with the periodic spatial  domain. Compared with the case of the whole space,
here we know that the Poincar\'e inequality holds,
 the exponential  convergence rate is  obtained.
 \begin{Theorem}\label{vt1.2}
 Let $\Omega=\mathbb{T}^{3}$  and $(f_0,\rho_0,u_0)$ be the initial data such that
   $F_{0}=M+\sqrt{M}f_{0}\geq 0,$ there exists $\varepsilon_{0}>0$,
 $\|f_{0}\|_{H^{4}_{x,v}}+\|(\rho_{0},u_{0})\|_{H^{4}}<\varepsilon_{0}$,
  and
 \begin{align}
 &\int_{\mathbb{T}^{3}}a_{0}\, {\rm d}x =0, \quad  \int_{\mathbb{T}^{3}}\rho_{0}\, {\rm d}x =0, \nonumber\\
 &\int_{\mathbb{T}^{3}}\left(b_{0}+(1+\rho_{0})u_{0}\right)\, {\rm d}x =0, \nonumber\\
 &\int_{\mathbb{T}^{3}}(1+\rho_{0})(\theta_{0}+\frac{1}{2}|u_{0}|^{2})+\frac{\sqrt{6}}{2}\omega_{0}\, {\rm d}x =0,\nonumber
 \end{align}
 where
 \begin{align}
 a_0=\int_{\mathbb{T}^{3}}\sqrt{M}f_0(x,v) \,{\rm d}v, \;
 b_0=\int_{\mathbb{T}^{3}}v\sqrt{M}f_0(x,v) \,{\rm d}v, \;
 \omega_0=\int_{\mathbb{T}^{3}}\frac{|v|^{2}-3}{\sqrt{6}}\sqrt{M}f_0(x,v) \,{\rm d}v.\nonumber
  \end{align}
Then, the Cauchy problem \eqref{v1.5}-\eqref{v1.8} admits a unique global solution
 $(f,\rho,u)$ satisfying $ F=M+\sqrt{M}f\geq 0$ and
 \begin{gather*}
f\in C([0,\infty);H^{4}(\mathbb{T}^{3}\times\mathbb{R}^{3}));\quad (\rho, \, u,\,\theta)\in C([0,\infty);H^{4}(\mathbb{T}^{3}))^{3};\nonumber\\
\|f(t)\|_{H^{4}_{x,v}}+\|(\rho,u,\theta)(t)\|_{H^{3}}
\leq C\big(\|f_{0}\|_{H^{4}_{x,v}}+\|\rho_{0},u_{0},\theta_{0}\|_{H^{4}}\big)e^{-\lambda t}, \label{v1.12}
\end{gather*}
with $\lambda>0$ some constant, for all $t\geq 0.$
 \end{Theorem}
%

In  the rest of this paper,  we   shall  omit the integral domain $\Omega\times \mathbb{R}^3$ or
$\mathbb{R}^3$ in the integrals for simplicity.

\bigskip

\section{Global existence of classical solutions in the whole space}
   In this section, we shall establish the global existence of classical solutions to the  problem \eqref{v1.5}-\eqref{v1.8}
   in the whole space $\mathbb{R}^3$. We first obtain the uniform a priori estimates.
   Then we construct the unique local solution by an iteration process, and obtain the
   global solution by the  continuum argument.

\subsection{A priori estimates}
In this subsection, we will show the uniform-in-time a priori estimates in the space $\Omega=\mathbb{R}^{3}$ or $\mathbb{T}^{3}$.
We  assume that 
%
$(f,\rho,u,\theta)$ is a classical solution to the Cauchy problem \eqref{v1.5}-\eqref{v1.8} for $0< t<T$ with a fixed $T>0$,
 and
\begin{align}
\sup_{0<t<T}\Big\{\|f(t)\|_{H^{4}_{x,v}}+\|(\rho,u,\theta)(t)\|_{H^{4}_{x}}\Big\}\leq\delta, \label{v2.2}
\end{align}
with $0<\delta<1 $  sufficiently small  constant.

The following lemma  in  \cite{CDM} is useful for the forthcoming estimates.
\begin{Lemma}[see \cite{CDM}]\label{vl2.1}
There exists a   constant $C>0$ such that  for any $f,g\in H^{4 }(\mathbb{R}^{3})$ and
any multi-index $\gamma$ with $1\leq |\gamma|\leq 4$,
\begin{align}
\|f\|_{L^{\infty}(\mathbb{R}^{3})}&\leq C\|\nabla_{x}f\|_{L^{2}(\mathbb{R}^{3})}^{1/2}\|\nabla^{2}_{x}f\|_{L^{2}(\mathbb{R}^{3})}^{1/2},\label{v2.3}\\
\|fg\|_{H^{2}(\mathbb{R}^{3})}&\leq C \|f\|_{H^{2}(\mathbb{R}^{3})}\|\nabla _{x}g\|_{H^{2}(\mathbb{R}^{3})},\label{v2.4}\\
\|\partial^{\gamma}_{x}(fg)\|_{L^{2}(\mathbb{R}^{3})}&\leq C \|\nabla _{x}f\|_{H^{3}(\mathbb{R}^{3})}\|\nabla _{x}g\|_{H^{3}(\mathbb{R}^{3})}.\label{v2.5}
\end{align}
\end{Lemma}

The a priori estimates will be proved in the next two subsections, one for the pure $x$-variable part,  and one for the mixed derivative part.

\subsubsection{ Energy estimates in the $x$-variable}

\begin{Proposition}\label{vl2.2}
For classical solution of the system \eqref{v1.5}-\eqref{v1.8}, we have
\begin{align}
&\frac{1}{2}\frac{\rm d}{{\rm d}t}\|(f,\rho,u,\theta)(t)\|^{2}+\lambda (\|\{\mathbf{I}-\mathbf{P}\}f\|_{\nu}^{2}+
\|b-u\|^{2}+\|\sqrt{2}\omega-\sqrt{3}\theta\|^{2}+\|\nabla u,\nabla\theta\|^{2})
\nonumber\\
&\quad\leq \,C\big(\|(\rho,u,\theta)\|_{H^{2}}+\|\rho\|_{H^{1}}\|u,\theta\|_{H^{1}}+\|u,\theta\|^{2}_{H^{1}}\big)\nonumber\\
 &\quad\quad\times \Big(\|\nabla_{x}(a,b,\omega,\rho,u,\theta)\|^{2}+\|u-b,\sqrt{2}\omega-\sqrt{3}\theta\|^{2}+\|\{\mathbf{I}-\mathbf{P}\}f\|^{2}_{\nu}\Big)\label{v2.6}
\end{align}
for all $0\leq t<T $ with any $T>0 .$
\end{Proposition}

\begin{proof}
 Multiplying \eqref{v1.5}-\eqref{vx1.7} by $f,\rho,u,\theta$ respectively, then taking integration and summation, we finally get
 \begin{align}
 &\frac{1}{2}\frac{\rm d}{{\rm d}t}\big(\|f\|^{2} +\|\rho\|^{2}+\|u\|^{2}+\|\theta\|^{2}\big)+\int\langle -\mathcal{L}\{\mathbf{I}-\mathbf{P}\}f,f\rangle   \, {\rm d}x\nonumber\\
 &\qquad\,\,\qquad+\int\frac{|\nabla u|^{2}}{1+\rho}\, {\rm d}x +\int\frac{|\nabla \theta|^{2}}{1+\rho}\, {\rm d}x +\|b-u\|^{2}+\|\sqrt{2}\omega-\sqrt{3}\theta\|^{2}\nonumber\\
&\,\,=\int u\langle \frac{1}{2}vf,f\rangle   \, {\rm d}x -\langle\frac{au}{1+\rho},u\rangle -\langle\frac{u\cdot b}{1+\rho},\theta\rangle\nonumber\\
&\,\,\quad -\int \theta(\nabla_{v}f-\frac{v}{2}f) (\nabla_{v}f+\frac{v}{2}f){\rm d}x{\rm d}v - \langle\frac{3a\theta}{1+\rho},\theta\rangle\nonumber\\
 &\,\,\quad-\int(u\cdot\nabla u)\cdot u\, {\rm d}x -\int\rho\nabla \rho\cdot u\, {\rm d}x - -\int u\cdot\nabla\theta \theta\, {\rm d}x-\int\frac{\theta-\rho}{1+\rho}\nabla \rho\cdot u\, {\rm d}x\nonumber\\
 &\,\,\quad+\int\frac{(\nabla\rho\cdot\nabla)u}{(1+\rho)^{2}}\cdot u\, {\rm d}x+\int\frac{\nabla\rho}{(1+\rho)^{2}}\nabla \theta\cdot\theta\, {\rm d}x
 -\int (\theta^{2}+\rho^{2})\dv u\, {\rm d}x \nonumber\\
 &\,\,\quad+\langle\frac{(u-b)\cdot u+a|u|^{2}}{1+\rho},\theta\rangle-\langle\frac{\rho}{1+\rho}(u-b),u\rangle
 -\langle\frac{\rho}{1+\rho}\big(\sqrt{6}\omega-3\theta\big),\theta\rangle. \label{v2.7}
\end{align}
By \eqref{v2.1}, we have
$$
\langle -\mathcal{L}\{(\mathbf{I}-\mathbf{P})\}f, f\rangle \geq \lambda_{0}|\{(\mathbf{I}-\mathbf{P})\}f|^{2}_{\nu}.
$$
Next, we  make estimates of the terms on the right hand side of the equality \eqref{v2.7}.\\

(1) The estimates for $\int u\langle \frac{1}{2}vf,f\rangle   \, {\rm d}x -\langle\frac{au}{1+\rho},u\rangle -\langle\frac{u\cdot b}{1+\rho},\theta\rangle$ are obtained as follows.\\
We first notice that
\begin{align}
&\int u\langle \frac{1}{2}vf,f\rangle   \, {\rm d}x -\langle\frac{au}{1+\rho},u\rangle -\langle\frac{u\cdot b}{1+\rho},\theta\rangle \nonumber\\
&=\int u\langle\frac{1}{2}vf,f\rangle \, {\rm d}x - \int a |u|^{2}\, {\rm d}x - \int u\cdot b\theta\, {\rm d}x+\langle\frac{\rho}{1+\rho}au,u\rangle
+\langle\frac{\rho}{1+\rho}u\cdot b,\theta\rangle.\nonumber
\end{align}
According to the macro-micro decomposition  \eqref{decomp}, we get
\begin{align}
&\int u\langle\frac{1}{2}vf,f\rangle \, {\rm d}x - \int a |u|^{2}\, {\rm d}x - \int u\cdot b\theta\, {\rm d}x\nonumber\\
&\quad =\int au\cdot(u-b)\, {\rm d}x +\frac{\sqrt{3}}{3}\int (\sqrt{2}\omega-\sqrt{3}\theta)u\cdot b\, {\rm d}x\nonumber\\
&\qquad+\int u\langle v\mathbf{P}f,(\mathbf{I-P})f\rangle\, {\rm d}x
+\frac{1}{2}\int u\langle \theta,|(\mathbf{I-P})f|^{2}\rangle\, {\rm d}x\nonumber\\
&\quad \leq \|a\|_{L^{6}}\|u\|_{L^{3}}\|u-b\|_{L^{2}}+\|b\|_{L^{6}}\|u\|_{L^{3}}\|\sqrt{2}\omega-\sqrt{3}\theta\|_{L^{2}}\nonumber\\
&\qquad+C\|(a,b,\omega)\|_{L^{6}}\|u\|_{L^{3}}\|(\mathbf{I-P})f\|_{L^{2}}+C\|u\|_{L^{\infty}}\|(\mathbf{I-P})f\|^{2}_{\nu}\nonumber\\
&\quad \leq C\big(\|u\|_{H^{1}}+\|\nabla u\|_{H^{1}}\big)\|(\mathbf{I-P})f\|^{2}_{\nu}\nonumber\\
&\qquad+C\|u\|_{H^{1}}\big(\|u-b\|^{2}_{L^{2}}+\|\sqrt{2}\omega-\sqrt{3}\theta\|^{2}_{L^{2}}+\|\nabla(a,b,\omega)\|^{2}_{L^{2}}\big),\nonumber
\end{align}
and meanwhile by  H\"{o}lder and Sobolev inequalities, one gets
\begin{align}
\langle\frac{\rho}{1+\rho}au,u\rangle
+\langle\frac{\rho}{1+\rho}u\cdot b,\theta\rangle 
&\leq C\int|\rho||a||u|^{2}\, {\rm d}x+C\int|\rho||\theta||u||b|\, {\rm d}x\nonumber\\
&\leq C \|\rho\|_{H^{1}}\big(\|u\|_{H^{1}}+\|\theta\|_{H^{1}}\big)\|\nabla(a,b,u)\|_{L^{2}}.\nonumber
\end{align}
Thus, we eventually have
\begin{align}
&\int u\langle\frac{1}{2}vf,f\rangle \, {\rm d}x - \int a |u|^{2}\, {\rm d}x - \int u\cdot b\theta\, {\rm d}x\nonumber\\
&\quad \leq C\big(\|u\|_{H^{1}}+\|\nabla u\|_{H^{1}}+\|\rho\|_{H^{1}}(\|u\|_{H^{1}}+\|\theta\|_{H^{1}})\big)\nonumber\\
&\qquad \cdot\Big(\|(\mathbf{I-P})f\|^{2}_{\nu}+\|u-b\|^{2}_{L^{2}}+\|\sqrt{2}\omega-\sqrt{3}\theta\|^{2}_{L^{2}}+\|\nabla(a,b,\omega,u)\|^{2}_{L^{2}}\Big).\label{vx2.7}
\end{align}

(2) The estimates for$-\int \theta(\nabla_{v}f-\frac{v}{2}f) (\nabla_{v}f+\frac{v}{2}f){\rm d}x{\rm d}v - \langle\frac{3a\theta}{1+\rho},\theta\rangle$
are obtained as follows.\\
First, the term $-\int \theta(\nabla_{v}f-\frac{v}{2}f) (\nabla_{v}f+\frac{v}{2}f){\rm d}x{\rm d}v$  can be rewritten as
\begin{align}
&\int \theta\big(\nabla_{v}f-\frac{v}{2}f\big) \big(\nabla_{v}f+\frac{v}{2}f\big){\rm d}x{\rm d}v \nonumber\\
&\quad=\int \theta \Big(|\nabla_{v}f|^{2}-\frac{|v|^{2}}{4}f^{2}\Big){\rm d}x{\rm d}v\nonumber\\
&\quad=\int \theta \Big(|\nabla_{v}\{\mathbf{I-P}\}f|^{2}-\frac{|v|^{2}}{4}|\{\mathbf{I-P}\}f|^{2}\Big){\rm d}x{\rm d}v\nonumber\\
&\qquad+\int \theta \Big(|\nabla_{v}\mathbf{P}f|^{2}-\frac{|v|^{2}}{4}|\mathbf{P}f|^{2}\Big){\rm d}x{\rm d}v\nonumber\\
&\qquad+\int 2 \theta \Big(\nabla_{v}\{\mathbf{I-P}\}f\cdot\nabla_{v}\mathbf{P}f-\frac{|v|^{2}}{4}\{\mathbf{I-P}\}f\cdot\mathbf{P}f\Big){\rm d}x{\rm d}v\nonumber\\
&\quad=\int \theta \Big(|\nabla_{v}\{\mathbf{I-P}\}f|^{2}-\frac{|v|^{2}}{4}|\{\mathbf{I-P}\}f|^{2}\Big){\rm d}x{\rm d}v\nonumber\\
&\qquad+\int 2 \theta \Big(\nabla_{v}\{\mathbf{I-P}\}f\cdot\nabla_{v}\mathbf{P}f-\frac{|v|^{2}}{4}\{\mathbf{I-P}\}f\cdot\mathbf{P}f\Big){\rm d}x{\rm d}v
-\sqrt{6}\int a\omega \theta{\rm d}x.\nonumber
\end{align}
Thus,  the following estimates hold
\begin{align}
&-\int \theta(\nabla_{v}f-\frac{v}{2}f) (\nabla_{v}f+\frac{v}{2}f){\rm d}x{\rm d}v - \langle\frac{3a\theta}{1+\rho},\theta\rangle\nonumber\\
&\,\,\, =\sqrt{3}\int a\theta(\sqrt{2}\omega-\sqrt{3}\theta){\rm d}x+\langle\frac{3\rho a\theta}{1+\rho},\theta\rangle\nonumber\\
&\qquad-\int \theta \Big(|\nabla_{v}\{\mathbf{I-P}\}f|^{2}-\frac{|v|^{2}}{4}|\{\mathbf{I-P}\}f|^{2}\Big){\rm d}x{\rm d}v\nonumber\\
&\qquad-\int 2 \theta \Big(\nabla_{v}\{\mathbf{I-P}\}f\cdot\nabla_{v}\mathbf{P}f-\frac{|v|^{2}}{4}\{\mathbf{I-P}\}f\cdot\mathbf{P}f\Big){\rm d}x{\rm d}v\nonumber\\
&\,\,\,\leq C\|a\|_{L^{6}}\|\theta\|_{L^{3}}\|\sqrt{2}\omega-\sqrt{3}\theta\|_{L^{2}}+C\|a\|_{L^{6}}\|\rho\|_{L^{6}}\|\theta\|^{2}_{L^{3}}\nonumber\\
&\qquad +C\|\theta\|_{L^{\infty}}\|\{\mathbf{I-P}\}f\|^{2}_{\nu}+C\|\{\mathbf{I-P}\}f\|_{\nu}\|(a,b,\omega)\|_{L^{6}}\|\theta\|_{L^{3}}\nonumber\\
&\,\,\, \leq C \Big(\|\theta\|_{H^{2}}+\|\theta\|^{2}_{H^{1}}\Big)\nonumber\\
&\qquad \times \Big(\|\{\mathbf{I-P}\}f\|^{2}_{\nu}+\|\nabla_{x}(a,b,\omega)\|^{2}_{L^{2}}
+\|\sqrt{2}\omega-\sqrt{3}\theta\|^{2}_{L^{2}}\Big).\label{vx2.8}
\end{align}

(3)  For the estimates of all the remaining terms, using H\"{o}lder and Sobolev inequalities as well as Lemma \ref{vl2.1}, we easily get
\begin{align}
\int(u\cdot\nabla u)\cdot u\, {\rm d}x &+\int(u\cdot\nabla \rho)\cdot \rho\, {\rm d}x +\int(u\cdot\nabla \theta)\cdot \theta\, {\rm d}x \nonumber\\
&\leq C\|u\|_{L^{3}}\|(\nabla_{x}u,\nabla\rho,\nabla\theta)\|_{L^{2}}\|(u,\rho,\theta)\|_{L^{6}}\nonumber\\
&\leq C\|u\|_{H^{1}}\|(\nabla u,\nabla \rho,\nabla \theta)\|^{2}_{L^{2}},\nonumber\\
\int\frac{1}{(1+\rho)^{2}}\nabla \rho\big(\nabla u\cdot u&+\nabla \theta \cdot \theta\big)\, {\rm d}x \nonumber\\
&\leq C \|(u,\theta)\|_{H^{2}}\big(\|\nabla \rho\|^{2}+\|\nabla u\|^{2}+\|\nabla \theta\|^{2}\big),\nonumber\\
\int\frac{\theta-\rho}{1+\rho}\nabla\rho \cdot u\, {\rm d}x & \leq C \|(\rho,\theta)\|_{L^{3}}\|\nabla_{x}\rho\|_{L^{2}}\|u\|_{L^{6}}\nonumber\\
&\leq C\|(\rho,\theta)\|_{H^{1}}\big(\|\nabla \rho\|^{2}+\|\nabla u\|^{2}\big),\nonumber\\
\int \big(\rho^{2}+\theta^{2}\big)\dv u\, {\rm d}x &\leq C\|(\rho,\theta)\|_{L^{3}}\|\nabla_{x}u\|_{L^{2}}\|(\rho,\theta)\|_{L^{6}}\nonumber\\
&\leq C\|(\rho,\theta)\|_{H^{1}}(\|(\nabla \rho,\nabla\theta\|^{2}+\|\nabla u\|^{2}),\nonumber\\
\Big\langle\frac{\rho}{1+\rho}(b-u),u\Big\rangle&+\Big\langle\frac{u}{1+\rho}(u-b),\theta\Big\rangle +\Big\langle\frac{\rho}{1+\rho}(\sqrt{6}\omega-3\theta),\theta\Big\rangle\nonumber\\
&\leq C \|(u,\theta)\|_{H^{1}}\Big(\|(b-u,\sqrt{2}\omega-\sqrt{3}\theta)\|^{2}+\|(\nabla \theta,\nabla\rho)\|^{2}\Big),\nonumber\\
\Big\langle\frac{1}{1+\rho}a|u|^{2},\theta\Big\rangle &\leq C\|a\|_{L^{6}}\|\theta\|_{L^{6}}\|u\|^{2}_{L^{3}}\nonumber\\
&\leq C\|u\|^{2}_{H^{1}}\Big(\|\nabla a\|^{2}+\|\nabla\theta\|^{2}\Big). \label{vx2.9}
\end{align}

Substituting all the above estimates \eqref{vx2.7}, \eqref{vx2.8}, \eqref{vx2.9}  into \eqref{v2.7} and using \eqref{v2.2}, we   obtain \eqref{v2.6}.
\end{proof}

\begin{Proposition}\label{vl2.3}
For smooth solutions of the problem \eqref{v1.5}-\eqref{v1.8}, we have
\begin{align}
&\frac{1}{2}\frac{\rm d}{{\rm d}t}\sum_{1\leq |\alpha|\leq 4}\Big\{\big\|\partial^{\alpha}f \big\|^{2}+ \big\|\partial^{\alpha}\rho\big\|^{2}
+\big\|\partial^{\alpha}u\big\|^{2}+\big\|\partial^{\alpha}\theta\big\|^{2}\Big\}\nonumber\\
&\quad+\lambda \sum_{1\leq |\alpha|\leq 4}\Big\{\|\{\mathbf{I}-\mathbf{P}\}\partial^{\alpha}f\|_{\nu}^{2}+
\|\partial^{\alpha}(b-u,\sqrt{2}\omega-\sqrt{3}\theta)\|^{2}+\|\nabla \partial^{\alpha} (u,\theta)\|^{2}\Big\}\nonumber\\
&\leq \,C\big\{\|(\rho,u,\theta)\|_{H^{4}}+\|(\rho,u,\theta)\|^{2}_{H^{4}}\big\}\cdot\big(1+\|\rho\|^{4}_{H^{4}}\big)^{2}
\nonumber\\
&\quad\times \Big\{\|\nabla_{x}(a,b,\omega,\rho,u,\theta)\|^{2}_{H^{3}}+\sum_{1\leq |\alpha'|\leq 4}\|\{\mathbf{I}-\mathbf{P}\}\partial^{\alpha'}f\|^{2}_{\nu}\Big\},\label{v2.8}
\end{align}
for all $0\leq t<T $ with any $T>0 .$
\end{Proposition}
\begin{proof}Here, we follow  the similar argument \cite[Proposition 4.2]{LMW}.
Applying differentiation $\partial^{\alpha} (1\leq |\alpha|\leq 4)$ to  \eqref{v1.5}-\eqref{vx1.7} yields
\begin{align}
&\partial_{t}(\partial^{\alpha}f)+v\cdot \nabla_{x}(\partial^{\alpha}f) +u\cdot \nabla_{v}(\partial^{\alpha}f)
-\partial^{\alpha}u\cdot v M^{\frac{1}{2}}-(|v|^{2}-3)M^{\frac{1}{2}}\partial^{\alpha}\theta-\mathcal{L}\partial^{\alpha}f\nonumber\\
 &\qquad =\frac{1}{2}\partial^{\alpha}\big(u\cdot vf\big)
 +[-\partial^{\alpha},u\cdot \nabla_{v}]f+\partial^{\alpha}\big(\theta\cdot M^{-\frac{1}{2}}\Delta_{v}(M^{\frac{1}{2}}f)\big),\label{v2.9}\\
 &\partial_{t}(\partial^{\alpha}\rho)+u\cdot \nabla \partial^{\alpha}\rho+(1+\rho)\dv \partial^{\alpha}u
 =[-\partial^{\alpha},\rho \nabla_{x}\cdot]u+[-\partial^{\alpha}, u \cdot\nabla_{x}]\rho,\label{v2.10}\\
 &\partial_{t}(\partial^{\alpha}u)+u\cdot \nabla (\partial^{\alpha} u)+\nabla\partial^{\alpha}\theta
 +\nabla\partial^{\alpha}\rho
 -\partial^{\alpha}\Big(\frac{1}{1+\rho}\Delta u\Big)-\partial^{\alpha}(b-u)\nonumber\\
 &\qquad=[-\partial^{\alpha},u\cdot\nabla_{x}]u+\partial^{\alpha}\Big(\frac{\rho-\theta}{1+\rho}\nabla_{x}\rho\Big)
 -\partial^{\alpha}\Big(\frac{1}{1+\rho}ua\Big),\label{v2.11}\\
 &\partial_{t}(\partial^{\alpha}\theta)+\dv(\partial^{\alpha}u)-\partial^{\alpha}\Big(\frac{1}{1+\rho}\Delta \theta\Big)
 -\sqrt{3}\big(\sqrt{2}\partial^{\alpha}\omega-\sqrt{3}\partial^{\alpha}\theta\big)\nonumber\\
 &\qquad =-\partial^{\alpha}\big(u\cdot\nabla\theta+\theta\dv u\big)+\partial^{\alpha}\Big(\frac{1}{1+\rho}\big((u-b)\cdot u+a|u|^{2}-u\cdot b\big)\Big)\nonumber\\
 &\qquad\quad-\sqrt{3}\Big(\frac{\rho}{1+\rho}\big(\sqrt{2}\partial^{\alpha}\omega-\sqrt{3}\partial^{\alpha}\theta\big)\Big)
 -\partial^{\alpha}\Big(\frac{3}{1+\rho}a\cdot\theta\Big),\label{vx2.11}
\end{align}
where $[A,B]$ denotes the commutator $AB-BA$\, for two operators $A$ and $B.$
Now, multiplying \eqref{v2.9}-\eqref{vx2.11} by $\partial^{\alpha}f,\frac{p'(1+\rho)}{(1+\rho)^{2}}\partial^{\alpha}\rho$,$\partial^{\alpha}u,\partial^{\alpha}\theta$
 respectively, then taking integration and summation, we obtain
 \begin{align}
&\frac{1}{2}\frac{\rm d}{{\rm d}t}\Big\{\|\partial^{\alpha}f\|^{2} +\|\partial^{\alpha}\rho\|^{2}
+\|\partial^{\alpha}u\|^{2}+\|\partial^{\alpha}u\|^{2}\Big\}+\lambda\int\big(|\nabla(\partial^{\alpha}u)|^{2}+|\nabla(\partial^{\alpha}\theta)|^{2}\big)\, {\rm d}x\nonumber\\
&\qquad+\int\langle -\mathcal{L}\{\mathbf{I}-\mathbf{P}\}\partial^{\alpha}f,\{\mathbf{I}-\mathbf{P}\}\partial^{\alpha}f\rangle   \, {\rm d}x
  +\|\partial^{\alpha}(b-u)\|^{2}+\|\partial^{\alpha}(\sqrt{2}\omega-\sqrt{3}\theta)\|^{2}\nonumber\\
&\,\,=\int\frac{1}{2}\langle \partial^{\alpha}\big(u\cdot vf\big),\partial^{\alpha}f\rangle   \, {\rm d}x
+\int\langle \partial^{\alpha}\big(\theta\cdot (M^{-\frac{1}{2}}\Delta_{v}(M^{\frac{1}{2}}f)\big),\partial^{\alpha}f\rangle   \, {\rm d}x\nonumber\\
&\qquad +\int\langle [-\partial^{\alpha},u\cdot \nabla_{v}]f,\partial^{\alpha}f\rangle\, {\rm d}x
+\int [-\partial^{\alpha},u\cdot \nabla_{x}]\rho\partial^{\alpha}\rho\, {\rm d}x
\nonumber\\
&\qquad+\int[-\partial^{\alpha},u\cdot\nabla_{x}]u\partial^{\alpha}u\, {\rm d}x
 +\int [-\partial^{\alpha},\rho \nabla_{x}\cdot]u\partial^{\alpha}\rho\, {\rm d}x \nonumber\\
&\qquad -\int\partial^{\alpha}\Big(\frac{a u}{1+\rho}\Big)\cdot\partial^{\alpha}u\, {\rm d}x
 -\int\partial^{\alpha}\Big(\frac{u\cdot b}{1+\rho} \Big)\partial^{\alpha}\theta\, {\rm d}x
  -\int\partial^{\alpha}\Big(\frac{3a \theta}{1+\rho}\Big)\cdot\partial^{\alpha}\theta\, {\rm d}x\nonumber\\
&\qquad+\int\partial^{\alpha}\big(\frac{\rho-\theta}{1+\rho}\nabla \rho\big)\partial^{\alpha}u\, {\rm d}x-\int\rho\dv\partial^{\alpha} u\partial^{\alpha}\rho \, {\rm d}x
-\int\partial^{\alpha} \big(\theta\dv u\big)\partial^{\alpha}\theta\, {\rm d}x\nonumber\\
& \qquad-\int u\cdot\nabla \partial^{\alpha}\rho\partial^{\alpha}\rho\, {\rm d}x-\int u\cdot\nabla \partial^{\alpha}u\cdot\partial^{\alpha}u\, {\rm d}x
-\int \partial^{\alpha}( u\cdot\nabla\theta)\partial^{\alpha}\theta\, {\rm d}x \nonumber\\
&\qquad-\sum_{1\leq\beta\leq\alpha}c_{\alpha,\beta}\int\partial^{\beta}\Big(\frac{1}{1+\rho}\Big)\partial^{\alpha-\beta}\Delta u\partial^{\alpha}u\, {\rm d}x -\sum_{i=1}^{3}\int\nabla\frac{1}{1+\rho}\nabla\partial^{\alpha}u_{i}\partial^{\alpha}u_{i}\, {\rm d}x
\nonumber\\
&\qquad-\sum_{1\leq\beta\leq\alpha}c_{\alpha,\beta}\int\partial^{\beta}\Big(\frac{1}{1+\rho}\Big)\partial^{\alpha-\beta}\Delta \theta\partial^{\alpha}\theta\, {\rm d}x
-\int\nabla\frac{1}{1+\rho}\nabla\partial^{\alpha}\theta\partial^{\alpha}\theta\, {\rm d}x\nonumber\\
&\qquad-\int\partial^{\alpha}\Big(\frac{\rho}{1+\rho}(b-u)\Big)\cdot\partial^{\alpha}u \, {\rm d}x
-\sqrt{3}\int\partial^{\alpha}\Big(\frac{\rho}{1+\rho}(\sqrt{2}\omega-\sqrt{3}\theta)\Big)\cdot\partial^{\alpha}\theta \, {\rm d}x\nonumber\\
&\qquad+\int\partial^{\alpha}\Big(\frac{1}{1+\rho}(u-b)\cdot u\Big)\partial^{\alpha}\theta\, {\rm d}x
+\int\partial^{\alpha}\Big(\frac{1}{1+\rho}a|u|^{2}\Big)\partial^{\alpha}\theta\, {\rm d}x \nonumber\\
&\,\,:=\sum_{j=1}^{23}I_{j}, \label{v2.12}
\end{align}
where $C_{\alpha,\beta}$ is constant depending only on $\alpha$ and $\beta$.

We need to estimate   each term on the right hand side of \eqref{v2.12}. Firstly
\begin{align}
I_{1}\leq\,& C \|\partial^{\alpha}(uf)\|\|v\partial^{\alpha} f\| \nonumber\\
\leq\,& C\|\nabla u\|_{H^{3}}\|\nabla_{x}f\|_{L^{2}_{v}(H^{3}_{x})}\|v\partial^{\alpha}f\|,\nonumber\\
I_{2}=\,&-\int\partial^{\alpha}\Big(\theta(\nabla_{v}f-\frac{v}{2}f)\Big)\cdot\Big(\nabla_{v}\partial^{\alpha}f+\frac{v}{2}\partial^{\alpha}f)\Big)\nonumber\\
\leq\,& \|\partial^{\alpha}\Big(\theta(\nabla_{v}f-\frac{v}{2}f)\Big)\|\|\nabla_{v}\partial^{\alpha}f+\frac{v}{2}\partial^{\alpha}f)\|\nonumber\\
\leq\,&C\|\nabla\theta\|_{H^{3}}\|(\nabla_{v}-\frac{v}{2})\nabla_{x} f\|_{L^{2}_{v}(H^{3})}\|(\nabla_{v}+\frac{v}{2})\nabla_{x} f\|_{L^{2}_{v}(H^{3})}\nonumber\\
\leq\,&C\|\nabla\theta\|_{H^{3}}\Big\{\sum_{1\leq |\alpha'|\leq 4}\|\{\mathbf{I-P}\}\partial^{\alpha'}f\|_{\nu}^{2}+\|\nabla(a,b,\omega)\|^{2}_{H^{3}}\Big\},\nonumber
\end{align}
where we have used   the following fact
\begin{align}
\|(\nabla_{v}-&\frac{v}{2})\nabla f\|_{L^{2}_{v}(H^{3})}\nonumber\\
\leq\,& \sum_{1\leq |\alpha'|\leq 4}\|\{(\nabla_{v}-\frac{v}{2})\}\partial^{\alpha'}f\|_{L^{2}_{x,v}}\nonumber\\
\leq\,& \sum_{1\leq |\alpha'|\leq 4}\Big(\|\{(\nabla_{v}-\frac{v}{2})\}\{\mathbf{I-P}\}\partial^{\alpha'}f\|
+\|\{(\nabla_{v}-\frac{v}{2})\}\mathbf{P}\partial^{\alpha'}f\|\Big)\nonumber\\
\leq\,&\sum_{1\leq |\alpha'|\leq 4}\Big(\|\{\mathbf{I-P}\}\partial^{\alpha'}f\|_{\nu}+\|\partial^{\alpha'}(a,b,\omega)\|\Big),\nonumber
\end{align}
and the similar estimate holds for $\|(\nabla_{v}+\frac{v}{2})\nabla f\|_{L^{2}_{v}(H^{3})}$.

Using H\"{o}lder, Sobolev and Young inequalities, we  get the following bounds
\begin{align}
I_{3}=\,&\int\langle \partial^{\alpha}\big(uf\big),\partial_{v} \partial^{\alpha} f\rangle\, {\rm d}x
\leq\, \| \partial^{\alpha}\big(uf\big)\|\|\nabla_{v} \partial^{\alpha} f\|,\nonumber\\
\leq\, & C\|\nabla u\|_{H^{3}}\|\nabla_{x}f\|_{L^{2}_{v}(H^{3}_{x})}\|\nabla_{v}\partial^{\alpha}f\|\nonumber\\
I_{4}\leq\,& C\| u\|_{H^{4}}\Big(\|\nabla \rho\|^{2}_{H^{3}}+\|\partial^{\alpha}\rho\|^{2}\Big),\nonumber\\
I_{5}\leq\,& C\| u\|_{H^{4}}\Big(\|\nabla u\|^{2}_{H^{3}}+\|\partial^{\alpha} u\|^{2}\Big),\nonumber\\
I_{6}\leq\,& C\| \rho\|_{H^{4}}\Big(\|\nabla \rho\|^{2}_{H^{3}}+\|\partial^{\alpha}\rho\|^{2}\Big),\nonumber
\end{align}

\begin{align}
I_{7}\leq\,& \Big\|\partial^{\alpha}\Big(\frac{au}{1+\rho}\Big)\Big\|\big\|\partial^{\alpha}u\big\|
\leq\, C \Big\|\nabla \Big(\frac{u}{1+\rho}\Big)\Big\|_{H^{3}}\big\|\nabla a\big\|_{H^{3}}\big\|\partial^{\alpha}u\big\|\nonumber\\
\leq\,& C\big(1+\|\rho\|^{2}_{H^{4}}\big)\|u\|_{H^{4}}\Big(\|\nabla a\|^{2}_{H^{3}}+\|\partial^{\alpha}u\|\Big),\nonumber\\
I_{8}\leq\,& \Big\|\partial^{\alpha}\Big(\frac{u\cdot b}{1+\rho}\Big)\Big\|\big\|\partial^{\alpha}\theta\big\|
\leq\, C \Big\|\nabla \Big(\frac{u}{1+\rho}\Big)\Big\|_{H^{3}}\big\|\nabla b\big\|_{H^{3}}\big\|\partial^{\alpha}\theta\big\|\nonumber\\
\leq\,& C\big(1+\|\rho\|^{2}_{H^{4}}\big)\|u\|_{H^{4}}\Big(\|\nabla b\|^{2}_{H^{3}}+\|\partial^{\alpha}\theta\|\Big),\nonumber\\
I_{9}\leq\,& \Big\|\partial^{\alpha}\Big(\frac{a \theta}{1+\rho}\Big)\Big\|\big\|\partial^{\alpha}\theta\big\|
\leq\, C \Big\|\nabla \Big(\frac{\theta}{1+\rho}\Big)\Big\|_{H^{3}}\big\|\nabla a\big\|_{H^{3}}\big\|\partial^{\alpha}\theta\big\|\nonumber\\
\leq\,& C\big(1+\|\rho\|^{2}_{H^{4}}\big)\|\theta\|_{H^{4}}\Big(\|\nabla a\|^{2}_{H^{3}}+\|\partial^{\alpha}\theta\|\Big),\nonumber
\end{align}

\begin{align}
I_{10}=\,& \sum_{1\leq\beta\leq\alpha}C_{\alpha,\beta}\int\partial^{\beta}\Big(\frac{\rho-\theta}{1+\rho}\partial^{\alpha-\beta}\nabla\rho\Big)\partial^{\alpha}u\, {\rm d}x
+\int\frac{\rho-\theta}{1+\rho}\nabla\partial^{\alpha}\rho\partial^{\alpha}u \, {\rm d}x \nonumber\\
=\,& \sum_{1\leq\beta\leq\alpha}C_{\alpha,\beta}\int\partial^{\beta}\Big(\frac{\rho-\theta}{1+\rho}\partial^{\alpha-\beta}\nabla\rho\Big)\partial^{\alpha}u\, {\rm d}x\nonumber\\
\,&-\int\frac{\rho-\theta}{1+\rho}\partial^{\alpha}\rho\partial^{\alpha}\dv u \, {\rm d}x
-\int\nabla\frac{\rho-\theta}{1+\rho}\partial^{\alpha}\rho\partial^{\alpha} u \, {\rm d}x\nonumber\\
\leq\,&C\big(1+\|\rho\|^{4}_{H^{4}}\big)\|(\rho,\theta)\|_{H^{4}}\big(\|\nabla \rho\|^{2}_{H^{3}}+\|\partial^{\alpha}u\|^{2}\big)\nonumber\\
\,&+C\|(\rho,\theta)\|_{H^{4}}\|\partial^{\alpha}\nabla u\|\|\partial^{\alpha}\rho\|
+C\|\rho\|_{H^{4}}\|(\rho,\theta)\|_{H^{4}}\|\partial^{\alpha}u\|\|\partial^{\alpha}\rho\|\nonumber\\
\leq\,& \varepsilon\|\partial^{\alpha}\nabla u\|^{2}+C_{\varepsilon}\big(1+\|\rho\|^{4}_{H^{4}}\big)
\|(\rho,\theta)\|_{H^{4}}\big(\|\nabla \rho\|^{2}_{H^{3}}+\|\partial^{\alpha}u\|^{2}\big),\nonumber\\
I_{11}\leq\,&\varepsilon\|\partial^{\alpha}\nabla u\|^{2}+C_{\varepsilon}\|\rho\|^{2}_{H^{4}}\|\partial^{\alpha}\rho\|^{2},\nonumber\\
I_{12}\leq\,&\varepsilon \|\nabla^{2}u\|^{2}_{H^{3}}+C_{\varepsilon}\|\theta\|^{2}_{H^{4}}\|\partial^{\alpha}\theta\|^{2}\,\nonumber\\
I_{13}\leq\,& C\|u\|_{H^{4}}\|\partial^{\alpha}\rho\|^{2},\nonumber\\
I_{14}\leq\,& C\|u\|_{H^{4}}\|\partial^{\alpha}u\|^{2},\nonumber\\
I_{15}\leq\,&\varepsilon \|\nabla^{2}\theta\|^{2}_{H^{3}}+C_{\varepsilon}\|u\|^{2}_{H^{4}}\|\partial^{\alpha}\theta\|^{2},\nonumber
\end{align}
with $\varepsilon>0$ a small constant, and
\begin{align}
I_{16}\leq\,& C\|\nabla \rho\|_{H^{2}}\|\nabla \partial^{\alpha}u\|\|\partial^{\alpha}u\|+C\big(\|\nabla\rho\|_{H^{3}}+\|\nabla\rho\|^{4}_{H^{3}}\big)\|\nabla u\|^{2}_{H^{3}}\nonumber\\
\leq\,&\varepsilon \|\nabla\partial^{\alpha}u\|^{2}+C_{\varepsilon}\big(\|\nabla\rho\|_{H^{3}}+\|\nabla\rho\|^{4}_{H^{3}}\big)\|\nabla u\|^{2}_{H^{3}},\nonumber\\
I_{17}\leq\,& C\|\nabla \rho\|_{L^{\infty}}\|\nabla \partial^{\alpha}u\|\|\partial^{\alpha}u\|\nonumber\\
\leq\,&\varepsilon \|\nabla\partial^{\alpha}u\|^{2}+C_{\varepsilon}\|\nabla\rho\|^{2}_{H^{3}}\|\partial^{\alpha} u\|^{2},\nonumber\\
I_{18}\leq\,&\varepsilon \|\nabla\partial^{\alpha}\theta\|^{2}+C_{\varepsilon}\big(\|\nabla\rho\|_{H^{3}}+\|\nabla\rho\|^{4}_{H^{3}}\big)\|\nabla \theta\|^{2}_{H^{3}},\nonumber\\
I_{19}\leq\,& \varepsilon \|\nabla\partial^{\alpha}\theta\|^{2}+C_{\varepsilon}\|\nabla\rho\|^{2}_{H^{3}}\|\partial^{\alpha} \theta\|^{2},\nonumber
\end{align}
\begin{align}
I_{20}\leq\,& \Big\|\partial^{\alpha}\frac{\rho(b-u)}{1+\rho}\Big\|\|\partial^{\alpha}u\|\leq\|\nabla \frac{\rho}{1+\rho}\|_{H^{3}}\|\nabla(b-u)\|_{H^{3}}\|\partial^{\alpha}u\|
\nonumber\\
\leq\,&C \big(1+\|\rho\|^{4}_{H^{4}}\big)\|\rho\|_{H^{4}}\Big(\sum_{1\leq|\alpha'|\leq4}\|\partial^{\alpha'}(b-u)\|\Big)\|\partial^{\alpha}u\|\nonumber\\
\leq\,& \varepsilon\sum_{1\leq|\alpha'|\leq4}\|\partial^{\alpha'}(b-u)\|^{2} +C_{\varepsilon}\big(1+\|\rho\|^{4}_{H^{4}}\big)^{2}\|\rho\|^{2}_{H^{4}}\|\partial^{\alpha}u\|^{2},\nonumber\\
I_{21}\leq\,& \varepsilon\sum_{1\leq|\alpha'|\leq4}\|\partial^{\alpha'}(\sqrt{2}\omega-\sqrt{3}\theta)\|^{2} +C_{\varepsilon}\big(1+\|\rho\|^{4}_{H^{4}}\big)^{2}\|\rho\|^{2}_{H^{4}}\|\partial^{\alpha}\theta\|^{2},\nonumber\\
I_{22}\leq\,&\varepsilon\sum_{1\leq|\alpha'|\leq4}\|\partial^{\alpha'}(b-u)\|^{2} +C_{\varepsilon}\big(1+\|\rho\|^{4}_{H^{4}}\big)^{2}\|u\|^{2}_{H^{4}}\|\partial^{\alpha}\theta\|^{2},\nonumber\\
I_{23}\leq\,& C\Big\|\nabla\frac{|u|^{2}}{1+\rho}\Big\|_{H^{3}}\|\nabla a\|_{H^{3}}\|\partial^{\alpha}\theta\|\nonumber\\
\leq\,&C\big(1+\|\rho\|^{4}_{H^{4}}\big)\|u\|^{2}_{H^{4}}\big(\|\nabla a\|^{2}_{H^{3}}+\|\partial^{\alpha}\theta\|^{2}\big).\nonumber
\end{align}

Substituting all  the above estimates on $I_{i}\,(1\leq i\leq 23)$ into \eqref{v2.12} and taking summation over $1\leq |\alpha|\leq 4$, we complete  \eqref{v2.8}.
\end{proof}

To estimate the energy dissipation of $(a,b,\omega)$, namely, $\|\nabla_{x}(a,b,\omega)\|_{H^{3}}$, we need to  deduce the equations satisfied by $(a,b,\omega)$.
We first introduce the moment functional and the energy functional as
\begin{align*}
\Gamma_{i,j}g =\langle (v_{i}v_{j}-1)\sqrt{M},g\rangle,\quad
Q_{i}g =\Big\langle\frac{1}{\sqrt{6}}v_{i}(|v|^{2}-3)M^{\frac{1}{2}},g\Big\rangle,
\end{align*}
for any $g=g(v),\,1\leq i,j\leq3$.
It is easy to verify  that  the  equations of $(a,b,\omega)$ is as follows
\begin{align}
&\partial_{t}a+\dv b=0,  \label{v2.13}  \\
&\partial_{t}b_{i}+\partial_{i}a +\frac{2}{\sqrt{6}}\partial_{i}\omega+\sum_{j=1}^{3}\partial_{x_{j}}\Gamma_{i,j}(\{\mathbf{I}-\mathbf{P}\}f)
=-b_{i}+u_{i}+u_{i}a,\label{v2.14}  \\
&\partial_{t}\omega+\sqrt{2}(\sqrt{2}\omega-\sqrt{3}\theta)-\sqrt{6}a\theta+\frac{2}{\sqrt{6}}\dv b-\frac{1}{\sqrt{6}}u\cdot b\nonumber\\
&\quad\,\,\,+\sum_{i=1}^{3}\partial_{x_{i}}Q_{i}(\{\mathbf{I}-\mathbf{P}\}f)
+\frac{1}{2}\sum_{i=1}^{3}u_{i}Q_{i}(\{\mathbf{I}-\mathbf{P}\}f)=0,\label{vx2.14}\\
&\partial_{j}b_{i}+\partial_{i}b_{j}-(u_{i}b_{j}+u_{j}b_{i})\nonumber\\
&\quad\,\,\,-\frac{2}{\sqrt{6}}\delta_{ij}
\Big(\frac{2}{\sqrt{6}}\dv b-\frac{1}{\sqrt{6}}u\cdot b+\sum_{i=1}^{3}\partial_{x_{i}}Q_{i}(\{\mathbf{I}-\mathbf{P}\}f)
+\frac{1}{2}\sum_{i=1}^{3}u_{i}Q_{i}(\{\mathbf{I}-\mathbf{P}\}f)\Big)\nonumber\\
&\quad\,\,\,=-\partial_{t}\Gamma_{i,j}\{\mathbf{I}-\mathbf{P}\}f +\Gamma_{i,j}(l+r+s), \label{v2.15}\\
&\frac{5}{3}\Big(\partial_{i}\omega+\omega u_{i}-\sqrt{6}\theta b_{i}\Big)-\frac{2}{\sqrt{6}}\sum_{j=1}^{3}\partial_{x_{j}}\Gamma_{i,j}(\{\mathbf{I}-\mathbf{P}\}f)
\nonumber\\
&\quad\,\,\,=-\partial_{t}Q_{i}\big(\{\mathbf{I}-\mathbf{P}\}f\big)+Q_{i}(l+r+s),\label{vx2.15}
\end{align}
where  $l,r,s$ are defined by
\begin{align}
&l:=-v\cdot \nabla_{x}\{\mathbf{I}-\mathbf{P}\}f+\mathcal{L}\{\mathbf{I}-\mathbf{P}\}f,\nonumber\\
&r:=-u\cdot\nabla_{v}\{\mathbf{I}-\mathbf{P}\}f+\frac{1}{2}u\cdot v\{\mathbf{I}-\mathbf{P}\}f,\nonumber\\
&s:=\theta M^{-\frac{1}{2}}\dv_{v}\Big(M^{\frac{1}{2}}(\nabla_{v}-\frac{v}{2})\{\mathbf{I}-\mathbf{P}\}f\Big).\nonumber
\end{align}
In fact,
  multiplying \eqref{v1.5} by $\sqrt{M},
\,v_i\sqrt{M}\,(1\leq i\leq 3),\,\frac{|v|^{2}-3}{\sqrt{6}}M^{\frac{1}{2}}$ respectively and then integrating  with respect to velocity  over $\mathbb{R}^3$, we deduce that \eqref{v2.13}-\eqref{vx2.14} hold .

In order to  get\eqref{v2.15}-\eqref{vx2.15},
we can  rewrite \eqref{v1.5} as
\begin{align}
   \partial_{t}\mathbf{P}f&+v\cdot \nabla_{x}\mathbf{P}f +u\cdot \nabla_{v}\mathbf{P}f-\frac{1}{2}u\cdot v\mathbf{P}f-u\cdot v M^{\frac{1}{2}}
   -\theta\big(|v|^{2}-3\big)M^{\frac{1}{2}}\nonumber\\
   &+\mathbf{P}_1f+2\mathbf{P}_2f-\theta M^{-\frac{1}{2}}\dv_{v}\Big(M^{\frac{1}{2}}(\nabla_{v}-\frac{v}{2})\mathbf{P}f\Big)\nonumber\\
   &=-\partial_t\{\mathbf{I}-\mathbf{P}\}f+l+r+s.\label{vxx2.15}
 \end{align}
Applying  $\Gamma_{ij}$ to \eqref{vxx2.15} and combining \eqref{v2.13}, \eqref{vx2.14}, we   get the equation \eqref{v2.15}.
 Similarly, applying $Q_{i}$ to \eqref{vxx2.15} and combing \eqref{v2.14}, we can obtain the equation \eqref{vx2.15}.

 Define a temporal functional
\begin{align}
\mathcal{E}_{0}(t):=\,&\sum_{|\alpha|\leq 3}\sum_{i,j}\int_{\mathbb{R}^{3}}\partial^{\alpha}
(\partial_{j}b_{i}+\partial_{i}b_{j})\partial^{\alpha}_{x}\Gamma_{i,j}(\{\mathbf{I}-\mathbf{P}\}f)\, {\rm d}x \nonumber\\
&+\sum_{|\alpha|\leq 3}\sum_{i}\int_{\mathbb{R}^{3}}\partial^{\alpha}\partial_{i}\omega\partial^{\alpha}Q_{i}\{\mathbf{I}-\mathbf{P}\}f\, {\rm d}x\nonumber\\
&+\frac{2}{21}\sum_{|\alpha|\leq 3}\int_{\mathbb{R}^{3}}\partial^{\alpha}_{x}a\partial^{\alpha}
\Big(\frac{\sqrt{6}}{5}\sum_{i}\partial_{i}Q_{i}\{\mathbf{I}-\mathbf{P}\}f-\dv b\Big) \, {\rm d}x. \label{v2.16}
\end{align}
We have the following estimate. 

\begin{Proposition}\label{vl2.4}
For classical solution of the system \eqref{v1.5}-\eqref{v1.8}, we have
\begin{align}
\frac{\rm d}{{\rm d}t}\mathcal{E}_{0}(t)&+\lambda \|\nabla_{x}(a,b,\omega)\|^{2}_{H^{3}}\nonumber\\
\leq&\, C\Big(\|\{\mathbf{I}-\mathbf{P}\}f\|^{2}_{L^{2}_{v}(H^{4}_{x})}
+\|u-b\|^{2}_{H^{3}}+\|\sqrt{2}\omega-\sqrt{3}\theta\|^{2}_{H^{3}}\Big)
\nonumber\\
& +C\big(\|u,\theta\|_{H^{3}}+\|u,\theta\|^{2}_{H^{3}}\big)\nonumber\\
&\quad \times\Big(\|\nabla_{x}\{\mathbf{I}-\mathbf{P}\}f\|^{2}_{L^{2}_{v}(H^{3}_{x})}+\|\nabla_{x}(a,b,\omega)\|^{2}_{H^{3}}\Big),\label{v2.17}
\end{align}
for all $0\leq t<T $ with any $T>0 .$
\end{Proposition}
\begin{proof}
We divide this proof into three steps.

{\bf Step 1:}
It follows from \eqref{v2.15} that
\begin{align}
&\sum_{i,j}\|\partial^{\alpha}(\partial_{i}b_{j}+\partial_{j}b_{i})\|^{2}-\sum_{i,j}\int\frac{2}{3}\delta_{ij}\partial^{\alpha}\dv b
\cdot\partial^{\alpha}\big(\partial_{i}b_{j}+\partial_{j}b_{i}\big)\, {\rm d}x \nonumber\\
&\,=-\frac{\rm d}{{\rm d}t}\sum_{i,j}\int\partial^{\alpha}(\partial_{i}b_{j}+\partial_{j}b_{i})\partial^{\alpha}\Gamma_{i,j}(\{\mathbf{I}-\mathbf{P}\}f)\, {\rm d}x \nonumber\\
&\,\quad+\sum_{i,j}\int\partial^{\alpha}(\partial_{i}\partial_{t}b_{j}+\partial_{j}\partial_{t}b_{i})\partial^{\alpha}\Gamma_{i,j}(\{\mathbf{I}-\mathbf{P}\}f)\, {\rm d}x \nonumber\\
&\,\quad+\sum_{i,j}\int\partial^{\alpha}(\partial_{i}b_{j}+\partial_{j}b_{i})\partial^{\alpha}
\Big[\big(u_{i}b_{j}+u_{j}b_{i}\big)+\frac{2}{\sqrt 6}\delta_{ij}\big(\sum_{k}\partial_{k}Q_{k}\{\mathbf{I-P}\}f\nonumber\\
&\,\quad+\frac{1}{2}\sum_{k}u_{k}Q_{k}\{\mathbf{I-P}\}f-\frac{1}{\sqrt{6}}u\cdot b\big)+\Gamma_{i,j}\big(l+r+s\big)\Big]\, {\rm d}x.\label{v2.18}
\end{align}
Using \eqref{v2.14}, Lemma \ref{vl2.1} and Young  inequality, we get
\begin{align}
&\!\!\!\!\!\!\!\!\!\!\!\!\!\!\!\!\!\!\!\!\!\!\!\!
\sum_{i,j}\int\partial^{\alpha}(\partial_{i}\partial_{t}b_{j}+\partial_{j}\partial_{t}b_{i})\partial^{\alpha}\Gamma_{i,j}(\{\mathbf{I}-\mathbf{P}\}f)\, {\rm d}x\nonumber\\
\leq\, & \varepsilon \|\nabla_{x}(a,\omega)\|^{2}_{H^{3}}
+C_{\varepsilon}\|\nabla_{x}\{\mathbf{I}-\mathbf{P}\}f\|^{2}_{L^{2}_{v}(H^{3}_{x})}\nonumber\\
&+C\big(\|u-b\|^{2}_{H^{3}}+\|u\|^{2}_{H^{3}}\|\nabla_{x}a\|^{2}_{H^{2}}\big), \nonumber
\end{align}
with $\varepsilon >0$  sufficiently small.

The  estimate  of the final term on the right hand side of \eqref{v2.18} is as follows,
\begin{align}
&\!\!\!
\sum_{i,j}\int\partial^{\alpha}(\partial_{i}b_{j}+\partial_{j}b_{i})\partial^{\alpha}
\Big[\big(u_{i}b_{j}+u_{j}b_{i}\big)+\frac{2}{\sqrt 6}\delta_{ij}\big(\sum_{k}\partial_{k}Q_{k}\{\mathbf{I-P}\}f\nonumber\\
&\,\quad+\frac{1}{2}\sum_{k}u_{k}Q_{k}\{\mathbf{I-P}\}f-\frac{1}{\sqrt{6}}u\cdot b\big)+\Gamma_{i,j}\big(l+r+s\big)\Big]\, {\rm d}x\nonumber\\
&\leq\, \frac{1}{4}\sum_{i,j}\|\partial^{\alpha}(\partial_{i}b_{j}
+\partial_{j}b_{i})\|^{2}
+C\sum_{i,j}\Big(\|\partial^{\alpha}\big(u_{i}b_{j}+u_{j}b_{i}\big)\|^{2}
+\|\partial^{\alpha}\partial_{k}Q_{k}\{\mathbf{I-P}\}f\|^{2}\nonumber\\
&\,\quad+\|\partial^{\alpha}\big(u_{k}Q_{k}\{\mathbf{I-P}\}f\big)\|^{2}+\|\partial^{\alpha}\big(u\cdot b\big)\|^{2}+\|\partial^{\alpha}\Gamma_{i,j}(l+r+s)\|^{2}\Big),\nonumber\\
&\leq\, \frac{1}{4}\sum_{i,j}\|\partial^{\alpha}(\partial_{i}b_{j}
+\partial_{j}b_{i})\|^{2}+C\|\{\mathbf{I-P}\}f\|^{2}_{L^{2}_{v}(H^{4})}\nonumber\\
&\,\quad+C\|(u,\theta)\|^{2}_{H^{3}}\big(\|\nabla b\|^{2}_{H^{3}}+\|\nabla_{x} \{\mathbf{I-P}\}f \|^{2}_{L^{2}_{v}(H^{3})}\big).\nonumber
\end{align}
Here, we have used the following estimates  
\begin{align}
\sum_{i,j}\|\partial^{\alpha}(u_{i}b_{j}+u_{j}b_{i})\|^{2}+\|\partial^{\alpha}(u\cdot b)\|^{2}&\leq C\|u\|^{2}_{H^{3}}\|\nabla_{x}b\|^{2}_{H^{3}},\nonumber\\
\|\partial^{\alpha}\partial_{k}Q_{k}\{\mathbf{I-P}\}f\|^{2}&\leq C\|\nabla_{x}\{\mathbf{I-P}\}f\|^{2}_{L^{2}_{v}(H^{3})},\nonumber\\
\|\partial^{\alpha}\big(u_{k}Q_{k}\{\mathbf{I-P}\}f\big)\|^{2}&\leq C\|u\|^{2}_{H^{3}}\|\nabla_{x}\{\mathbf{I-P}\}f\|^{2}_{L^{2}_{v}(H^{3})},\nonumber\\
\sum_{i,j}\|\partial^{\alpha}\Gamma_{i,j}(l)\|^{2}&\leq C\|\{\mathbf{I}-\mathbf{P}\}f\|^{2}_{L^{2}_{v}(H^{4})},\nonumber\\
\sum_{i,j}\|\partial^{\alpha}\Gamma_{i,j}(r)\|^{2}&\leq C\|u\|^{2}_{H^{3}}\|\nabla_{x}\{\mathbf{I}-\mathbf{P}\}f\|^{2}_{L^{2}_{v}(H^{3})},\nonumber\\
\sum_{i,j}\|\partial^{\alpha}\Gamma_{i,j}(s)\|^{2}&\leq C\|\theta\|^{2}_{H^{3}}\|\nabla_{x}\{\mathbf{I}-\mathbf{P}\}f\|^{2}_{L^{2}_{v}(H^{3})}.\nonumber
\end{align}
Notice that
\begin{align}
\sum_{i,j}\|\partial^{\alpha}(\partial_{i}b_{j}+\partial_{j}b_{i})\|^{2}=2\|\nabla_{x}\partial^{\alpha}b\|^{2}+2\|\nabla_{x}\cdot\partial^{\alpha}b\|^{2},\nonumber\\
\sum_{i,j}\frac{2}{3}\delta_{ij}\int\partial^{\alpha}\dv b\partial^{\alpha}\big(\partial_{i}b_{j}+\partial_{j}b_{i}\big)\, {\rm d}x
=\frac{4}{3}\|\partial^{\alpha}\dv b\|^{2}.\nonumber
\end{align}
By the above equality, we  substitute the above estimates into \eqref{v2.18}, and then the sum  over $|\alpha|\leq 3$ is obtained as
\begin{align}
&\frac{\rm d}{{\rm d}t}\sum_{|\alpha|\leq3}\sum_{i,j}\int\partial^{\alpha}(\partial_{i}b_{j}+\partial_{j}b_{i})\partial^{\alpha}\Gamma_{i,j}(\{\mathbf{I}-\mathbf{P}\}f)
\, {\rm d}x\nonumber\\
&\,\quad+\sum_{|\alpha|\leq3}\Big(\frac{3}{2}\|\nabla_{x}\partial^{\alpha}b\|^{2}+\frac{1}{6}\|\nabla_{x}\cdot\partial^{\alpha}b\|^{2}\Big)\nonumber\\
&\,\leq\varepsilon \|\nabla \big(a,\omega\big)\|^{2}_{H^{3}}+C_{\varepsilon}\|\{\mathbf{I}-\mathbf{P}\}f\|^{2}_{L^{2}_{x}(H^{4}_{x})}
+C\|u-b\|^{2}_{H^{3}}\nonumber\\
&\,\quad+C\|(u,\theta)\|^{2}_{H^{3}}
\cdot\Big(\|\nabla_{x}\{\mathbf{I}-\mathbf{P}\}f\|^{2}_{L^{2}_{v}(H^{3}_{x})}+\|\nabla_{x}(a,b)\|^{2}_{H^{3}}\Big),\label{v2.19}
\end{align}
with $\varepsilon >0$  sufficiently small.

{\bf Step 2:}
According to \eqref{vx2.15}, through direct calculation, we obtain
\begin{align}
\|\partial^{\alpha}\partial_{i}\omega\|^{2}=\,&-\frac{3}{5}\frac{\rm d}{{\rm d}t}\int\partial^{\alpha}\partial_{i}\omega
\partial^{\alpha}Q_{i}\{\mathbf{I}-\mathbf{P}\}f\, {\rm d}x+\frac{3}{5}\int\partial^{\alpha}\partial_{i}\partial_{t}\omega
\partial^{\alpha}Q_{i}\{\mathbf{I}-\mathbf{P}\}f\, {\rm d}x\nonumber\\
&+\int\partial^{\alpha}\partial_{i}\omega
\partial^{\alpha}\big(\omega u_{i}+\sqrt{6}\theta b_{i}\big)\, {\rm d}x+
\int\partial^{\alpha}\partial_{i}\omega\partial^{\alpha}Q_{i}(l+r+s)\, {\rm d}x\nonumber\\
&+\frac{\sqrt{6}}{5}\int\partial^{\alpha}\partial_{i}\omega
\partial^{\alpha}\big(\sum_{j}\partial_{j}\Gamma_{i,j}\{\mathbf{I}-\mathbf{P}\}f\big)\, {\rm d}x.\label{v2.20}
\end{align}
By means of \eqref{vx2.14} and direct calculation, one gets
\begin{align}
&\int\partial^{\alpha}\partial_{i}\partial_{t}\omega
\partial^{\alpha}Q_{i}\{\mathbf{I}-\mathbf{P}\}f\, {\rm d}x\nonumber\\
&\,=\int\partial_{i}\partial^{\alpha}\Big(\sqrt{6}a\theta-\sqrt{2}(\sqrt{2}\omega-\sqrt{3}\theta)-\frac{2}{\sqrt{6}}\dv b+\frac{1}{\sqrt{6}}u\cdot b \nonumber\\
&\,\quad-\sum_{k}\partial_{k}Q_{k}\{\mathbf{I}-\mathbf{P}\}f-\frac{1}{2}\sum_{k}u_{k}Q_{k}\{\mathbf{I}-\mathbf{P}\}f
\Big)\cdot\partial^{\alpha}Q_{i}\{\mathbf{I}-\mathbf{P}\}f\, {\rm d}x\nonumber\\
&\,\leq \varepsilon \|\nabla b\|^{2}_{H^{3}}+C_{\varepsilon}\|\{\mathbf{I}-\mathbf{P}\}f\|^{2}_{L^{2}_{v}(H^{4})}+C\|\sqrt{2}\omega-\sqrt{3}\theta\|^{2}_{H^{3}}\nonumber\\
&\,\quad+C\big(\|(u,\theta)\|_{H^{3}}+\|u\|^{2}_{H^{3}}\big)\cdot\Big(\|\nabla(a,b)\|^{2}_{H^{3}}+\|\nabla_{x}\{\mathbf{I}-\mathbf{P}\}f\|^{2}_{L^{2}_{v}(H^{3})}\Big).\nonumber
\end{align}
For the remaining  terms on the right hand side of \eqref{v2.20}, we have
\begin{align}
\int\partial^{\alpha}\partial_{i}\omega
\partial^{\alpha}\big(\omega u_{i}& +\sqrt{6}\theta b_{i}\big)\, {\rm d}x
+\int\partial^{\alpha}\partial_{i}\omega\partial^{\alpha}Q_{i}(l+r+s)\, {\rm d}x  \nonumber\\
&+\frac{\sqrt{6}}{5}\int\partial^{\alpha}\partial_{i}\omega
\partial^{\alpha}\big(\sum_{j}\partial_{j}\Gamma_{i,j}\{\mathbf{I}-\mathbf{P}\}f\big)\, {\rm d}x\nonumber\\
\leq& \frac{2}{5}\|\partial^{\alpha}\partial_{i}\omega\|^{2}+C\|\{\mathbf{I}-\mathbf{P}\}f\|^{2}_{L^{2}_{v}(H^{4})}\nonumber\\
&+C\|(u,\theta)\|^{2}_{H^{3}}\Big(\|\nabla(b,\omega)\|^{2}_{H^{3}}+\|\nabla_{x}\{\mathbf{I}-\mathbf{P}\}f\|^{2}_{L^{2}_{v}(H^{3})}\Big).\nonumber
\end{align}
Thus, substituting the above estimates into \eqref{v2.20}, and then taking summation over $|\alpha|\leq 3, 1\leq i\leq 3$, we obtain
\begin{align}
\frac{\rm d}{{\rm d}t}&\sum_{\alpha,i}\int\partial^{\alpha}\partial_{i}\omega
\partial^{\alpha}Q_{i}\{\mathbf{I}-\mathbf{P}\}f\, {\rm d}x+\sum_{\alpha}\|\partial^{\alpha}\nabla \omega\|^{2}\nonumber\\
&\leq \varepsilon \|\nabla b\|^{2}_{H^{3}}+C_{\varepsilon}\|\{\mathbf{I}-\mathbf{P}\}f\|^{2}_{L^{2}_{v}(H^{4})}+C\|\sqrt{2}\omega-\sqrt{3}\theta\|^{2}_{H^{3}}\nonumber\\
&\,\quad +C\big(\|(u,\theta)\|_{H^{3}}+\|(u,\theta)\|^{2}_{H^{3}}\big)\nonumber\\
&\,\quad\times \big(\|\nabla(a,b,\omega)\|^{2}_{H^{3}}
+\|\nabla_{x}\{\mathbf{I}-\mathbf{P}\}f\|^{2}_{L^{2}_{v}(H^{3})}\big),\label{vx2.20}
\end{align}
with $\varepsilon >0$  sufficiently small.

{\bf Step 3:}
Making full use of the equations \eqref{v2.14} and \eqref{vx2.15}, by direct calculations, we obtain  the following equality
\begin{align}
\|\partial^{\alpha}\partial_{i}a\|^{2}=& \frac{\rm d}{{\rm d}t}\int \partial^{\alpha}\partial_{i}a
 \partial^{\alpha}\Big(\frac{\sqrt{6}}{5}Q_{i}\{\mathbf{I}-\mathbf{P}\}f-b_{i}\Big)\, {\rm d}x\nonumber\\
 &-\int \partial^{\alpha}\partial_{i}a_{t}
 \partial^{\alpha}\Big(\frac{\sqrt{6}}{5}Q_{i}\{\mathbf{I}-\mathbf{P}\}f-b_{i}\Big)\, {\rm d}x\nonumber\\
 &+\int \partial^{\alpha}\partial_{i}a
 \partial^{\alpha}\Big((u_{i}-b_{i})+u_{i}a-\frac{2}{\sqrt{6}}\omega u_{i}-2\theta b_{i}\nonumber\\
 &-\frac{7}{5}\sum_{j}\partial_{j}\Gamma_{ij}\{\mathbf{I}-\mathbf{P}\}f-\frac{\sqrt{6}}{5}Q_{i}(l+r+s)\Big)\, {\rm d}x.
\label{vxx2.20}
\end{align}
Using \eqref{v2.13}, one has
\begin{align}
-&\int \partial^{\alpha}\partial_{i}a_{t}
 \partial^{\alpha}\Big(\frac{\sqrt{6}}{5}Q_{i}\{\mathbf{I}-\mathbf{P}\}f-b_{i}\Big)\, {\rm d}x\nonumber\\
&=\int \partial^{\alpha}\dv b
 \partial^{\alpha}\Big(b_{i}-\frac{\sqrt{6}}{5}Q_{i}\{\mathbf{I}-\mathbf{P}\}f\Big)\, {\rm d}x\nonumber\\
&\leq \int \partial^{\alpha}\dv b \partial^{\alpha} \partial_{i} b_{i}\, {\rm d}x +\frac{1}{4}\|\partial^{\alpha}\dv b\|^{2}
+C\|\{\mathbf{I}-\mathbf{P}\}f\|^{2}_{L^{2}_{v}(H^{4})}.\nonumber
\end{align}
For other terms, the following estimate  is obtained:
\begin{align}
\int &\partial^{\alpha}\partial_{i}a
 \partial^{\alpha}\Big((u_{i}-b_{i})+u_{i}a-\frac{2}{\sqrt{6}}\omega u_{i}-2\theta b_{i}\nonumber\\
 &-\frac{7}{5}\sum_{j}\partial_{j}\Gamma_{ij}\{\mathbf{I}-\mathbf{P}\}f-\frac{\sqrt{6}}{5}Q_{i}(l+r+s)\Big)\, {\rm d}x\nonumber\\
\leq & \frac{1}{4}\|\partial^{\alpha}\partial_{i}a\|^{2}+C\Big(
\|\partial^{\alpha}(u_{i}-b_{i})\|^{2}+\|\partial^{\alpha}\big(u_{i}a-\frac{2}{\sqrt{6}}\omega u_{i}-2\theta b_{i}\big)\|^{2}\nonumber\\
&+\sum_{j}\|\partial^{\alpha}\partial_{j}\Gamma_{ij}\{\mathbf{I}-\mathbf{P}\}f\|^{2}+\|\partial^{\alpha}Q_{i}(l+r+s)\|^{2}
\Big)\nonumber\\
\leq &\frac{1}{4}\|\partial^{\alpha}\partial_{i}a\|^{2}+C\|u-b\|^{2}_{H^{3}}+C\|\{\mathbf{I}-\mathbf{P}\}f\|^{2}_{L^{2}_{v}(H^{4})}\nonumber\\
&+C\|(u,\theta)\|^{2}_{H^{3}}\cdot \Big(\|\nabla(a,b,\omega)\|^{2}_{H^{3}}+\|\nabla_{x}\{\mathbf{I}-\mathbf{P}\}f\|^{2}_{L^{2}_{v}(H^{3})}\Big).\nonumber
\end{align}
Thus, substituting the above estimates into \eqref{vxx2.20}, and then taking summation over $|\alpha|\leq 3$, we have
\begin{align}
\frac{\rm d}{{\rm d}t}&\sum_{|\alpha|\leq 3}\int \frac{2}{21} \partial^{\alpha} a \partial^{\alpha}\Big(\sum_{k}Q_{k}\{\mathbf{I}-\mathbf{P}\}f-\dv b\Big)\, {\rm d}x
+\frac{1}{14}\sum_{|\alpha|\leq 3}\|\partial^{\alpha}\nabla a\|^{2}\nonumber\\
&\leq \frac{1}{6}\sum_{|\alpha|\leq 3}\|\partial^{\alpha}\dv b\|^{2}+C\|u-b\|^{2}_{H^{3}}+C\|\{\mathbf{I}-\mathbf{P}\}f\|^{2}_{L^{2}_{v}(H^{4})}\nonumber\\
&\quad+C\|(u,\theta)\|^{2}_{H^{3}}\cdot \Big(\|\nabla(a,b,\omega)\|^{2}_{H^{3}}+\|\nabla_{x}\{\mathbf{I}-\mathbf{P}\}f\|^{2}_{L^{2}_{v}(H^{3})}\Big)\label{v2.21}
\end{align}

Integrating the estimates of the above three parts  \eqref{v2.19}, \eqref{vx2.20} and \eqref{v2.21},
and applying the definition of $\mathcal{E}_{0}(t)$, we conclude that  \eqref{v2.17} holds.
\end{proof}

\begin{Proposition}\label{vl2.5}
For classical solution of the system \eqref{v1.5}-\eqref{v1.8}, we have
\begin{align}
\frac{\rm d}{{\rm d}t}&\sum_{|\alpha|\leq 3}\int_{\mathbb{R}^{3}}\partial^{\alpha}u\cdot\partial^{\alpha}\nabla \rho \, {\rm d}x + \lambda \|\nabla \rho\|^{2}_{H^{3}}\nonumber\\
&\leq C\big(\|u-b\|^{2}_{H^{3}}+\|\nabla u\|^{2}_{H^{4}}+\|\nabla \theta\|^{2}_{H^{3}}\big)\nonumber\\
&\quad +C\big(1+\|\rho\|^{6}_{H^{3}}\big)\Big(\|\rho\|_{H^{4}}+\|(u,\theta)\|^{2}_{H^{3}}\Big)\Big(\|\nabla(a,\rho,u)\|^{2}_{H^{3}}+\|u-b\|^{2}_{H^{3}}\Big)\label{v2.22}
\end{align}
for all $0\leq t<T $ with any $T>0 .$
\end{Proposition}
\begin{proof}
Applying $\partial^\alpha \,(|\alpha|\leq 3)$ to \eqref{v1.7}, and  carrying a  direct calculation, we achieve
\begin{align}
\|\nabla \partial^{\alpha} \rho\|^{2}
\,=&-\int\nabla \partial^{\alpha} \rho \partial^{\alpha}\partial_{t} u\, {\rm d}x  + \int\nabla \partial^{\alpha} \rho \partial^{\alpha}
\Big(\frac{1}{1+\rho}(b-u)-\nabla \theta\Big)\, {\rm d}x \nonumber\\
& +\int\nabla \partial^{\alpha} \rho \partial^{\alpha}\Big\{-u\cdot \nabla u+\frac{1}{1+\rho}\Delta u
+\frac{\rho-\theta}{1+\rho}\nabla \rho-\frac{1}{1+\rho}u a\Big\}\, {\rm d}x \nonumber\\
:=\,& Y_{1}+Y_{2}+Y_{3}. \label{v2.23}
\end{align}
For $Y_{i}\, (i=1,2,3)$, using \eqref{v1.6}, H\"{o}lder, Sobolev and Young inequalities yields
\begin{align}
Y_{1}=\,&-\frac{\rm d}{{\rm d}t}\int\nabla \partial^{\alpha} \rho \partial^{\alpha}u \, {\rm d}x
+\int \partial^{\alpha}\ \dv u \partial^{\alpha}\Big((1+\rho)\ \dv u+u\cdot \nabla \rho\Big)\, {\rm d}x \nonumber\\
\quad\leq\,& -\frac{\rm d}{{\rm d}t}\int\nabla \partial^{\alpha} \rho \partial^{\alpha}u \, {\rm d}x
+C\|\partial^{\alpha}\ \dv u\|^{2}+C\|\rho\|_{H^{4}}\|\nabla u\|^{2}_{H^{3}},\nonumber\\
Y_{2}\leq\, &\frac{1}{8}\|\nabla\partial^{\alpha}\rho\|^{2}+ C\|(u-b,\nabla \theta)\|^{2}_{H^{3}}+C\|\rho\|^{6}_{H^{3}}\|b-u\|^{2}_{H^{3}},\nonumber\\
Y_{3}\leq\,& \frac{1}{8}\|\nabla\partial^{\alpha}\rho\|^{2}+C\|u\cdot \nabla u\|^{2}_{H^{3}}
+C\Big\|\frac{1}{1+\rho}\Delta u\Big\|^{2}_{H^{3}}\nonumber\\
& +C\Big\|\frac{\rho-\theta}{1+\rho}\nabla \rho\Big\|^{2}_{H^{3}}+C\Big\|\frac{1}{1+\rho}u a\Big\|^{2}_{H^{3}}\nonumber\\
\leq\,& \frac{1}{8}\|\nabla\partial^{\alpha}\rho\|^{2} +C\|\nabla u\|^{2}_{H^{4}}
+C\big(1+\|\rho\|^{6}_{H^{3}}\big)\|(\rho,u,\theta)\|^{2}_{H^{3}}\|\nabla(a,\rho,u)\|^{2}_{H^{3}}.\nonumber
\end{align}
With the help of \eqref{v2.2}, substituting the above estimates into \eqref{v2.23}, we obtain \eqref{v2.22}.
\end{proof}
Now let us define  $\mathcal{E}_{1}(t)$ and  $\mathcal{D}_{1}(t)$ by
\begin{align}
\mathcal{E}_{1}(t):=\,&\|f\|^{2} +\|(\rho,u,\theta)\|^{2}+\sum_{1\leq |\alpha|\leq 4}\bigg\{\|\partial^{\alpha}f\|^{2}
 +\|\big(\partial^{\alpha}\rho,\partial^{\alpha}u,\partial^{\alpha}\theta\big)\|^{2}
\bigg\}\nonumber\\
&+\tau_{1}\mathcal{E}_{0}(t)+\tau_{2}\sum_{|\alpha|\leq 3}\int_{\mathbb{R}^{3}}\partial^{\alpha}u\cdot\partial^{\alpha}\nabla \rho \, {\rm d}x  ,\nonumber\\
\mathcal{D}_{1}(t):=\,&\|\nabla (a,b,\omega,\rho,u,\theta)\|^{2}_{H^{3}}+\|(b-u,\sqrt{2}\omega-\sqrt{3}\theta)\|^{2}_{H^{4}}\nonumber\\
&+\sum_{|\alpha|\leq 4}\Big(\|\partial^{\alpha}\{\mathbf{I}-\mathbf{P}\}f\|^{2}_{\nu}+\|\partial^{\alpha}(\nabla u,\nabla \theta)\|^{2}\Big),\nonumber
\end{align}
 where $\mathcal{E}_{1}(t),\mathcal{D}_{1}(t) $ denote the temporal energy functional and the corresponding dissipation rate respectively,  $\tau_{1}$ and $\tau_{2}$ are sufficiently small constants  and will  be determined later.
 Reorganizing the above estimates  in  Propositions \ref{vl2.2}-\ref{vl2.5}, we obtain
 \begin{align}
\frac{\rm d}{{\rm d}t}&\mathcal{E}_{1}(t)+\lambda \mathcal{D}_{1}(t)\nonumber\\
&\leq\, C\Big(1+\|\rho\|^{8}_{H^{4}}\Big)\cdot\Big(\|(\rho,u,\theta)\|_{H^{4}}+\|(\rho,u,\theta)\|^{2}_{H^{4}}
\Big)
 \cdot\Big(\|\nabla (a,b,\omega,\rho,u,\theta)\|^{2}_{H^{3}}\nonumber\\
&\,\quad+\|\big(b-u,\sqrt{2}\omega-\sqrt{3}\theta\big)\|^{2}_{H^{4}}+\sum_{|\alpha|\leq 4}\|\partial^{\alpha}\{\mathbf{I}-\mathbf{P}\}f\|^{2}_{\nu}\Big).\label{v2.24}
\end{align}

\subsubsection{ Energy estimates for mixed space-velocity derivatives}
In this subsection, we shall deal with the energy estimates for the mixed space-velocity derivatives
of $f$, i.e., $\partial^\alpha_\beta f$.
Since $\|\partial^{\alpha}_{\beta}\mathbf{P}f\|\leq C \|\partial^{\alpha}f\| $ for any $\alpha$ and $\beta$,
 thus, it is enough  to  estimate $\|\partial^{\alpha}_{\beta}\{\mathbf{I}-\mathbf{P}\}f\|$ below.

First,we have the following facts
\begin{align}
& \{\mathbf{I}-\mathbf{P}\}(g\cdot h)=g\cdot \{\mathbf{I}-\mathbf{P}\}f- \mathbf{P}\big(g\cdot \{\mathbf{I}-\mathbf{P}\}f\big)+\{\mathbf{I}-\mathbf{P}\}\big(g\cdot \mathbf{P} f \big),\nonumber\\
 &\{\mathbf{I}-\mathbf{P}\}( v\sqrt{M})= 0 ,\quad \{\mathbf{I}-\mathbf{P}\}\big( (|v|^{2}-3)\sqrt{M}\big)= 0,\nonumber\\
 & \{\mathbf{I}-\mathbf{P}\}\mathcal{L}f =\mathcal{L}\{\mathbf{I}-\mathbf{P}\}f,\nonumber
\end{align}
 then, using $\{\mathbf{I}-\mathbf{P}\}$ to the equality \eqref{v1.5}, one gets
\begin{align}
&\partial_{t}\{\mathbf{I}-\mathbf{P}\}f+v\cdot \nabla_{x}\{\mathbf{I}-\mathbf{P}\}f+u\cdot \nabla_{v}\{\mathbf{I}-\mathbf{P}\}f
-\frac{1}{2}u\cdot v\{\mathbf{I}-\mathbf{P}\}f\nonumber\\
&\quad-\theta M^{-\frac{1}{2}}\Delta _{v}\big( M^{\frac{1}{2}}\{\mathbf{I}-\mathbf{P}\}f\big)-\mathcal{L}\{\mathbf{I}-\mathbf{P}\}f\nonumber\\
&=\mathbf{P}\Big\{v\cdot \nabla_{x}\{\mathbf{I}-\mathbf{P}\}f+u\cdot \nabla_{v}\{\mathbf{I}-\mathbf{P}\}f-\frac{1}{2}u\cdot v\{\mathbf{I}-\mathbf{P}\}f\nonumber\\
&\quad-\theta M^{-\frac{1}{2}}\Delta _{v}\big( M^{\frac{1}{2}}\{\mathbf{I}-\mathbf{P}\}f\big)\Big\}
-\{\mathbf{I}-\mathbf{P}\}\Big\{v\cdot \nabla_{x}\mathbf{P}f+u\cdot\nabla_{v}\mathbf{P}f\nonumber\\
&\quad-\frac{1}{2}u\cdot v\mathbf{P}f
-\theta M^{-\frac{1}{2}}\Delta _{v}\big( M^{\frac{1}{2}}\mathbf{P}f\big)\Big\} .\label{v2.25}
\end{align}

\begin{Proposition}\label{vl2.6}
Let $1\leq k\leq4$. Let $(f,\rho,u,\theta)$ is a smooth solution of the system\eqref{v1.5}-\eqref{v1.8}, we have
\begin{align}
\frac{\rm d}{{\rm d}t}&\sum_{\substack{|\beta|=k \\ |\alpha|+|\beta|\leq 4}}\|\partial^{\alpha}_{\beta}\{\mathbf{I}-\mathbf{P}\}f\|^{2}
+\lambda \sum_{\substack{|\beta|=k \\ |\alpha|+|\beta|\leq 4}}\|\partial^{\alpha}_{\beta}\{\mathbf{I}-\mathbf{P}\}f\|^{2}_{\nu}\nonumber\\
\leq \,&C\chi_{2\leq k\leq 4}\sum_{\substack{1\leq |\beta'|\leq k-1\\
|\alpha'|+|\beta'|\leq 4 }}\|\partial^{\alpha'}_{\beta'}\{\mathbf{I}-\mathbf{P}\}f\|^{2}_{\nu}\nonumber\\
&+C\|\nabla (b,\omega)\|^{2}_{H^{4-k}}+\sum_{|\alpha'|\leq4-k+1}\|\partial^{\alpha'}\{\mathbf{I}-\mathbf{P}\}f\|^{2}_{\nu}\nonumber\\
&+C\|(u,\theta)\|^{2}_{H^{3}}\sum_{|\alpha'|\leq4-k+1}\|\partial^{\alpha'}\{\mathbf{I}-\mathbf{P}\}f\|^{2}_{\nu}\nonumber\\
&+C(\|\theta\|_{H^{3}}+\|u\|^{2}_{H^{3}})
\sum_{\substack{1\leq |\beta'|\leq 4 \\|\alpha'|+|\beta'|\leq 4}}\|\partial^{\alpha'}_{\beta'}\{\mathbf{I}-\mathbf{P}\}f\|^{2}_{\nu}\nonumber\\
&+C\|(u,\theta)\|^{2}_{H^{4-k+1}}\|\nabla (b,\omega)\|^{2}_{H^{3}}
 \label{v2.26}
\end{align}
for all $0\leq t<T $ with any $T>0 .$ Here  $\chi_{E}$ is the characteristic function on the  set $E$.
\end{Proposition}

\begin{proof} This proof is based on  some ideas of \cite[Lemma 4.3]{DFT}.

Fix $k\,(1\leq k\leq 4)$.  let $\alpha$ and $\beta$ satisfy  $|\beta|=k$ and $|\alpha|+|\beta|\leq 4$,
  For \eqref{v2.25},  by standard $L^{2}$ energy estimate, we get
\begin{align}
\frac{1}{2}\frac{\rm d}{{\rm d}t}\|\partial^{\alpha}_{\beta}\{\mathbf{I}-\mathbf{P}\}f\|^{2}+\int\langle -L\partial^{\alpha}_{\beta}
\{\mathbf{I}-\mathbf{P}\}f,\partial^{\alpha}_{\beta}\{\mathbf{I}-\mathbf{P}\}f\rangle   \, {\rm d}x :=\sum_{i=1}^{7}R_{i}\label{v2.27}
\end{align}
with
\begin{align}
&R_{1}=\int\langle -\partial^{\alpha}_{x}[\partial^{\beta}_{v},v\cdot \nabla_{x}]\{\mathbf{I}-\mathbf{P}\}f,\partial^{\alpha}_{\beta}\{\mathbf{I}-\mathbf{P}\}f\rangle   \, {\rm d}x ,
\nonumber\\
&R_{2}=\int\langle \partial^{\alpha}_{x}[\partial^{\beta}_{v},-|v|^{2}]
\{\mathbf{I}-\mathbf{P}\}f,\partial^{\alpha}_{\beta}\{\mathbf{I}-\mathbf{P}\}f\rangle   \, {\rm d}x ,\nonumber\\
&R_{3}=\int\langle -\partial^{\alpha}_{\beta}(u\cdot \nabla_{v}\{\mathbf{I}-\mathbf{P}\}f),\partial^{\alpha}_{\beta}\{\mathbf{I}-\mathbf{P}\}f\rangle   \, {\rm d}x ,\nonumber\\
&R_{4}=\int\langle \frac{1}{2}\partial^{\alpha}_{\beta}(u\cdot v\{\mathbf{I}-\mathbf{P}\}f),\partial^{\alpha}_{\beta}\{\mathbf{I}-\mathbf{P}\}f\rangle   \, {\rm d}x ,\nonumber\\
&R_{5}=\int\langle \partial^{\alpha}_{\beta}\big(\theta M^{-\frac{1}{2}}\Delta_{v}(\sqrt{M}\{\mathbf{I}-\mathbf{P}\}f)\big),\partial^{\alpha}_{\beta}\{\mathbf{I}-\mathbf{P}\}f\rangle   \, {\rm d}x\nonumber\\
&R_{6}=\int\Big\langle \partial^{\alpha}_{\beta}\mathbf{P}\Big(v\cdot \nabla_{x}\{\mathbf{I}-\mathbf{P}\}f+u\cdot \nabla_{v}\{\mathbf{I}-\mathbf{P}\}f\nonumber\\
&\qquad   -\frac{1}{2}u\cdot v\{\mathbf{I}-\mathbf{P}\}f-\theta M^{-\frac{1}{2}}\Delta_{v}(\sqrt{M}\{\mathbf{I}-\mathbf{P}\}f)\Big),\partial^{\alpha}_{\beta}\{\mathbf{I}-\mathbf{P}\}f\Big\rangle   \, {\rm d}x ,\nonumber\\
&R_{7}=\int\Big\langle -\partial^{\alpha}_{\beta}\{\mathbf{I}-\mathbf{P}\}\Big(v\cdot \nabla_{x}\mathbf{P}f+u\cdot\nabla_{v}\mathbf{P}f-\frac{1}{2}u\cdot v\mathbf{P}f\nonumber\\
&\qquad-\theta M^{-\frac{1}{2}}\Delta_{v}(\sqrt{M}\mathbf{P}f)\Big),
\partial^{\alpha}_{\beta}\{\mathbf{I}-\mathbf{P}\}f\Big\rangle   \, {\rm d}x .\nonumber
\end{align}
Here the fact   $[\partial^{\beta}_{v},\mathcal{L}]=[\partial^{\beta}_{v},-|v|^{2}]$ has been used.

Now  we make estimates for each term $R_{i}$ in \eqref{v2.27} as the following 
\begin{align}
R_{1}&\leq \eta \|\partial^{\alpha}_{\beta}\{\mathbf{I}-\mathbf{P}\}f\|^{2}
+C_{\eta}\|[\partial^{\beta}_{v},v\cdot\nabla_{v}]\partial^{\alpha}_{x}\{\mathbf{I}-\mathbf{P}\}f\|^{2}\nonumber\\
&\leq \eta \|\partial^{\alpha}_{\beta}\{\mathbf{I}-\mathbf{P}\}f\|^{2}
+C_{\eta}\sum_{|\alpha'|\leq 4-k}\|\partial^{\alpha'}\nabla_{x}\{\mathbf{I}-\mathbf{P}\}f\|^{2}\nonumber\\
&\quad +\chi_{2\leq k\leq4} C_{\eta}\sum_{\substack{1\leq |\beta'|\leq k-1 \\ |\alpha'|+|\beta'|\leq 4}}\|\partial^{\alpha'}_{\beta'}\{\mathbf{I}-\mathbf{P}\}f\|^{2},\nonumber\\
R_{2}&\leq \eta \|\partial^{\alpha}_{\beta}\{\mathbf{I}-\mathbf{P}\}f\|^{2}
+C_{\eta}\|[\partial^{\beta}_{v},-|v|^{2}]\partial^{\alpha}_{x}\{\mathbf{I}-\mathbf{P}\}f\|^{2}\nonumber\\
&\leq \eta \|\partial^{\alpha}_{\beta}\{\mathbf{I}-\mathbf{P}\}f\|^{2}
+C_{\eta}\sum_{|\alpha'|\leq 4-k}\|\partial^{\alpha'}\{\mathbf{I}-\mathbf{P}\}f\|^{2}_{\nu}\nonumber\\
&\quad  +\chi_{2\leq k\leq4} C_{\eta}\sum_{\substack{1\leq |\beta'|\leq k-1 \\ |\alpha'|+|\beta'|\leq 4}}\|\partial^{\alpha'}_{\beta'}\{\mathbf{I}-\mathbf{P}\}f\|^{2}_{\nu},\nonumber\\
R_{3}&\leq \eta \|\partial^{\alpha}_{\beta}\{\mathbf{I}-\mathbf{P}\}f\|^{2}
+C_{\eta}\|\partial^{\alpha}_{x}(u\cdot\nabla_{v}\partial^{\beta}_{v}\{\mathbf{I}-\mathbf{P}\}f)\|^{2}\nonumber\\
&\leq \eta \|\partial^{\alpha}_{\beta}\{\mathbf{I}-\mathbf{P}\}f\|^{2}
+C_{\eta}\|u\|^{2}_{H^{3}}\sum_{\substack{1\leq |\beta'|\leq 4 \\  |\alpha'|+|\beta'|\leq 4 }}\|\partial^{\alpha'}_{\beta'}\{\mathbf{I}-\mathbf{P}\}f\|^{2},\nonumber\\
R_{4}&\leq \eta \|\partial^{\alpha}_{\beta}\{\mathbf{I}-\mathbf{P}\}f\|^{2}
+C_{\eta}\|\partial^{\alpha}_{x}\big(u\cdot\partial^{\beta}_{v}(v\{\mathbf{I}-\mathbf{P}\}f)\big)\|^{2}\nonumber\\
&\leq \eta \|\partial^{\alpha}_{\beta}\{\mathbf{I}-\mathbf{P}\}f\|^{2}
+C_{\eta}\|u\|^{2}_{H^{3}}\sum_{\substack{1\leq |\beta'|\leq 4 \\  |\alpha'|+|\beta'|\leq 4 }}\|\partial^{\alpha'}_{\beta'}\{\mathbf{I}-\mathbf{P}\}f\|^{2}_{\nu}\nonumber\\
&\quad +C_{\eta}\|u\|^{2}_{H^{3}}\sum_{  |\alpha'|\leq 4-k }\|\partial^{\alpha'}\{\mathbf{I}-\mathbf{P}\}f\|^{2},\nonumber\\
R_{5}
&\leq \eta \|\partial^{\alpha}_{\beta}\{\mathbf{I}-\mathbf{P}\}f\|^{2}
+C_{\eta}\|\theta\|_{H^{3}}\sum_{\substack{1\leq |\beta'|\leq 4 \\  |\alpha'|+|\beta'|\leq 4 }}\|\partial^{\alpha'}_{\beta'}\{\mathbf{I}-\mathbf{P}\}f\|^{2}_{\nu}\nonumber\\
&\quad +C_{\eta}\|\theta\|^{2}_{H^{3}}\sum_{  |\alpha'|\leq 4-k+1 }\|\partial^{\alpha'}\{\mathbf{I}-\mathbf{P}\}f\|^{2},\nonumber\\
R_{6}&\leq \eta \|\partial^{\alpha}_{\beta}\{\mathbf{I}-\mathbf{P}\}f\|^{2}+
C_{\eta}\|\partial^{\alpha}_{\beta}\mathbf{P}(v\cdot \nabla_{x}\{\mathbf{I}-\mathbf{P}\}f)\|^{2}\nonumber\\
&\quad +C_{\eta}\|\partial^{\alpha}_{\beta}\mathbf{P}(u\cdot \nabla_{v}\{\mathbf{I}-\mathbf{P}\}f)\|^{2}
+C_{\eta}\|\partial^{\alpha}_{\beta}\mathbf{P}(u\cdot v\{\mathbf{I}-\mathbf{P}\}f)\|^{2}\nonumber\\
&\quad+C_{\eta}\|\partial^{\alpha}_{\beta}\mathbf{P}\big(\theta M^{-\frac{1}{2}}\Delta_{v}(M^{\frac{1}{2}}\{\mathbf{I}-\mathbf{P}\}f)\big)\|^{2}\nonumber\\
&\leq \eta \|\partial^{\alpha}_{\beta}\{\mathbf{I}-\mathbf{P}\}f\|^{2}
+C_{\eta}\sum_{  |\alpha'|\leq 4-k }\|\nabla_{x}\partial^{\alpha'}\{\mathbf{I}-\mathbf{P}\}f\|^{2}\nonumber\\
&\quad+C_{\eta}\|(u,\theta)\|^{2}_{H^{3}}\sum_{  |\alpha'|\leq 4-k }\|\partial^{\alpha'}\{\mathbf{I}-\mathbf{P}\}f\|^{2},\nonumber\\
R_{7}&\leq \eta \|\partial^{\alpha}_{\beta}\{\mathbf{I}-\mathbf{P}\}f\|^{2}
+C_{\eta}\|\partial^{\alpha}_{\beta}\{\mathbf{I}-\mathbf{P}\}(v\cdot \nabla_{x}\mathbf{P}f)\|^{2}\nonumber\\
&\quad +C_{\eta}\|\partial^{\alpha}_{\beta}\{\mathbf{I}-\mathbf{P}\}(u\cdot \nabla_{v}\mathbf{P}f)\|^{2}
+C_{\eta}\|\partial^{\alpha}_{\beta}\{\mathbf{I}-\mathbf{P}\}(u\cdot v \mathbf{P}f)\|^{2}\nonumber\\
&\quad+C_{\eta}\|\partial^{\alpha}_{\beta}\mathbf{I-P}\big(\theta M^{-\frac{1}{2}}\Delta_{v}(M^{\frac{1}{2}}\mathbf{P}f)\big)\|^{2}\nonumber\\
&\leq \eta \|\partial^{\alpha}_{\beta}\{\mathbf{I}-\mathbf{P}\}f\|^{2}
+C_{\eta}\|\nabla (b,\omega)\|^{2}_{H^{4-k}}\nonumber\\
&\quad+C_{\eta}\|(u,\theta)\|^{2}_{H^{4-k+1}}\|\nabla (b,\omega)\|^{2}_{H^{2}}.\nonumber
\end{align}

We estimate the term $ \int\langle -L\partial^{\alpha}_{\beta}\{\mathbf{I}-\mathbf{P}\}f, \partial^{\alpha}_{\beta}\{\mathbf{I}-\mathbf{P}\}f\rangle   \, {\rm d}x$  as
\begin{align}
&\quad \int\langle -L\partial^{\alpha}_{\beta}\{\mathbf{I}-\mathbf{P}\}f, \partial^{\alpha}_{\beta}\{\mathbf{I}-\mathbf{P}\}f\rangle   \, {\rm d}x\nonumber\\
&\geq\lambda_{1}\|\{\mathbf{I}-\mathbf{P}_{0}\}\partial^{\alpha}_{\beta}\{\mathbf{I}-\mathbf{P}\}f\|^{2}_{\nu}\nonumber\\
&\geq\frac{\lambda_{1}}{2}\|\partial^{\alpha}_{\beta}\{\mathbf{I}-\mathbf{P}\}f\|^{2}_{\nu}
-\lambda_{1}\|\mathbf{P}_{0}\partial^{\alpha}_{\beta}\{\mathbf{I}-\mathbf{P}\}f\|^{2}_{\nu}\nonumber\\
&\geq\frac{\lambda_{1}}{2}\|\partial^{\alpha}_{\beta}\{\mathbf{I}-\mathbf{P}\}f\|^{2}_{\nu}
-C\|\partial^{\alpha}\{\mathbf{I}-\mathbf{P}\}f\|^{2}.\nonumber
\end{align}
Substituting all the above estimates into \eqref{v2.27} and selecting $\eta$  small enough, we deduce that \eqref{v2.26} holds.
\end{proof}
\begin{Remark}\label{vr2.1}
  Choosing some suitable constants $C_{k}$ together with the above lemma, we have
\begin{align}
\frac{\rm d}{{\rm d}t}&\sum_{1\leq k\leq4}C_{k}\sum_{\substack{|\beta|=k \\ |\alpha|+|\beta|\leq 4}}\|\partial^{\alpha}_{\beta}\{\mathbf{I}-\mathbf{P}\}f\|^{2}
+\lambda\sum_{\substack{1\leq |\beta|\leq 4 \\ |\alpha|+|\beta|\leq 4}}\|\partial^{\alpha}_{\beta}\{\mathbf{I}-\mathbf{P}\}f\|^{2}_{\nu}\nonumber\\
&\leq C\Big(\|\nabla (b,\omega)\|^{2}_{H^{3}}+\sum_{|\alpha|\leq 4}\|\partial^{\alpha}\{\mathbf{I}-\mathbf{P}\}f\|^{2}_{\nu}\Big)\nonumber\\
&\quad+ C(\|u\|^{2}_{H^{3}}+\|\theta\|^{2}_{H^{3}})\sum_{|\alpha|\leq 4}\|\partial^{\alpha}\{\mathbf{I}-\mathbf{P}\}f\|^{2}_{\nu}\nonumber\\
&\quad+C(\|\theta\|_{H^{3}}+\|u\|^{2}_{H^{3}})
\sum_{\substack{1\leq||\beta|\leq 4  \\ \alpha|+|\beta|\leq 4}}\|\partial^{\alpha}_{\beta}\{\mathbf{I}-\mathbf{P}\}f\|^{2}_{\nu}\nonumber\\
&\quad+C\|(u,\theta)\|^{2}_{H^{4}}\|\nabla (b,\omega)\|^{2}_{H^{3}}.
\label{v2.28}
\end{align}
\end{Remark}

Now, let us  define $\mathcal{E}_{2}(t)$ and $\mathcal{D}_{2}(t)$ as follows
\begin{align}
\mathcal{E}_{2}(t)& :=\sum_{1\leq k\leq4}C_{k}\sum_{\substack{|\beta|=k \\ |\alpha|+|\beta|\leq 4}}\|\partial^{\alpha}_{\beta}\{\mathbf{I}-\mathbf{P}\}f\|^{2},\nonumber\\
\mathcal{D}_{2}(t)& :=\sum_{\substack{1\leq |\beta|\leq 4 \\ |\alpha|+|\beta|\leq 4}}\|\partial^{\alpha}_{\beta}\{\mathbf{I}-\mathbf{P}\}f\|^{2}_{\nu}.\nonumber
\end{align}
Then  the total energy functional $\mathcal{E}(t)$ and the  dissipation rate $\mathcal{D}(t)$ can be defined by
\begin{align}
\mathcal{E}(t)& :=\mathcal{E}_{1}(t)+\tau_{3}\mathcal{E}_{2}(t),\nonumber\\
\mathcal{D}(t)& :=\mathcal{D}_{1}(t)+\tau_{3}\mathcal{D}_{2}(t),\nonumber
\end{align}
where $\tau_{3}> 0$  is very small and will be determined later.

According to the inequalities \eqref{v2.24}, \eqref{v2.28} and \eqref{v2.2}, we have
\begin{align}
\frac{\rm d}{{\rm d}t}\mathcal{E}(t)+\lambda \mathcal{D}(t)\leq C\big(\mathcal{E}^{\frac{1}{2}}(t)+\mathcal{E}^{2}(t)\big)\mathcal{D}(t)
\leq C(\delta+\delta^{2}) \mathcal{D}(t). \label{v2.29}
\end{align}
Thus, as long as $0<\delta < 1$ is sufficiently small, the  integration of \eqref{v2.29} with respect to time gives
\begin{align}
\mathcal{E}(t)+\lambda \int^{t}_{0}\mathcal{D}(s){\rm d}s\leq \mathcal{E}(0) \label{v2.30}
\end{align}
for all $0\leq t<T $.  In addition, \eqref{v2.2} can be proved by choosing
$  
\mathcal{E}(0) \sim \|f_{0}\|^{2}_{H^{4}_{x,v}}+\|(\rho_{0},u_{0})\|^{2}_{H^{4}}  
$  
sufficiently small.

\subsection{Global Existence}
 In this subsection, we will show that there exists a unique global-in-time solution to the problem \eqref{v1.5}-\eqref{v1.8}.
  The proof is based on the uniform energy estimates for the iteration sequence of   approximate solutions.
 The sequence $(f^{n},\rho^{n},u^{n},\theta^{n})_{n=0}^{\infty}$ satisfies the following system:
\begin{align}
\partial_{t}f^{n+1} +v\cdot \nabla_{x}f^{n+1}-\mathcal{L}f ^{n+1}=&
 -u^{n}\cdot\Big(\nabla_{v}f^{n+1}-\frac{v}{2}f^{n+1}-v\sqrt{M}\Big)\nonumber\\
 &+\theta^{n}M^{-\frac{1}{2}}\Delta_{v}\big(M+\sqrt{M}f^{n+1}\big),\label{v2.31}
 \end{align}
 \begin{align}
\partial_{t}\rho^{n+1} +u^{n}\cdot \nabla \rho^{n+1}+(1+\rho^{n})\ \dv u^{n+1}=&0,\label{v2.32}
 \end{align}
 \begin{align}
\partial_{t}u^{n+1} -\frac{1}{1+\rho^{n}}\Delta u^{n+1}=&-u^{n}\cdot \nabla u^{n}-\nabla \theta^{n}
-\frac{1+\theta^{n}}{1+\rho^{n}}\nabla \rho^{n}\nonumber\\
&+\frac{1}{1+\rho^{n}}\big(b^{n}-u^{n}-u^{n}a^{n}\big),\label{v2.33}
 \end{align}
 \begin{align}
\partial_{t}\theta^{n+1} -\frac{1}{1+\rho^{n}}\Delta \theta^{n+1}=&-u^{n}\cdot\nabla\theta^{n}-\theta^{n}\dv u^{n}-\dv u^{n}+\frac{\sqrt{3}}{1+\rho^{n}}(\sqrt{2}\omega^{n}-\sqrt{3}\theta^{n})\nonumber\\
&+\frac{1}{1+\rho^{n}}\big((u^{n}-b^{n})\cdot u^{n}-u^{n}\cdot b^{n}+a^{n}|u^{n}|^{2}-3a^{n}\theta^{n}\big),\label{vx2.33}
\end{align}
where $n=0,1,2,\dots$, and $(f^{0},\rho^{0},u^{0},\theta^{0})=(f_{0},\rho_{0},u_{0},\theta_{0})$ is the starting value of iteration.

We try to find solutions in the following  function space $X(0,T;A)$
\begin{align}
X(0,T;A) :=\left\{\begin{array}{c}
f\in C([0,T],H^{4}(\mathbb{R}^{3}\times\mathbb{R}^{3})),\, M+\sqrt{M}f\geq 0;\\
\rho\in C([0,T],H^{4}(\mathbb{R}^{3}))\bigcap C^{1}([0,T],H^{3}(\mathbb{R}^{3}));\\
u\in C([0,T],H^{4}(\mathbb{R}^{3}))\bigcap C^{1}([0,T],H^{2}(\mathbb{R}^{3}));\\
\theta\in C([0,T],H^{4}(\mathbb{R}^{3}))\bigcap C^{1}([0,T],H^{2}(\mathbb{R}^{3}));\\
\sup_{0\leq t\leq T}\Big\{\|f(t)\|_{H^{4}_{x,v}}+\|(\rho,u,\theta)\|_{H^{4}}\Big\}\leq A;\\
\rho_{1} =\frac{1}{2}(-1+\inf \rho(0,x))>  -1; \\
\inf_{0\leq t\leq T, x\in \mathbb{R}^{3}}\rho(t,x)\geq \rho_{1}.
\end{array}
\right\}   \label{v2.34}
\end{align}

We now state the local existence theorem.

\begin{Theorem}\label{vt2.1}
There exist $A_{0}>  0 $ and $ T^{*}> 0,$ such that if $f_{0}\in H^{4}(\mathbb{R}^{3}\times\mathbb{R}^{3}),
\rho_{0}\in H^{4}(\mathbb{R}^{3}),u_{0}\in H^{4}(\mathbb{R}^{3}),\theta_{0}\in H^{4}(\mathbb{R}^{3})$ with $F_{0}=M+\sqrt{M}f_{0}\geq 0$
and $\mathcal{E}(0)\leq \frac{A_{0}}{2}$, then for each $n\geq 1$, $(f^{n},\rho^{n},u^{n},\theta^{n})$ is well-defined with
\begin{align}
(f^{n},\rho^{n},u^{n},\theta^{n})\in X(0,T^{*};A_{0}). \label{v2.35}
\end{align}
Further, we obtain:
\begin{enumerate}[\rm{(}1\rm{)}]
  \item $(f^{n},\rho^{n},u^{n},\theta^{n})_{n\geq 0}$ is a Cauchy sequence in the Banach space $C([0,T^{*}];H^{3}(\mathbb{R}^{3}\times\mathbb{R}^{3}))$,
    \item let $(f,\rho,u,\theta)$ be the  limit function, and then $(f,\rho,u,\theta)\in X(0,T^{*};A_{0})$,

  \item $(f,\rho,u,\theta)$ satisfies  the system
\eqref{v1.5}-\eqref{v1.8},

\item  $(f,\rho,u,\theta)$ is  the  unique solution of \eqref{v1.5}-\eqref{v1.8} in $X(0,T^{*};A_{0})$.

\end{enumerate}
\end{Theorem}


\begin{proof} Let $T^*>0$ be a constant which will be fixed later.
For simplicity, without loss of generality  $(f^{n},\rho^{n},u^{n},\theta^{n})$ are assumed to be smooth enough, if not,
we can consider the  regularized iterative system as
\begin{align}
\partial_{t}f^{n+1,\varepsilon} +v\cdot \nabla_{x}f^{n+1,\varepsilon}-\mathcal{L}f ^{n+1,\varepsilon}=&
 -u^{n,\varepsilon}\cdot\Big(\nabla_{v}f^{n+1,\varepsilon}-\frac{v}{2}f^{n+1,\varepsilon}-v\sqrt{M}\Big)\nonumber\\
 &+\theta^{n,\varepsilon}M^{-\frac{1}{2}}\Delta_{v}\big(M+\sqrt{M}f^{n+1,\varepsilon}\big),\nonumber
  \end{align}
 \begin{align}
\partial_{t}\rho^{n+1,\varepsilon}&+u^{n,\varepsilon}\cdot \nabla \rho^{n+1,\varepsilon}+(1+\rho^{n,\varepsilon})\ \dv u^{n+1,\varepsilon}=0,\nonumber
 \end{align}
 \begin{align}
\partial_{t}u^{n+1,\varepsilon} -\frac{1}{1+\rho^{n,\varepsilon}}\Delta u^{n+1,\varepsilon}=&-u^{n,\varepsilon}\cdot \nabla u^{n,\varepsilon}-\nabla \theta^{n,\varepsilon}
-\frac{1+\theta^{n,\varepsilon}}{1+\rho^{n,\varepsilon}}\nabla \rho^{n,\varepsilon}\nonumber\\
&+\frac{1}{1+\rho^{n,\varepsilon}}\big(b^{n,\varepsilon}-u^{n,\varepsilon}-u^{n,\varepsilon}a^{n,\varepsilon}\big),\nonumber
 \end{align}
 \begin{align}
\partial_{t}\theta^{n+1,\varepsilon} -\frac{1}{1+\rho^{n,\varepsilon}}\Delta \theta^{n+1,\varepsilon}=&-u^{n,\varepsilon}\cdot\nabla\theta^{n,\varepsilon}
-\theta^{n,\varepsilon}\dv u^{n,\varepsilon}-\dv u^{n,\varepsilon}\nonumber\\
&+\frac{1}{1+\rho^{n,\varepsilon}}\big((u^{n,\varepsilon}-b^{n,\varepsilon})\cdot u^{n,\varepsilon}-u^{n,\varepsilon}\cdot b^{n,\varepsilon}+a^{n,\varepsilon}|u^{n,\varepsilon}|^{2}\big)\nonumber\\
&+\frac{\sqrt{3}}{1+\rho^{n,\varepsilon}}\big(\sqrt{2}\omega^{n,\varepsilon}-\sqrt{3}\theta^{n,\varepsilon}-\sqrt{3}a^{n,\varepsilon}\theta^{n,\varepsilon}\big),\nonumber
 \end{align}
 \begin{align}
 (f^{n+1,\varepsilon},\rho^{n+1,\varepsilon},u^{n+1,\varepsilon},\theta^{n+1,\varepsilon})(0)=(f^{\varepsilon}_{0},\rho^{\varepsilon}_{0},u^{\varepsilon}_{0},\theta^{\varepsilon}_{0}),\nonumber
\end{align}
for any $\varepsilon>0  $ with $(f^{\varepsilon}_{0},\rho^{\varepsilon}_{0},u^{\varepsilon}_{0},\theta^{\varepsilon}_{0})$  a smooth approximation of $(f_{0},\rho_{0},u_{0},\theta_{0})$ and pass to the limit
by letting $\varepsilon\rightarrow 0.$

Applying $\partial^{\alpha}_{x}$ with $|\alpha|\leq 4$ to the equation \eqref{v2.31}, multiplying the result by $\partial^{\alpha}_{x}f^{n+1}$
and then taking integration over $\mathbb{R}^3$, one has
\begin{align}
& \frac{1}{2}\frac{\rm d}{{\rm d}t}\|\partial^{\alpha}f^{n+1}\|^{2}+\int\langle -\mathcal{L}\partial^{\alpha}f^{n+1},\partial^{\alpha}f^{n+1}\rangle   \, {\rm d}x \nonumber\\
&=-\iint\partial^{\alpha}\Big\{u^{n}M^{-\frac{1}{2}}\nabla_{v}(M+\sqrt{M}f^{n+1})\Big\}\partial^{\alpha}f^{n+1}\, {\rm d}x  {\rm d}v\nonumber\\
&\quad+\iint\partial^{\alpha}\Big\{\theta^{n}M^{-\frac{1}{2}}\Delta_{v}(M+\sqrt{M}f^{n+1})\Big\}\partial^{\alpha}f^{n+1}\, {\rm d}x  {\rm d}v\nonumber\\
&\leq C\|(u^{n},\theta^{n})\|_{H^{4}}\|f^{n+1}\|_{L^{2}_{v}(H^{4})}\nonumber\\
&\quad+C\|u^{n}\|_{H^{4}}\|f^{n+1}\|_{L^{2}_{v}(H^{4})}\|\partial^{\alpha}f^{n+1}\|_{\nu}\nonumber\\
&\quad+C\|\theta^{n}\|_{H^{4}}\Big(\sum_{|\alpha'|\leq 4}\|\partial^{\alpha'f^{n+1}}\|_{\nu}\Big)\|\partial^{\alpha}f^{n+1}\|_{\nu}.\label{v2333}
\end{align}
Notice that   $$\int\langle -\mathcal{L}\partial^{\alpha}f^{n+1},\partial^{\alpha}f^{n+1}\rangle   \, {\rm d}x
\geq \lambda \|\{\mathbf{I}-\mathbf{P}_{0}\}\partial^{\alpha}f^{n+1}\|^{2}_{\nu}.$$
 By adding $\|\mathbf{P}_{0}\partial^{\alpha}f^{n+1}\|^{2}_{\nu}$ to
 the   inequality \eqref{v2333}, we  get  the sum on  $|\alpha|\leq 4$,
\begin{align}
\frac{1}{2}\frac{\rm d}{{\rm d}t}&\sum_{|\alpha|\leq 4}\|\partial^{\alpha}f^{n+1}\|^{2}+\lambda\sum_{|\alpha|\leq 4}\|\partial^{\alpha}f^{n+1}\|^{2}_{\nu}\nonumber\\
&\leq C\|f^{n+1}\|^{2}_{L^{2}_{v}(H^{4})}+C\|(u^{n},\theta^{n})\|_{H^{4}}\|f^{n+1}\|_{L^{2}_{v}(H^{4})}\nonumber\\
&\quad +C\|u^{n}\|_{H^{4}}\|f^{n+1}\|^{2}_{L^{2}_{v}(H^{4})}
 +C\|(u^{n},\theta^{n})\|_{H^{4}}\sum_{|\alpha|\leq 4}\|\partial^{\alpha}f^{n+1}\|^{2}_{\nu}.\nonumber
\end{align}
Similar argument gives, for any $0\leq t\leq T\leq T^*$,
\begin{align}
\frac{1}{2}\frac{\rm d}{{\rm d}t}&\|f^{n+1}\|^{2}_{H^{4}_{x,v}}+\lambda \sum_{|\alpha|+|\beta|\leq 4}\|\partial^{\alpha}_{\beta}f^{n+1}\|^{2}_{\nu}\nonumber\\
&\leq C\|f^{n+1}\|^{2}_{H^{4}_{x,v}}+C\|(u^{n},\theta^{n})\|^{2}_{H^{4}}+C\|u^{n}\|_{H^{4}}\|f^{n+1}\|^{2}_{H^{4}}\nonumber\\
&\quad+C\|(u^{n},\theta^{n})\|_{H^{4}}\sum_{|\alpha|+|\beta|\leq 4}\|\partial^{\alpha}_{\beta}f^{n+1}\|^{2}_{\nu}
.\label{v2.36}
\end{align}

Next, from \cite{MN},  
 the system \eqref{v2.32}-\eqref{v2.33} has a unique solution $(\rho^{n+1}, u^{n+1}, \theta^{n+1})$ such that
$\rho^{n+1}\geq \rho_{1}$, and
$$\rho^{n+1}\in C([0,T],H^{4}(\mathbb{R}^{3}))\cap C^{1}([0,T],H^{3}(\mathbb{R}^{3})),$$
$$u^{n+1},\theta^{n+1}\in C([0,T],H^{4}(\mathbb{R}^{3}))\cap C^{1}([0,T],H^{2}(\mathbb{R}^{3})).$$

Now we  estimate $\frac{\rm d}{{\rm d}t}\|(\rho^{n+1},u^{n+1},\theta^{n+1})\|^{2}_{H^{4}}$.
Applying $\partial^{\alpha}\,(|\alpha|\leq 4)$ to the system \eqref{v2.32}- \eqref{vx2.33}, multiplying the results by
$\partial^{\alpha}\rho^{n+1}$, $ \partial^{\alpha}u^{n+1},\partial^{\alpha}\theta^{n+1}$ respectively, and then taking integration and  summation, we get
\begin{align}
&\frac{1}{2}\frac{\rm d}{{\rm d}t}\Big\|\Big(\rho^{n+1},u^{n+1},\theta^{n+1}\Big)\Big\|^{2}_{H^{4}}
 +\lambda\sum_{|\alpha|\leq 4}\int\Big(|\nabla \partial^{\alpha}u^{n+1}|^{2}+|\nabla \partial^{\alpha}u^{n+1}|^{2}\Big){\rm d}x\nonumber\\
&\leq\, C(1+\|\rho^{n}\|^{2}_{H^{4}}+\|u^{n}\|_{H^{4}})\|\rho^{n+1}\|^{2}_{H^{4}}
+C(1+\|u\|^{2}_{H^{4}}+\|\rho^{n}\|^{8}_{H^{4}})\|(u^{n+1},\theta^{n+1})\|^{2}_{H^{4}}\nonumber\\
&\quad+C(1+\|\rho^{n}\|^{6}_{H^{4}})(1+\|\theta^{n}\|^{2}_{H^{4}})\|\rho^{n}\|^{2}_{H^{4}}
+C(1+\|f^{n}\|^{2}_{H^{4}}+\|u^{n}\|^{2}_{H^{4}})\|u^{n}\|^{2}_{H^{4}}\nonumber\\
&\quad+C(1+\|\rho^{n}\|^{6}_{H^{4}}+\|u^{n}\|^{2}_{H^{4}})\|\theta^{n}\|^{2}_{H^{4}}
+C(1+\|\theta^{n}\|^{2}_{H^{4}}+\|u^{n}\|^{2}_{H^{4}})\|f^{n}\|^{2}_{H^{4}}.\label{v2.37}
\end{align}
Adding up   \eqref{v2.36} and \eqref{v2.37} gives
\begin{align}
&\frac{1}{2}\frac{\rm d}{{\rm d}t}\Big(\|f^{n+1}\|^{2}_{H^{4}_{x,v}}+\|\rho^{n+1}\|^{2}_{H^{4}}+\|u^{n+1}\|^{2}_{H^{4}}+\|\theta^{n+1}\|^{2}_{H^{4}}\Big)\nonumber\\
&\quad+\lambda
\sum_{|\alpha|+|\beta|\leq 4}\|\partial^{\alpha}_{\beta}f^{n+1}\|^{2}_{\nu}+\lambda\sum_{|\alpha|\leq 4}\Big(\|\nabla\partial^{\alpha}u^{n+1}\|^{2}+\|\nabla\partial^{\alpha}\theta^{n+1}\|^{2}\Big)\nonumber\\
&\leq C(1+\|u^{n}\|^{2}_{H^{4}})\|f^{n+1}\|^{2}_{H^{4}_{x,v}}+C(1+\|\rho^{n}\|^{2}_{H^{4}}+\|u^{n}\|_{H^{4}})\|\rho^{n+1}\|^{2}_{H^{4}}\nonumber\\
&\quad+C(1+\|u\|^{2}_{H^{4}}+\|\rho^{n}\|^{8}_{H^{4}})\|(u^{n+1},\theta^{n+1})\|^{2}_{H^{4}}
\nonumber\\
&\quad+C(1+\|\rho^{n}\|^{6}_{H^{4}})(1+\|\theta^{n}\|^{2}_{H^{4}})\|\rho^{n}\|^{2}_{H^{4}}
+C(1+\|\rho^{n}\|^{6}_{H^{4}}+\|u^{n}\|^{2}_{H^{4}})\|\theta^{n}\|^{2}_{H^{4}}\nonumber\\
&\quad+C(1+\|\theta^{n}\|^{2}_{H^{4}}+\|u^{n}\|^{2}_{H^{4}})\|f^{n}\|^{2}_{H^{4}}+C(1+\|f^{n}\|^{2}_{H^{4}}+\|u^{n}\|^{2}_{H^{4}})\|u^{n}\|^{2}_{H^{4}}\nonumber\\
&\quad+C\|(u^{n},\theta^{n})\|_{H^{4}}\sum_{|\alpha|+|\beta|\leq 4}\|\partial^{\alpha}_{\beta}f^{n+1}\|^{2}_{\nu}.\label{v2.38}
\end{align}
Using induction, we may assume $A_{n}(T)\leq A_{0}$ and $A_{n}(0)\leq \frac{A_{0}}{2}$ for some $A_0>0$  with
$$A_{n}(T) :=\sup_{0\leq t\leq T}\Big\{\|\rho^{n}(t)\|^{2}_{H^{4}}+\|u^{n}(t)\|^{2}_{H^{4}}+\|\theta^{n}(t)\|^{2}_{H^{4}}+\|f^{n}(t)\|^{2}_{H^{4}_{x,v}}\Big\}.$$
Integrating \eqref{v2.38} over $[0,T]$ yields
\begin{align}
& A_{n+1}(T)+\lambda \int_{0}^{T}\Big\{\sum_{|\alpha|\leq 4}\Big(\|\nabla \partial^{\alpha}u^{n+1}\|^{2}+\|\nabla \partial^{\alpha}\theta^{n+1}\|^{2}\Big)
+\sum_{|\alpha|+|\beta|\leq 4}\|\partial^{\alpha}_{\beta}f^{n+1}\|^{2}_{\nu}\Big\}{\rm d} t\nonumber\\
&\leq A_{n+1}(0)+C(1+A^{\frac{1}{2}}_{n}(T)+A^{2}_{n}(T))A_{n+1}(T)T+CA_{n}(T)(1+A^{4}_{n}(T))T\nonumber\\
&\quad+CA^{\frac{1}{2}}_{n}(T)\int_{0}^{T}\sum_{|\alpha|+|\beta|\leq 4}\|\partial^{\alpha}_{\beta}f^{n+1}\|^{2}_{\nu}\,{\rm d} t\nonumber\\
&\leq\frac{A_{0}}{2}+C(1+A^{4}_{0})TA_{n+1}(T)+C(A_{0}+A^{5}_{0})T\nonumber\\
&\quad+CA^{\frac{1}{2}}_{0}\sum_{|\alpha|+|\beta|\leq 4}\int_{0}^{T}\|\partial^{\alpha}_{\beta}f^{n+1}\|^{2}_{\nu}\,{\rm d} t.\label{v2.39}
\end{align}
It is easy to obtain that for $t\leq T^{*}$,
$$\big(1-C(1+A^{4}_{0})T\big)A_{n+1}(T)\leq \frac{A_{0}}{2}+C(A_{0}+A^{5}_{0})T.$$
Choosing $T^*$ satisfying  $T^{*}\leq \frac{A_{0}}{2}$, and $A_{0}$   small enough, one gains
$$A_{n+1}\leq A_{0}.$$

By the equation \eqref{v2.31},  $F^{n+1}=M+\sqrt{M}f^{n+1}$ satisfies the following equation,
$$\partial_{t}F^{n+1}+v\cdot\nabla_{x}F^{n+1}+(u-v)\cdot \nabla_{v}F^{n+1}-(1+\theta^{n})\Delta_{v}F^{n+1}-3F^{n+1}=0,$$
on the basis of  the maximum principle, one achieves
$$F^{n+1}\geq 0.$$

Now we show that $\|f^{n+1}\|^{2}_{H^{4}_{x,v}}$ is continuous over $0\leq t\leq T^{*}$. \\
Actually, similar to the proof of  \eqref{v2.36}, the following inequality holds
\begin{align}
 &  \!\!\!\!\!\!\!\!\!\!\Big|\|f^{n+1}(t)\|^{2}_{H^{4}_{x,v}}-\|f^{n+1}(s)\|^{2}_{H^{4}_{x,v}}\Big|\nonumber\\
 =\,&\Big|\int_{s}^{t}\frac{\mathrm{d}}{\mathrm{d}\eta}\|f^{n+1}(\eta)\|^{2}_{H^{4}_{x,v}}\mathrm{d}\eta\Big|\nonumber\\
\leq\, & C A^{\frac{1}{2}}_{0}\sum_{|\alpha|+|\beta|\leq 4} \int_{s}^{t}\|\partial^{\alpha}_{\beta}f^{n+1}\|^{2}_{\nu}\mathrm{d}\eta
 +C(A_{0}+A^{\frac{3}{2}}_{0})|t-s|. \label{v2.40}
 \end{align}
Moreover, $\parallel\partial^{\alpha}_{\beta}f^{n+1}\parallel^{2}_{\nu}$ is integrable over $[0,T^{*}]$.
Thus, \eqref{v2.35} is true for $n+1$, namely it follows that \eqref{v2.35} holds  for any $n\geq 0$.

 By straightforward calculation, $f^{n+1}-f^{n},\rho^{n+1}-\rho^{n},u^{n+1}-u^{n},\theta^{n+1}-\theta^{n}$ satisfies the following system:
\begin{align}
&\partial_{t}(f^{n+1}-f^{n})+v\cdot \nabla_{x}(f^{n+1}-f^{n})-\mathcal{L}(f^{n+1}-f^{n})\nonumber\\
&\,=-u^{n}\cdot\big(\nabla_{v}-\frac{v}{2}\big)\big(f^{n+1}-f^{n}\big)-\big(u^{n}-u^{n-1}\big)\cdot\big(\nabla_{v}-\frac{v}{2}\big)f^{n}+M^{\frac{1}{2}}v\cdot\big(u^{n}-u^{n-1}\big)\nonumber\\
&\,\quad +\theta^{n}M^{-\frac{1}{2}}\Delta_{v}\Big(M^{\frac{1}{2}}\big(f^{n+1}-f^{n}\big)\Big)+M^{-\frac{1}{2}}\Delta_{v}M
\big(\theta^{n}-\theta^{n-1}\big),\nonumber\\
&\partial_{t}(\rho^{n+1}-\rho^{n})+u^{n}\cdot \nabla (\rho^{n+1}-\rho^{n})\nonumber\\
& \, =-(u^{n}-u^{n-1})\cdot \nabla \rho^{n}-(1+\rho^{n})\ \dv(u^{n+1}-u^{n})-(\rho^{n}-\rho^{n-1})\ \dv u^{n},\nonumber\\
&\partial_{t}(u^{n+1}-u^{n})-\frac{1}{1+\rho^{n}}\Delta(u^{n+1}-u^{n})\nonumber\\
&\,=\Big(\frac{1}{1+\rho^{n}}-\frac{1}{1+\rho^{n-1}}\Big)\Big(\Delta u^{n}+b^{n}-u^{n}-u^{n}a^{n}\Big)\nonumber\\
&\,\quad -\Big((u^{n}-u^{n-1})\cdot \nabla u^{n}+u^{n-1}\cdot \nabla(u^{n}-u^{n-1})\Big)-\nabla\Big(\theta^{n}-\theta^{n-1}\Big)
\nonumber\\
&\,\quad +\Big(\frac{1+\theta^{n}}{1+\rho^{n}}\nabla \big(\rho^{n}-\rho^{n-1}\big)+\big(\frac{1+\theta^{n}}{1+\rho^{n}}-\frac{1+\theta^{n-1}}{1+\rho^{n-1}}\big)\nabla \rho^{n-1}\Big)\nonumber\\
&\quad +\frac{1}{1+\rho^{n}}\Big(\big(b^{n}-b^{n-1}\big)-\big(u^{n}-u^{n-1}\big)-\big(u^{n}a^{n}-u^{n-1}a^{n-1}\big)\Big),\nonumber\\
&\partial_{t}(\theta^{n+1}-\theta^{n})-\frac{1}{1+\rho^{n}}\Delta(\theta^{n+1}-\theta^{n})\nonumber\\
&\,=-\Big(\big(u^{n}-u^{n-1}\big)\cdot \nabla \theta^{n}+u^{n-1}\cdot \nabla\big(\theta^{n}-\theta^{n-1}\big)\Big)+\dv \big(u^{n}-u^{n-1}\big)\nonumber\\
&\,\quad -\big(\theta^{n}-\theta^{n-1}\big)\dv u^{n}-\dv\big (u^{n}-u^{n-1}\big) \theta^{n}\nonumber\\
&\,\quad+\Big(\frac{1}{1+\rho^{n}}-\frac{1}{1+\rho^{n-1}}\Big)\Big(\Delta \theta^{n}+\sqrt{6}\omega^{n}-3\theta^{n}-3\theta^{n}a^{n}\Big)
\nonumber\\
&\,\quad+\Big(\frac{1}{1+\rho^{n}}-\frac{1}{1+\rho^{n-1}}\Big) \Big(|u^{n}|^{2}-2u^{n}\cdot b^{n}+a^{n}|u^{n}|^{2}\Big)\nonumber\\
&\,\quad +\frac{1}{1+\rho^{n}}\Big(|u^{n}|^{2}-|u^{n-1}|^{2}-2u^{n}\cdot b^{n}+2u^{n-1}\cdot b^{n-1}+a^{n}|u^{n}|^{2}-a^{n-1}|u^{n-1}|^{2}\Big)\nonumber\\
&\,\quad+\frac{1}{1+\rho^{n}}\Big(\sqrt{6}\big(\omega^{n}-\omega^{n-1}\big)-3\big(\theta^{n}-\theta^{n-1}\big)-3\big(a^{n}\theta^{n}-a^{n-1}\theta^{n-1}\big)\Big).\nonumber
\end{align}
Similarly to  \eqref{v2.38}, we obtain
\begin{align}
 &\frac{1}{2}\frac{\rm d}{{\rm d}t}\Big\{\|f^{n+1}-f^{n}\|^{2}_{H^{3}_{x,v}}+\|\rho^{n+1}-\rho^{n}\|^{2}_{H^{3}}+\|u^{n+1}-u^{n}\|^{2}_{H^{3}}+\|\theta^{n+1}-\theta^{n}\|^{2}_{H^{3}}\Big\}\nonumber\\
&\,\quad+\lambda\sum_{|\alpha|+|\beta|\leq 3}\|\partial^{\alpha}_{\beta}(f^{n+1}-f^{n})\|^{2}_{\nu}+\lambda\|\nabla(u^{n+1}-u^{n}),\nabla(\theta^{n+1}-\theta^{n})\|^{2}_{H^{3}}\nonumber\\
&\,\leq\mathcal{ C }
 \Big(\|f^{n+1}-f^{n}\|^{2}_{H^{3}_{x,v}}+\|\rho^{n+1}-\rho^{n}\|^{2}_{H^{3}}+\|u^{n+1}-u^{n}\|^{2}_{H^{3}}+\|\theta^{n+1}-\theta^{n}\|^{2}_{H^{3}}\Big)\nonumber\\
&\,\quad+\mathcal{ C }
 \Big(\|f^{n}-f^{n-1}\|^{2}_{H^{3}_{x,v}}+\|\rho^{n}-\rho^{n-1}\|^{2}_{H^{3}}+\|u^{n}-u^{n-1}\|^{2}_{H^{3}}+\|\theta^{n}-\theta^{n-1}\|^{2}_{H^{3}}\Big),\nonumber
\end{align}
where $\mathcal{ C }>0$ is a   constant  that depends on $\|(\rho^{n},\rho^{n-1})\|_{H^{4}}$, $\|(u^{n},u^{n-1})\|_{H^{4}}$, $\|(\theta^{n},\theta^{n-1})\|_{H^{4}}$,
$\|(f^{n},f^{n-1})\|_{H^{4}_{x,v}}$ and $\sum_{|\alpha|+|\beta|\leq 4}\|\partial^{\alpha}_{\beta}f^{n}\|^{2}_{\nu}$.
Due to $\mathcal{E}(0)$, $T^{*}$ and $A_{0}$  small enough, according to \eqref{v2.39}, we know that
\begin{align}
\sup_{n}\int_{0}^{T^{*}}\sum_{|\alpha|+|\beta|\leq 4}\|\partial^{\alpha}_{\beta}f^{n}\|^{2}_{\nu}\,\mathrm{d}s \nonumber
\end{align}
is also small enough. Therefore, there exists  $\delta \in (0, 1)$ such that
\begin{align}
&\sup_{0< t\leq T^{*}}\Big\{\|f^{n+1}-f^{n}\|_{H^{3}}+\|\rho^{n+1}-\rho^{n}\|_{H^{3}}+\|u^{n+1}-u^{n}\|_{H^{3}}+\|\theta^{n+1}-\theta^{n}\|_{H^{3}}\Big\}  \nonumber\\
&\quad\leq\delta\sup_{0< t\leq T^{*}}\Big\{\|f^{n}-f^{n-1}\|_{H^{3}}+\|\rho^{n}-\rho^{n-1}\|_{H^{3}}\nonumber\\
&\qquad\qquad+\|u^{n}-u^{n-1}\|_{H^{3}}+\|\theta^{n}-u^{n-1}\|_{H^{3}}\Big\}.\label{v2.41}
\end{align}
According to \eqref{v2.41}, we conclude that $(f^{n},\rho^{n},u^{n},\theta^{n})_{n\geq 0}$ is a Cauchy sequence in the Banach space
$C\big([0,T^{*}],H^{3}(\mathbb{R}^{3}\times\mathbb{R}^{3})\big)\times\Big(C\big([0,T^{*}],H^{3}(\mathbb{R}^{3})\big)\Big)^{3}$, so we denote
the  limit function of this sequence by $(f,\rho,u,\theta)$ and then $(f,\rho,u,\theta)$ satisfies the system
 \eqref{v1.5}-\eqref{v1.8} by letting $n\rightarrow \infty$.
The fact that $F^{n}(t,x,v)\geq 0$  and the Sobolev embedding theorem yields that
 $$F(t,x,v)\geq 0, \quad  \sup_{0\leq t\leq T^{*}}\|f(t)\|_{H^{4}_{x,v}}\leq A_{0}.$$
 Similar argument to \eqref{v2.40} yields that $f\in C\big([0,T^{*}],H^{3}(\mathbb{R}^{3}\times\mathbb{R}^{3})\big).$
Namely, we obtain that $(f,\rho,u,\theta)\in X(0,T^{*};A_{0})$.

 Finally, let $(\bar{f},\bar{\rho},\bar{u},\tilde{\theta})\in X(0,T^{*},A_{0})$  be another solution to the Cauchy problem \eqref{v1.5}-\eqref{v1.8}.
 Using similar process  to \eqref{v2.41},it follows that
 \begin{align}
 &\sup_{0< t\leq T^{*}}\Big\{\|f-\bar{f}\|_{H^{3}}+\|\rho-\bar{\rho}\|_{H^{3}}+\|u-\bar{u}\|_{H^{3}}+\|u-\bar{u}\|_{H^{3}}\Big\}  \nonumber\\
&\quad\leq\delta\sup_{0< t\leq T^{*}}\Big\{\|f-\bar{f}\|_{H^{3}}+\|\rho-\bar{\rho}\|_{H^{3}}+\|u-\bar{u}\|_{H^{3}}+\|\theta-\tilde{\theta}\|_{H^{3}}\Big\},\nonumber
\end{align}
for $0<\delta < 1$. So it is easy to conclude that $f\equiv \bar{f}, \rho\equiv \bar{\rho}, u\equiv\bar{u},\theta\equiv\tilde{\theta}$ ,i.e uniqueness holds.
\end{proof}

\begin{proof}[Proof of Theorem \ref{vt1.1}]
By $\mathcal{E}(0) \sim \|f_{0}\|^{2}_{H^{4}_{x,v}}+\|(\rho_{0},u_{0},\theta_{0})\|^{2}_{H^{4}}$,   there exists $\varepsilon_{0}$,
such that if
 $\|f_{0}\|_{H^{4}_{x,v}}+\|(\rho_{0},u_{0},\theta_{0})\|_{H^{4}}  <\varepsilon_{0}$, one gets $\mathcal{E}(0)\leq \frac{A_{0}}{2}$.
 Next, due to Theorem \ref{vt2.1},   the uniform priori estimate  \eqref{v2.30} holds for the local solution.
 Finally,  the standard bootstrap  arguments, as in \cite{DFT,YG,MN}
 eventually  conclude that the global existence and uniqueness part in Theorem \ref{vt1.1} hold.
\end{proof}


\bigskip

\section{Large time behavior}
In this section, our main concern is  the optimal time-decay rates of global solutions to the problem \eqref{v1.5}-\eqref{v1.8}.
First, we shall study the time-decay of the solution to the linearized Cauchy equations with a nonhomogeneous source, with the help of the Fourier analysis, we can obtain the algebraic decay when time tends to infinity, that is Theorem \ref{vt3.1} holds. Then, decomposing the nonlinear terms subtly and  choosing the appropriate functions  as  the nonhomogeneous source, together with Theorem \ref{vt3.1} and the energy-spectrum method  developed in \cite{DUYZ}, we finally conclude the time-decay rate \eqref{v1.11}.
To this end, we assume that all conditions in Theorem \ref{vt1.1} hold, and let $(f,\rho,u,\theta)$ be the solution to the system \eqref{v1.5}-\eqref{v1.8}.

We first  consider the linearized Cauchy equations with a nonhomogeneous source, namely,
\begin{equation}\label{v3.1}
\left\{\begin{aligned}
 &\partial_{t}f+v\cdot\nabla_{x}f-u\cdot v M^{\frac{1}{2}}-\theta(|v|^{2}-3)M^{\frac{1}{2}}=\mathcal{L}f+S_{f},\\
&\partial_{t}\rho+\dv u=0,\\
&\partial_{t}u-\Delta u+\nabla \theta+\nabla \rho+u-b=0,\\
&\partial_{t}\theta-\Delta \theta+\dv u+\sqrt{3}(\sqrt{3}\theta-\sqrt{2}\omega)=0.
\end{aligned}  \right.
\end{equation}
Here, $S_{f}$ has the following form
$$S_{f}=\dv _{v}G-\frac{1}{2}v\cdot G+\varphi,$$
with $G=(G_{i}),G_{i}=G_{i}(t,x,v)\in \mathbb{R}, 1\leq i\leq3,$ and $\varphi=\varphi(t,x,v)\in \mathbb{R}$, meanwhile we suppose that
\begin{align}
\mathbf{P}_{0}G_{i}=0, \quad \mathbf{P}_{1}G_{i}=0,\quad \mathbf{P}\varphi=0,\label{vx3.1}
\end{align}
for all $t\geq 0,x\in \mathbb{R}^{3}$.

For the linearized problem \eqref{v3.1}, we easily conclude that it is well-posed in $L^{2}$, i.e.,  the following lemma is true.
\begin{Lemma}\label{vl3.1}
There is a well-defined linear semigroup $E_{t}:\, L^{2}\rightarrow L^{2},\, t\geq 0,$ such that for any given $(f_{0},\rho_{0},u_{0},\theta_{0})\in L^{2}$,
  $E_{t}(f_{0},\rho_{0},u_{0},\theta_{0})$ is the unique distributional solution to \eqref{v3.1} with $S_{f}=0$. Further, for any  $(f_{0},\rho_{0},u_{0},\theta_{0})\in L^{2}$, then  there is a unique distributional solution to \eqref{v3.1} such that
\begin{align}
\big(f(t),\rho(t),u(t),\theta(t)\big)=E_{t}(f_{0},\rho_{0},u_{0},\theta_{0})+\int_{0}^{t}E_{t-\tau}\big(S_{f}(\tau),0,0,0\big)\, {\rm d} \tau .\label{v3.2}
\end{align}
\end{Lemma}
\begin{proof}
Similar argument as  the local existence theorem can give the well-posedness part,
thus we do not repeat here, and the variation of constants formula \eqref{v3.2} is again true by a direct computation.
\end{proof}

We first quote two lemmas of \cite{CDM} for later proofs.
\begin{Lemma}[\cite{CDM}]\label{vl3.2}
Given any $0 < \beta_{1}\neq 1$ and $\beta_{2}>1,$
$$\int_{0}^{t}(1+t-s)^{-\beta_{1}}(1+s)^{-\beta_{2}}\, {\rm d}s\leq C(1+t)^{-\min\{\beta_{1},\beta_{2}\}}$$
for all $t\geq 0.$
\end{Lemma}

\begin{Lemma}[\cite{CDM}]\label{vl3.3}
Let $\gamma> 1$ and $g_{1},g_{2}\in C^{0}(\mathbb{R}_{+},\mathbb{R}_{+})$ with $g_{1}(0)=0.$ For $A\in \mathbb{R}_{+}$, define
$\mathcal{\beta}_{A}:=\{y\in C^{0}(\mathbb{R}_{+},\mathbb{R}_{+})| \; y\leq A+g_{1}(A)y+g_{2}(A)y^{\gamma},y(0)\leq A\}$. Then, there exists a constant
$A_{0}\in (0,\min\{A_{1},A_{2}\})$ such that for any $0< A<A_{0}$,
$$y\in \mathcal{\beta}_{A}\Longrightarrow \sup_{t\geq 0}y(t) \leq 2A.$$
\end{Lemma}

\begin{Theorem}\label{vt3.1}
Let $1\leq q\leq 2$ and $(f_{0},\rho_{0},u_{0},\theta_{0})\in L^{2}$. For any $\alpha,\alpha'$ with $\alpha'\leq \alpha$ and $m=|\alpha-\alpha'|$,
\begin{align}
\|\partial^{\alpha}E_{t}(f_{0},\rho_{0},u_{0},\theta_{0})\|_{L^{2}}\leq\,& C(1+t)^{-\sigma_{q,m}}\big(\|\partial^{\alpha'}(f_{0},\rho_{0},u_{0},\theta_{0})\|_{\mathcal{Z}_{q}}\nonumber\\
&+\|\partial^{\alpha}(f_{0},\rho_{0},u_{0},\theta_{0})\|_{L^{2}}\big), \label{v3.3}
\end{align}
and
\begin{align}
\Big\|\partial^{\alpha}&\int_{0}^{t}E_{t-\tau}\big(S_{f},0,0,0\big)\, {\rm d} \tau \Big\|^{2}_{L^{2}}
\leq C\int_{0}^{t}(1+t-\tau)^{-2\sigma_{q,m}}\nonumber\\
&\times \big(\|\partial^{\alpha'}(G(\tau),\nu^{-\frac{1}{2}}\varphi(\tau))\|^{2}_{Z_{q}}
+\|\partial^{\alpha}(G(\tau),\nu^{-\frac{1}{2}}\varphi(\tau))\|^{2}_{L^{2}}\big)\, {\rm d} \tau , \label{v3.4}
\end{align}
hold for $t\geq 0$, where C is a positive constant depending only on $m,q$ and
$$\sigma_{q,m}=\frac{3}{2}\Big(\frac{1}{q}-\frac{1}{2}\Big)+\frac{m}{2}.$$
\end{Theorem}
\begin{proof}
Applying the Fourier transform to \eqref{v3.1} in $x$, we obtain
\begin{equation}\label{v3.5}
\left\{\begin{aligned}
 &\partial_{t}\hat{f}+i v\cdot \xi \hat{f}-\hat{u}\cdot v M^{\frac{1}{2}}-\hat{\theta}(|v|^{2}-3)M^{\frac{1}{2}}=\mathcal{L}\hat{f}
 +\nabla_{v}\cdot\hat{G}-\frac{1}{2}v\cdot\hat{G}+\hat{\varphi},\\
&\partial_{t}\hat{\rho}+i\xi\cdot \hat{u}=0,\\
&\partial_{t}\hat{u}+|\xi|^{2} \hat{u}+i\xi \hat{\theta}+i\xi \hat{\rho}+\hat{u}-\hat{b}=0,\\
&\partial_{t}\hat{\theta}+|\xi|^{2} \hat{\theta}+i\xi \hat{u}+\sqrt{3}(\sqrt{3}\hat{\theta}-\sqrt{2}\hat{\omega})=0.
\end{aligned}  \right.
\end{equation}
By taking the inner product of the  equations in \eqref{v3.5} with the conjugate of $\hat{f}$and integrating in $v$, its real part gives
\begin{align}
\frac{1}{2}\partial_{t}&\|\hat{f}\|^{2}_{L^{2}_{v}}+Re\int_{R^{3}}\langle-\mathcal{L}\{\mathbf{I-P}\}\hat{f}|\{\mathbf{I-P}\}\hat{f}\rangle\, {\rm d}v
+|\hat{b}|^{2}+2|\hat{\omega}|^{2}-Re\langle \hat{u}|\hat{b}\rangle-Re\langle \hat{\theta}|\sqrt{6}\hat{\omega}\rangle\nonumber\\
&=Re\int_{R^{3}}\langle\nabla_{v}\cdot G-\frac{1}{2}v\cdot \hat{G}|\hat{f}\rangle\, {\rm d}v
+Re\int_{R^{3}}\langle\hat{\varphi}|\hat{f}\rangle\, {\rm d}v\nonumber\\
&=Re\int_{R^{3}}\langle\nabla_{v}\cdot G-\frac{1}{2}v\cdot \hat{G}|\{\mathbf{I-P}\}\hat{f}\rangle\, {\rm d}v
+Re\int_{R^{3}}\langle\hat{\varphi}|\{\mathbf{I-P}\}\hat{f}\rangle\, {\rm d}v,\nonumber
\end{align}
here, on the basis of the assumptions  \eqref{vx3.1}, we can calculate
\begin{align}
\int_{R^{3}}\langle\nabla_{v}\cdot G-\frac{1}{2}v\cdot \hat{G}|\mathbf{P}\hat{f}\rangle\, {\rm d}v=0,\quad
\int_{R^{3}}\langle\hat{\varphi}|\{\mathbf{P}\}\hat{f}\rangle\, {\rm d}v=0.\nonumber
\end{align}
Similarly, from the last three equations in \eqref{v3.1} we have
\begin{align}
&\frac{1}{2}|\hat{\rho}|^{2}+Re\langle i\xi\hat{u}|\hat{\rho}\rangle=0,\nonumber\\
&\frac{1}{2}|\hat{u}|^{2}+Re\langle i\xi\hat{\theta}|\hat{u}\rangle+Re\langle i\xi\hat{\rho}|\hat{u}\rangle+|\hat{u}|^{2}+|\xi|^{2}|\hat{u}|^{2}
-Re\langle\hat{b}|\hat{u}\rangle=0,\nonumber\\
&\frac{1}{2}|\hat{\theta}|^{2}+Re\langle i\xi\hat{u}|\hat{\theta}\rangle+3|\hat{\theta}|^{2}+|\xi|^{2}|\hat{\theta}|^{2}-Re\langle\sqrt{6}\hat{\omega}|\hat{\theta}\rangle=0.\nonumber
\end{align}
Then, combining these estimates, using the coercivity of $-\mathcal{L}$ and the Cauchy-Schwarz inequality, we show
\begin{align}
\frac{1}{2}\partial_{t}&\Big(\|\hat{f}\|^{2}_{L^{2}_{v}}+|\hat{\rho}|^{2}+|\hat{u}|^{2}+|\hat{\theta}|^{2}\Big)+\lambda|\{\mathbf{I-P}\}\hat{f}|^{2}_{\nu}\nonumber\\
&\quad+|\hat{u}-\hat{b}|^{2}+|\sqrt{2}\hat{\omega}-\sqrt{3}\hat{\theta}|^{2}+|\xi|^{2}|\hat{u}|^{2}+|\xi|^{2}|\hat{\theta}|^{2}\nonumber\\
&\leq C\Big(\|\hat{G}\|^{2}+\|\nu^{-\frac{1}{2}}\hat{\varphi}\|^{2}\Big).\label{v3.6}
\end{align}

Next, we consider the estimates on $a,b,\omega$. Similar to   \eqref{v2.13}-\eqref{vx2.15}, for the system \eqref{v3.1},  one has
\begin{align}
&\partial_{t}a+\dv b=0, \nonumber \\
&\partial_{t}b_{i}+\partial_{i}a +\frac{2}{\sqrt{6}}\partial_{i}\omega+\sum_{j=1}^{3}\partial_{x_{j}}\Gamma_{i,j}(\{\mathbf{I}-\mathbf{P}\}f)
=u_{i}-b_{i},\nonumber  \\
&\partial_{t}\omega+\sqrt{2}(\sqrt{2}\omega-\sqrt{3}\theta)+\frac{2}{\sqrt{6}}\dv b
+\sum_{i=1}^{3}\partial_{x_{i}}Q_{i}(\{\mathbf{I}-\mathbf{P}\}f)
=0,\nonumber\\
&\partial_{j}b_{i}+\partial_{i}b_{j}
-\frac{2}{\sqrt{6}}\delta_{ij}
\Big(\frac{2}{\sqrt{6}}\dv b+\sum_{i=1}^{3}\partial_{x_{i}}Q_{i}(\{\mathbf{I}-\mathbf{P}\}f)
\Big)\nonumber\\
&\quad\,\,\,\,\,=-\partial_{t}\Gamma_{i,j}\{\mathbf{I}-\mathbf{P}\}f +\Gamma_{i,j}(l+S_{f}), \nonumber\\
&\frac{5}{3}\partial_{i}\omega-\frac{2}{\sqrt{6}}\sum_{j=1}^{3}\partial_{x_{j}}\Gamma_{i,j}(\{\mathbf{I}-\mathbf{P}\}f)
=-\partial_{t}Q_{i}\big(\{\mathbf{I}-\mathbf{P}\}f\big)+Q_{i}(l+S_{f}),\nonumber
\end{align}
where  $l$  is still expressed as
$$l=-v\cdot \nabla_{x}\{\mathbf{I}-\mathbf{P}\}f+\mathcal{L}\{\mathbf{I}-\mathbf{P}\}f.$$
In the same way, taking the Fourier transform in $x$, we obtain
\begin{align}
&\partial_{t}\hat{a}+i\xi\cdot \hat{b}=0, \nonumber \\
&\partial_{t}\hat{b}_{i}+i\xi_{i}\hat{a }+\frac{2}{\sqrt{6}}i\xi_{i}\hat{\omega}+\sum_{j=1}^{3}i\xi_{j}\Gamma_{i,j}(\{\mathbf{I}-\mathbf{P}\}\hat{f})
=\hat{u}_{i}-\hat{b}_{i},\nonumber  \\
&\partial_{t}\hat{\omega}+\sqrt{2}(\sqrt{2}\hat{\omega}-\sqrt{3}\hat{\theta})+\frac{2}{\sqrt{6}}i\xi\cdot \hat{b}
+\sum_{i=1}^{3}i\xi_{i} Q_{i}(\{\mathbf{I}-\mathbf{P}\}\hat{f})
=0,\nonumber\\
&i\xi_{j}\hat{b}_{i}+i\xi_{i}\hat{b}_{j}
-\frac{2}{\sqrt{6}}\delta_{ij}
\Big(\frac{2}{\sqrt{6}}i\xi\cdot \hat{b}+\sum_{i=1}^{3}i\xi_{i}Q_{i}(\{\mathbf{I}-\mathbf{P}\}\hat{f})
\Big)\nonumber\\
&\quad\,\,\,\,\,=-\partial_{t}\Gamma_{i,j}\{\mathbf{I}-\mathbf{P}\}\hat{f} +\Gamma_{i,j}(\hat{l}+\hat{S}_{f}), \nonumber\\
&\frac{5}{3}i\xi_{i}\hat{\omega}-\frac{2}{\sqrt{6}}\sum_{j=1}^{3}i\xi_{j}\Gamma_{i,j}(\{\mathbf{I}-\mathbf{P}\}\hat{f})
=-\partial_{t}Q_{i}\big(\{\mathbf{I}-\mathbf{P}\}\hat{f}\big)+Q_{i}(\hat{l}+\hat{S}_{f}).\nonumber
\end{align}
By adopting  the similar calculation method as in Proposition \ref{vl2.4}, we can conclude the following  inequalities:
\begin{align}
\partial_{t}&Re\sum_{i,j}\langle i\xi_{i}\hat{b}_{j}+i\xi_{j}\hat{b}_{i}|\Gamma_{i,j}(\{\mathbf{I}-\mathbf{P}\}\hat{f})\rangle
+\lambda \Big(|\xi|^{2}|\hat{b}|^{2}+|\xi\cdot\hat{b}|^{2}\Big)\nonumber\\
&\leq\varepsilon |\xi|^{2}\big(|\hat{a}|^{2}+|\hat{\omega}|^{2}\big)+C(1+|\xi|^{2})\|\{\mathbf{I}-\mathbf{P}\}\hat{f}\|^{2}_{L^{2}_{v}}
+C\big(|\hat{u}-\hat{b}|^{2}+\|\hat{G}\|^{2}+\|\nu^{-\frac{1}{2}}\hat{\varphi}\|^{2}\big),\nonumber\\
\partial_{t}&Re\sum_{i}\langle i\xi_{i}\hat{\omega}|Q_{i}(\{\mathbf{I}-\mathbf{P}\}\hat{f})\rangle
+|\xi|^{2}|\hat{\omega}|^{2}\nonumber\\
&\leq\varepsilon |\xi|^{2}|\hat{b}|^{2}+C(1+|\xi|^{2})\|\{\mathbf{I}-\mathbf{P}\}\hat{f}\|^{2}_{L^{2}_{v}}
+C\big(|\sqrt{2}\hat{\omega}-\sqrt{3}\hat{\theta}|^{2}+\|\hat{G}\|^{2}+\|\nu^{-\frac{1}{2}}\hat{\varphi}\|^{2}\big),\nonumber\\
\partial_{t}&Re\langle \hat{a}|i\frac{\sqrt{6}}{5}\sum_{j=1}^{3}\xi_{j}Q_{j}(\{\mathbf{I}-\mathbf{P}\}\hat{f})-i\xi\cdot \hat{b}\rangle
+\frac{3}{4} |\xi|^{2}|\hat{a}|^{2}\nonumber\\
&\leq\frac{5}{4}|\xi\cdot \hat{b}|^{2}+C(1+|\xi|^{2})\|\{\mathbf{I}-\mathbf{P}\}\hat{f}\|^{2}_{L^{2}_{v}}
+C\big(|\hat{u}-\hat{b}|^{2}+\|\hat{G}\|^{2}+\|\nu^{-\frac{1}{2}}\hat{\varphi}\|^{2}\big),\nonumber\\
\partial_{t}&Re\langle \hat{u}|i\xi\hat{\rho}\rangle+\frac{3}{4}|\xi|^{2}|\hat{\rho}|^{2}\leq C |\xi|^{2}\big(|\hat{\theta}|^{2}+|\hat{u}|^{2}\big)+C|\hat{u}-\hat{b}|^{2}.\nonumber
\end{align}
Choosing $\kappa_{1}$  small sufficiently, and setting $\mathcal{\tilde{E}}(\hat{f})$ as
\begin{align}
\mathcal{\tilde{E}}(\hat{f}):=&\frac{1}{1+|\xi|^{2}}\Big\{\sum_{i,j}\langle i\xi_{i}\hat{b}_{j}+i\xi_{j}\hat{b}_{i}|\Gamma_{i,j}(\{\mathbf{I}-\mathbf{P}\}\hat{f})\rangle\nonumber\\
&+\sum_{i}\langle i\xi_{i}\hat{\omega}|Q_{i}(\{\mathbf{I}-\mathbf{P}\}\hat{f})\rangle
+\kappa_{1}\langle \hat{a}|i\frac{\sqrt{6}}{5}\sum_{j=1}^{3}\xi_{j}Q_{j}(\{\mathbf{I}-\mathbf{P}\}\hat{f})-i\xi\cdot \hat{b}\rangle\Big\},\nonumber
\end{align}
one has
\begin{align}
\partial_{t}&Re \mathcal{\tilde{E}}(\hat{f})
+\frac{\lambda |\xi|^{2}}{1+|\xi|^{2}} \big(|\hat{a}|^{2}+|\hat{b}|^{2}+|\hat{\omega}|^{2}\big)\nonumber\\
&\leq C\big(\|\{\mathbf{I}-\mathbf{P}\}\hat{f}\|^{2}_{L^{2}_{v}}
+|\hat{u}-\hat{b}|^{2}+|\sqrt{2}\hat{\omega}-\sqrt{3}\hat{\theta}|^{2}
\big)+C\big(\|\hat{G}\|^{2}+\|\nu^{-\frac{1}{2}}\hat{\varphi}\|^{2}\big). \nonumber
\end{align}
Similarly, choosing $\kappa_{2}$  small sufficiently, and setting $\mathcal{\tilde{E}}(\hat{f},\hat{\rho},\hat{u},\hat{\theta})$ as
\begin{align}
\mathcal{\tilde{E}}(\hat{f},\hat{\rho},\hat{u},\hat{\theta}):=\mathcal{\tilde{E}}(\hat{f})+\kappa_{2}\frac{1}{1+|\xi|^{2}}\langle \hat{u}|i\xi\hat{\rho}\rangle,\nonumber
\end{align}
then we have
\begin{align}
\partial_{t}&Re \mathcal{\tilde{E}}(\hat{f},\hat{\rho},\hat{u},\hat{\theta})
+\frac{\lambda |\xi|^{2}}{1+|\xi|^{2}} \big(|\hat{a}|^{2}+|\hat{b}|^{2}+|\hat{\omega}|^{2}+|\hat{\rho}|^{2}+|\hat{u}|^{2}+|\hat{\theta}|^{2}\big)\nonumber\\
&\leq C\big(\|\{\mathbf{I}-\mathbf{P}\}\hat{f}\|^{2}_{L^{2}_{v}}
+|\hat{u}-\hat{b}|^{2}+|\sqrt{2}\hat{\omega}-\sqrt{3}\hat{\theta}|^{2}
\big)\nonumber\\
&\quad+C\big(\|\hat{G}\|^{2}+\|\nu^{-\frac{1}{2}}\hat{\varphi}\|^{2}\big).\label{v3.7}
\end{align}
Now, we define the functional $\mathcal{E}_{\mathcal{F}}(\hat{f},\hat{\rho},\hat{u},\hat{\theta})$ by
$$\mathcal{E}_{\mathcal{F}}(\hat{f},\hat{\rho},\hat{u},\hat{\theta}):=\|\hat{f}\|^{2}_{L^{2}_{v}}+|\hat{\rho}|^{2}+|\hat{u}|^{2}+|\hat{\theta}|^{2}
+\kappa_3 Re \mathcal{\tilde{E}}(\hat{f},\hat{\rho},\hat{u},\hat{\theta}), $$
where a small constant $\kappa_3 > 0$ is chosen  such that
\begin{align}
\mathcal{E}_{\mathcal{F}}(\hat{f},\hat{\rho},\hat{u},\hat{\theta})
&\thicksim  \|\hat{f}\|^{2}_{L^{2}_{v}}+|\hat{\rho}|^{2}+|\hat{u}|^{2}+|\hat{\theta}|^{2}\nonumber\\
&\thicksim \|\{\mathbf{I}-\mathbf{P}\}\hat{f}\|^{2}_{L^{2}_{v}}+|\hat{a}|^{2}+|\hat{b}|^{2}+|\hat{\omega}|^{2}+|\hat{\rho}|^{2}+|\hat{u}|^{2}+|\hat{\theta}|^{2}.\nonumber
\end{align}
Finally, the linear combination $\eqref{v3.6}+\kappa_3 \times \eqref{v3.7}$ gives
\begin{align}
\partial_{t}&\mathcal{E}_{\mathcal{F}}(\hat{f},\hat{\rho},\hat{u},\hat{\theta})
+\lambda\big(\|\{\mathbf{I}-\mathbf{P}\}\hat{f}\|^{2}_{L^{2}_{v}}
+|\hat{u}-\hat{b}|^{2}+|\sqrt{2}\hat{\omega}-\sqrt{3}\hat{\theta}|^{2}
\big)\nonumber\\
&\quad+\frac{\lambda |\xi|^{2}}{1+|\xi|^{2}} \big(|\hat{a}|^{2}+|\hat{b}|^{2}+|\hat{\omega}|^{2}+|\hat{\rho}|^{2}+|\hat{u}|^{2}+|\hat{\theta}|^{2}\big)\nonumber\\
&\leq
C\big(\|\hat{G}\|^{2}+\|\nu^{-\frac{1}{2}}\hat{\varphi}\|^{2}\big),\nonumber
\end{align}
which further implies
\begin{align}
\partial_{t}\mathcal{E}_{\mathcal{F}}(\hat{f},\hat{\rho},\hat{u},\hat{\theta})
+\frac{\lambda |\xi|^{2}}{1+|\xi|^{2}}\mathcal{E}_{\mathcal{F}}(\hat{f},\hat{\rho},\hat{u},\hat{\theta})
\leq C\big(\|\hat{G}\|^{2}+\|\nu^{-\frac{1}{2}}\hat{\varphi}\|^{2}\big).\nonumber
\end{align}
With  the help of Gronwall's inequality, we have
\begin{align}
\mathcal{E}_{\mathcal{F}}(\hat{f},\hat{\rho},\hat{u},\hat{\theta})
\leq e^{-\frac{\lambda |\xi|^{2}}{1+|\xi|^{2}}}\mathcal{E}_{\mathcal{F}}(\hat{f},\hat{\rho},\hat{u},\hat{\theta})
+\int_{0}^{t}e^{-\frac{\lambda |\xi|^{2}}{1+|\xi|^{2}}(t-\tau)}\big(\|\hat{G}\|^{2}+\|\nu^{-\frac{1}{2}}\hat{\varphi}\|^{2}\big)\, {\rm d}\tau. \nonumber
\end{align}
In order to obtain the desired time-decay estimates \eqref{v3.3} and \eqref{v3.4}, the same proof can be adopted  as in
\cite[Theorem 3.1]{DFT}, and  we omit the details.
\end{proof}

\begin{proof}[Proof of the large time behavior in Theorem \ref{vt1.1}]
By the definitions of $\mathcal{E}(t),\mathcal{D}(t)$ in the previous section, it follows that
\begin{align}
\mathcal{E}(t)&\leq C\Big(\|\{\mathbf{I}-\mathbf{P}\}f\|^{2}_{H^{4}}+\|(a,b,\omega)\|^{2}_{H^{4}}+\|(\rho,u,\theta)\|^{2}_{H^{4}}\Big)\nonumber\\
&\leq C\Big(\mathcal{D}(t)+\|f\|^{2}_{L^{2}}+\|(\rho,u,\theta)\|^{2}_{L^{2}}\Big).\label{v3.8}
\end{align}
From \eqref{v2.29}, one obtains
\begin{align}
\frac{\rm d}{{\rm d}t}\mathcal{E}(t) +\lambda \mathcal{D}(t)\leq 0,\nonumber
\end{align}
which, together with \eqref{v3.8}, yields
$$\frac{\rm d}{{\rm d}t}\mathcal{E}(t) +\lambda \mathcal{E}(t)\leq C\Big(\|f\|^{2}_{L^{2}}+\|(\rho,u,\theta)\|^{2}_{L^{2}}\Big).$$
According to Gronwall inequality, it follows that
\begin{align}
\mathcal{E}(t)\leq e^{-\lambda t}\mathcal{E}(0)+C\int_{0}^{t}e^{-\lambda (t-s)}\big(\|f(s)\|^{2}_{L^{2}}+\|(\rho(s),u(s),\theta(s))\|^{2}_{L^{2}}\big)\, {\rm d}s.\label{v3.9}
\end{align}
Next, the system  \eqref{v1.5}-\eqref{vx1.7} can be written as
$$\big(f(t),\rho(t),u(t),\theta(t)\big)=E_{t}(f_{0},\rho_{0},u_{0},\theta_{0})+\int_{0}^{t}E_{t-s}\big(S_{f}(s),S_{\rho}(s),S_{u}(s),S_{\theta}(s)\big)\, {\rm d}s,$$
with
\begin{align}
S_{\rho}=&-\dv(\rho u), \nonumber\\
S_{f}=&-u\cdot\nabla_{v}f+\frac{1}{2}u\cdot v f+\theta M^{-\frac{1}{2}}\dv_{v}\Big(M^{\frac{1}{2}}\big(\nabla_{v}f-\frac{v}{2}f\big)\Big),\nonumber\\
S_{u}=&-u\cdot\nabla u+\frac{\rho-\theta}{1+\rho}\nabla \rho+\frac{\rho}{1+\rho}(u-b)-\frac{1}{1+\rho}a u-\frac{\rho}{1+\rho}\Delta u ,\nonumber\\
S_{\theta}=&-u\cdot \nabla \theta -\theta\dv u+\frac{\sqrt{3}\rho}{1+\rho}\big(\sqrt{3}\theta-\sqrt{2}\omega\big)\nonumber\\
&+\frac{1}{1+\rho}\Big((1+a)|u|^{2}-3a \theta -2u\cdot b- \rho\Delta \theta\Big),   \nonumber
\end{align}
where $S_{f}$ can be decomposed as
\begin{align}
S_{f}
:=\nabla_{v}\cdot G-\frac{v}{2}v\cdot G + \varphi +au\cdot v M^{\frac{1}{2}}-u\cdot b M^{\frac{1}{2}} +a\theta (|v|^{2}-3)M^{\frac{1}{2}} \nonumber
\end{align}
with
\begin{align}
G:=-u\{\mathbf{I-P_{0}-P_{1}}\}f,\quad \varphi:=\theta\cdot \{I-P_{2}\}M^{-\frac{1}{2}}\dv_{v}\Big(M^{\frac{1}{2}}\big(\nabla_{v}f-\frac{v}{2}f\big)\Big).\nonumber
\end{align}
Therefore, $\big(f(t),\rho(t),u(t),\theta(t)\big)$ can be rewritten as the sum of six terms
\begin{align}
\big(f(t),\rho(t)&,u(t),\theta(t)\big)\nonumber\\
=&E_{t}(f_{0},\rho_{0},u_{0},\theta_{0})
+\int_{0}^{t}E_{t-s}\big(\nabla_{v}\cdot G-\frac{v}{2}v\cdot G + \varphi,0,0,0\big)\, {\rm d}s\nonumber\\
&+\int_{0}^{t}E_{t-s}\big(au\cdot v M^{\frac{1}{2}}-u\cdot b M^{\frac{1}{2}} +a\theta (|v|^{2}-3)M^{\frac{1}{2}},0,0,0\big)\, {\rm d}s\nonumber\\
&+\int_{0}^{t}E_{t-s}\big(0,S_{\rho}(s),0,0\big)\, {\rm d}s
+\int_{0}^{t}E_{t-s}\big(0,0,S_{u}(s),0\big)\, {\rm d}s\nonumber\\
&+\int_{0}^{t}E_{t-s}\big(0,0,0,S_{\theta}(s)\big)\, {\rm d}s\nonumber\\
=&U_{1}+U_{2}+U_{3}+U_{4}+U_{5}+U_{6}.\nonumber
\end{align}

Applying directly \eqref{v3.3} to $U_{1}$, one has
$$\|U_{1}(t)\|_{L^{2}}\leq C(1+t)^{-\frac{3}{4}}\|(f_{0},\rho_{0},u_{0},\theta_{0})\|_{\mathcal{Z}^{1}\cap L^{2}}.$$

With the help of H\"{o}lder and Sobolev inequalities, using \eqref{v3.4} to $U_{2}$, we deduce
\begin{align}
\|U_{2}(t)\|^{2}_{L^{2}}&\leq C\int_{0}^{t}(1+t-s)^{-\frac{3}{2}}
\Big(\|u\cdot\{\mathbf{I-P_{0}-P_{1}}\}f\|^{2}_{\mathcal{Z}_{1}\cap L^{2}}\nonumber\\
&\quad+\|\nu^{-\frac{1}{2}}\theta\cdot \{I-P_{2}\}M^{-\frac{1}{2}}\dv_{v}\big(M^{\frac{1}{2}}(\nabla_{v}f-\frac{v}{2}f)\big)\|\Big)\, {\rm d}s\nonumber\\
&\leq C\int_{0}^{t}(1+t-s)^{-\frac{3}{2}}\mathcal{E}^{2}(s)\, {\rm d}s+ C\int_{0}^{t}(1+t-s)^{-\frac{3}{2}}\|\theta\|^{2}_{H^{4}}\|\{I-P\}f\|^{2}_{\nu}\, {\rm d}s\nonumber\\
&\leq C\int_{0}^{t}(1+t-s)^{-\frac{3}{2}}\mathcal{E}^{2}(s)\, {\rm d}s+ C\int_{0}^{t}(1+t-s)^{-\frac{3}{2}}\mathcal{E}(s)\mathcal{D}(s)\, {\rm d}s. \nonumber
\end{align}
Similarly, for $U_{i}\,(3\leq i\leq6)$, by means of  H\"{o}lder and Sobolev inequalities, we can apply \eqref{v3.3} to them to compute
\begin{align}
\|U_{3}(t)\|_{L^{2}}&\leq C\int_{0}^{t}(1+t-s)^{-\frac{3}{4}}\|\big(au\cdot v M^{\frac{1}{2}},u\cdot b M^{\frac{1}{2}},a\theta (|v|^{2}-3)M^{\frac{1}{2}}\big)\|_{\mathcal{Z}_{1}\cap L^{2}}\, {\rm d}s \nonumber\\
&\leq C\int_{0}^{t}(1+t-s)^{-\frac{3}{4}}\mathcal{E}(s)\, {\rm d}s,\nonumber\\
\|U_{4}(t)\|_{L^{2}}&\leq C\int_{0}^{t}(1+t-s)^{-\frac{3}{4}}\|(\dv u \rho, u\cdot\nabla \rho )\|_{L^{1}\cap L^{2}}\, {\rm d}s\nonumber\\
&\leq C\int_{0}^{t}(1+t-s)^{-\frac{3}{4}}\mathcal{E}(s)\, {\rm d}s,\nonumber\\
|U_{5}(t)\|_{L^{2}}&\leq C\int_{0}^{t}(1+t-s)^{-\frac{3}{4}}\|S_{u}(s)\|_{L^{1}\cap L^{2}}\, {\rm d}s\nonumber\\
&\leq C\int_{0}^{t}(1+t-s)^{-\frac{3}{4}}\mathcal{E}(s)\, {\rm d}s,\nonumber\\
|U_{6}(t)\|_{L^{2}}&\leq C\int_{0}^{t}(1+t-s)^{-\frac{3}{4}}\|S_{\theta}(s)\|_{L^{1}\cap L^{2}}\, {\rm d}s\nonumber\\
&\leq C\int_{0}^{t}(1+t-s)^{-\frac{3}{4}}\mathcal{E}(s)\, {\rm d}s.\nonumber
\end{align}

Therefore, it follows that
\begin{align}
&\|\big(f(t),\rho(t),u(t),\theta(t)\big)\|^{2}_{L^{2}}\leq 2\sum_{i=1}^{6}\|U_{i}\|^{2}_{L^{2}}\nonumber\\
&\,\,\leq C(1+t)^{-\frac{3}{2}}\|(f_{0},\rho_{0},u_{0},\theta_{0})\|^{2}_{\mathcal{Z}^{1}\cap L^{2}}
+C\int_{0}^{t}(1+t-s)^{-\frac{3}{2}}\mathcal{E}^{2}(s)\, {\rm d}s\nonumber\\
&\,\,\quad+ C\int_{0}^{t}(1+t-s)^{-\frac{3}{2}}\mathcal{E}(s)\mathcal{D}(s)\, {\rm d}s
 +C\Big(\int_{0}^{t}(1+t-s)^{-\frac{3}{4}}\mathcal{E}(s)\, {\rm d}s\Big)^{2}.\label{v3.10}
\end{align}

Define
\begin{align}
\mathcal{E}_{\infty}(t):=\sup_{0\leq s\leq t}(1+s)^{\frac{3}{2}}\mathcal{E}(s).\label{v3.11}
\end{align}
Using \eqref{v3.10} and that $\mathcal{E}(t),\mathcal{E}_{\infty}(t)$  are non-increasing in time, with the aid of Lemma \ref{vl3.2},
we obtain
\begin{align}
\int_{0}^{t}&(1+t-s)^{-\frac{3}{4}}\mathcal{E}(s)\, {\rm d}s \nonumber\\
&= \int_{0}^{t}(1+t-s)^{-\frac{3}{4}}
\Big(\mathcal{E}(s)\Big)^{\frac{2}{3}+\gamma}\Big(\mathcal{E}(s)\Big)^{\frac{1}{3}-\gamma}\, {\rm d}s\nonumber\\
&\leq C\Big(\mathcal{E}_{\infty}(t)\Big)^{\frac{2}{3}+\gamma}\Big(\mathcal{E}(0)\Big)^{\frac{1}{3}-\gamma}
\int_{0}^{t}(1+t-s)^{-\frac{3}{4}}(1+s)^{-\frac{3}{2}(\frac{2}{3}+\gamma)}\, {\rm d}s\nonumber\\
&\leq C(1+t)^{-\frac{3}{4}}\Big(\mathcal{E}_{\infty}(t)\Big)^{\frac{2}{3}+\gamma}\Big(\mathcal{E}(0)\Big)^{\frac{1}{3}-\gamma},\nonumber\\
\int_{0}^{t}&(1+t-s)^{-\frac{3}{2}}\big(\mathcal{E}(s)\big)^{2}\, {\rm d}s \nonumber\\
&\leq
\mathcal{E}_{\infty}(t)\mathcal{E}(0)\int_{0}^{t}(1+t-s)^{-\frac{3}{2}}(1+s)^{-\frac{3}{2}}\, {\rm d}s\nonumber\\
&\leq C(1+t)^{-\frac{3}{2}}\mathcal{E}_{\infty}(t)\mathcal{E}(0),\nonumber\\
\int_{0}^{t}&(1+t-s)^{-\frac{3}{2}}\mathcal{E}(s)\mathcal{D}(s)\, {\rm d}s \nonumber\\
&\leq \mathcal{E}_{\infty}(t)\int_{0}^{t}(1+t-s)^{-\frac{3}{2}}(1+s)^{-\frac{3}{2}}\mathcal{D}(s)\, {\rm d}s\nonumber\\
&\leq \mathcal{E}_{\infty}(t)(1+t)^{-\frac{3}{2}}\int_{0}^{t}\mathcal{D}(s)\, {\rm d}s\nonumber\\
&\leq C(1+t)^{-\frac{3}{2}}\mathcal{E}_{\infty}(t)\mathcal{E}(0),\nonumber
\end{align}
with $0< \gamma < \frac{1}{3}$.
From this, $\|(f(t),\rho(t),u(t),\theta(t))\|_{L^{2}}$ can be further estimated as
\begin{align}
\|\big(f(t),\rho(t),u(t),\theta(t)\big)\|^{2}_{L^{2}}
&\leq C(1+t)^{-\frac{3}{2}} \Big\{\|(f_{0},\rho_{0},u_{0},\theta_{0})\|_{\mathcal{Z}^{1}\cap L^{2}}\nonumber\\
&\quad+\mathcal{E}_{\infty}(t)\mathcal{E}(0)
+\Big(\mathcal{E}_{\infty}(t)\Big)^{\frac{4}{3}+2\gamma}\Big(\mathcal{E}(0)\Big)^{\frac{2}{3}-2\gamma}\Big\}.\nonumber
\end{align}
Substituting this inequality into the right hand side of \eqref{v3.9}, and multiplying the resulting inequality by $(1+t)^{\frac{3}{2}}$,   we obtain
\begin{align}
(1+t)^{\frac{3}{2}}\mathcal{E}(t)&\leq e^{-\lambda t}(1+t)^{\frac{3}{2}}\mathcal{E}(0)+C\int_{0}^{t}e^{-\lambda (t-s)}\Big(\|(f_{0},\rho_{0},u_{0},\theta_{0})\|_{\mathcal{Z}^{1}\cap L^{2}}\nonumber\\
&\qquad +\mathcal{E}_{\infty}(t)\mathcal{E}(0)
+\big(\mathcal{E}_{\infty}(t)\big)^{\frac{4}{3}+2\gamma}\big(\mathcal{E}(0)\big)^{\frac{2}{3}-2\gamma}\Big)\, {\rm d}s\nonumber\\
&\leq C\Big(\|(f_{0},\rho_{0},u_{0},\theta_{0})\|^{2}_{\mathcal{Z}^{1}\cap H^{4}} 
  +\mathcal{E}_{\infty}(t)\mathcal{E}(0)
+\Big(\mathcal{E}_{\infty}(t)\Big)^{\frac{4}{3}+2\gamma}\Big(\mathcal{E}(0)\Big)^{\frac{2}{3}-2\gamma}\Big).\nonumber
\end{align}
Thus, by the definition \eqref{v3.11}, we further obtain
\begin{align}
\mathcal{E}_{\infty}(t)\leq C\Big(\|(f_{0},\rho_{0},u_{0},\theta_{0})\|^{2}_{\mathcal{Z}^{1}\cap H^{4}}
+\mathcal{E}_{\infty}(t)\mathcal{E}(0)
+\big(\mathcal{E}_{\infty}(t)\big)^{\frac{4}{3}+2\gamma}\big(\mathcal{E}(0)\big)^{\frac{2}{3}-2\gamma}\Big).\nonumber
\end{align}
Since $\|(f_{0},\rho_{0},u_{0},\theta_{0})\|_{\mathcal{Z}^{1}\cap H^{4}}$ and $\mathcal{E}(0)\thicksim \|(f_{0},\rho_{0},u_{0},\theta_{0})\|^{2}_{H^{4}}$
are small enough, and $1<2(\frac{2}{3}+\gamma) <2$, one has
\begin{align}
y(t)\leq A(1+y(t))+C^{1-2(\frac{1}{2}-\gamma)}A^{2(\frac{1}{3}-\gamma)}y^{2}(t)\nonumber
\end{align}
for all $t\geq 0,$ with $y(t)=\mathcal{E}_{\infty}(t), A=C\|(f_{0},\rho_{0},u_{0},\theta_{0})\|^{2}_{\mathcal{Z}^{1}\cap H^{4}}$,
due to Lemma \ref{vl3.3}, which implies
$$\mathcal{E}_{\infty}(t)\leq 2A=2C\|(f_{0},\rho_{0},u_{0},\theta_{0})\|^{2}_{\mathcal{Z}^{1}\cap H^{4}},$$
namely,
$$\mathcal{E}(t)\leq C\|(f_{0},\rho_{0},u_{0},\theta_{0})\|^{2}_{\mathcal{Z}^{1}\cap H^{4}}(1+t)^{-\frac{3}{2}}.$$
The proof of Theorem \ref{vt1.1} is completed.

\end{proof}

\bigskip

\section{The periodic case}
In this section we deal with the spatial periodic domain $\Omega:=\mathbb{T}^{3}$.
A straightforward calculation can  deduce the conservation laws  as follows
\begin{align}
&\frac{\rm d}{{\rm d}t}\iint  F \, {\rm d}x  {\rm d}v=0, \qquad \frac{\rm d}{{\rm d}t}\bigg\{\int   nu\, {\rm d}x  +\iint  v F\, {\rm d}x  {\rm d}v\bigg\}=0,\nonumber\\
&\frac{\rm d}{{\rm d}t}\int  n \, {\rm d}x =0,\qquad \qquad\frac{\rm d}{{\rm d}t}\bigg\{\int   nE\, {\rm d}x  +\iint  \frac{|v|^{2}}{2} F\, {\rm d}x  {\rm d}v\bigg\}=0,
\nonumber
\end{align}
and by the assumption of Theorem \ref{vt1.2}, we obtain
\begin{align}
 &\int a\, {\rm d}x =0, \quad  \int \rho \, {\rm d}x =0, \quad \int \big(b+(1+\rho)u\big)\, {\rm d}x =0, \nonumber\\
  &\int (1+\rho)(\theta+\frac{1}{2}|u|^{2})+\frac{\sqrt{6}}{2}\omega \, {\rm d}x =0,\label{v4.1}
 \end{align}
for all $t\geq 0.$

Now, we give a brief proof of Theorem \ref{vt1.2}.

\begin{proof}[Proof of Theorem \ref{vt1.2}]
Here, we only give the proof of  the global a priori estimates.
It follows from  Poincar\'{e} inequality and the conservation laws \eqref{v4.1},
\begin{align}
\|a\|_{L^{2}}&\leq C\|\nabla a\|_{L^{2}}, \quad \|\rho\|_{L^{2}}\leq C\|\nabla \rho\|_{L^{2}}, \label{v4.2}\\
\|u+b\|_{L^{2}}&\leq \|b+u+\rho u\|_{L^{2}}+\|\rho u\|_{L^{2}}\nonumber\\
&\leq C \|\nabla (b+u+\rho u)\|_{L^{2}}+ \|u\|_{L^{\infty}}\|\rho\|_{L^{2}}\nonumber\\
&\leq C \|\nabla (b,u)\|_{L^{2}}+ C\|u\|_{H^{2}}\|\nabla \rho\|_{L^{2}}+C\|\rho\|_{H^{2}}\|\nabla u\|_{L^{2}},\label{v4.3}\\
\|\sqrt{6}/2\omega+\theta\|_{L^{2}}&\leq \Big\|(1+\rho)(\theta+\frac{1}{2}|u|^{2})+\frac{\sqrt{6}}{2}\omega\Big\|_{L^{2}}+\|\rho \theta\|_{L^{2}}+\|\rho |u|^{2}\|_{L^{2}}\nonumber\\
&\leq C(\|\theta\|_{H^{2}}+\|u\|^{2}_{H^{2}})\|(\rho,\nabla \rho)\|_{L^{2}}+C\|\nabla \omega\|_{L^{2}}\nonumber\\
&\quad+C(1+\|\rho\|_{H^{2}})(\|\nabla \theta\|_{L^{2}}+\|u\|_{H^{2}}\|\nabla u\|_{H^{2}}),\nonumber\\
\|\omega\|_{L^{2}}+\|\theta\|_{L^{2}}&\leq C\big(\|\sqrt{6}/2\omega+\theta\|_{L^{2}}+\|\sqrt{2}\omega-\sqrt{3}\theta\|_{L^{2}}\big).\label{vx4.3}
\end{align}
Let the energy functionals $\mathcal{E}_{1}(t)\,\text{and}\,\mathcal{E}_{2}(t)$ and the  corresponding dissipation rate functional
$\mathcal{D}_{1}(t)\,\text{and}\,\mathcal{D}_{2}(t)$ be defined in the same way as in the case  $\Omega:=\mathbb{R}^{3}$.
Similarly, we  conclude that
\begin{align}
&\frac{\rm d}{{\rm d}t}\mathcal{E}_{1}(t)+\lambda\mathcal{D}_{1}(t) \leq C(\mathcal{E}^{\frac{1}{2}}_{1}+\mathcal{E}^{2}_{1})\mathcal{D}_{1}(t),\label{v4.4}\\
&\frac{\rm d}{{\rm d}t}\mathcal{E}_{2}(t)+\lambda\mathcal{D}_{2}(t) \leq  C\mathcal{D}_{1}(t)
+C(\mathcal{E}_{1}+\mathcal{E}^{2}_{1})\mathcal{D}_{1}(t)
+C(\mathcal{E}^{\frac{1}{2}}_{1}+\mathcal{E}^{2}_{1})\mathcal{D}_{2}(t).\label{v4.5}
\end{align}
Define
$$\mathcal{D}_{\mathbb{T},1}(t) :=\mathcal{D}_{1}(t)+\tau_{3}(\|a\|^{2}_{L^{2}}+\|\rho\|^{2}_{L^{2}})+\tau_{4}\|b+u\|^{2}_{L^{2}}+\tau_{5}\|\sqrt{6}/2\omega+\theta\|_{L^{2}},$$
where $\tau_{3},\,\tau_{4}\,\text{and} \tau_{5}$ are sufficiently small. Notice
\begin{align}
\mathcal{D}_{\mathbb{T},1}(t)\sim \sum_{|\alpha|\leq 4}\|\{\mathbf{I}-\mathbf{P}\}\partial^{\alpha}f\|^{2}_{\nu}+\|(a,b,\omega,\rho,u,\theta)\|^{2}_{H^{4}}.\label{v4.6}
\end{align}
Combining \eqref{v4.2},\,\eqref{v4.3} and \eqref{v4.4} together, we conclude that
\begin{align}
\frac{\rm d}{{\rm d}t}\mathcal{E}_{1}(t)+\lambda\mathcal{D}_{\mathbb{T},1}(t) \leq C(\mathcal{E}^{\frac{1}{2}}_{1}+\mathcal{E}^{2}_{1})\mathcal{D}_{\mathbb{T},1}(t).\label{v4.7}
\end{align}
Define the functionals $\mathcal{E}(t)\,\text{and}\,\mathcal{D}(t)$ as
$$\mathcal{E}(t):=\mathcal{E}_{1}(t)+\tau_{6}\mathcal{E}_{2}(t),\quad \mathcal{D}(t):=\mathcal{D}_{\mathbb{T},1}(t)+\tau_{6}\mathcal{D}_{2}(t),$$
where $\tau_{6}$ is sufficiently small.
 \eqref{v4.5} together with \eqref{v4.7} yields that
$$\frac{\rm d}{{\rm d}t}\mathcal{E}(t) +\lambda \mathcal{D}(t) \leq C(\mathcal{E}^{\frac{1}{2}}(t)+\mathcal{E}^{2}(t))\mathcal{D}(t).$$
Based on  $\mathcal{E}(t)$  small enough and uniformly in time, and $\mathcal{E}(t)\leq C \mathcal{D}(t)$,
 it follows that
$$\frac{\rm d}{{\rm d}t}\mathcal{E}(t) +\lambda \mathcal{E}(t) \leq 0$$
for all $t\geq 0.$
By Gronwall's inequality, it is easy to obtain the exponential decay, and  we finish the proof of Theorem \ref{vt1.2}.
\end{proof}

\bigskip

\section*{Acknowledgments}
Y. Mu was partially supported  by NSFC
(Grant No.11701268), Natural Science Foundation of Jiangsu Province of China
(BK20171040) and Chinese Postdoctoral Science Foundation (2018M642277).
D. Wang's research was supported in part by  the NSF grants DMS-1312800 and DMS-1613213.

%

 \bigskip

\end{document}